\documentclass{amsart}
\usepackage[utf8]{inputenc}
\usepackage{amsmath}
\usepackage{amsthm}
\usepackage{amssymb}
\usepackage{graphicx}
\usepackage{amsrefs}
\usepackage{verbatimbox} 
\usepackage[all]{xy}

\newtheorem{thm}{Theorem}[section]
\newtheorem{prop}[thm]{Proposition}
\newtheorem{lem}[thm]{Lemma}

\theoremstyle{definition}
\newtheorem{dfn}{Definition}[thm]

\theoremstyle{remark}
\newtheorem{rem}{Remark}[thm]

\newcommand{\C}{\mathbb{C}}

\newcommand{\Proj}{\mathbb{P}}
\newcommand{\X}{\mathcal{X}}

\begin{document}

\title{The Gromov-Witten Theory of Borcea-Voisin Orbifolds and its Analytic Continuations}
\date{December 2014}
\author{Andrew Schaug}
\address{Department of Mathematics, University of Michigan, Ann Arbor, MI, 48109}
\email{trygve@umich.edu}

\maketitle

\begin{abstract}
In the early 1990s, Borcea-Voisin orbifolds were some of the earliest examples of Calabi-Yau threefolds shown to exhibit mirror symmetry. However, their quantum theory has been poorly investigated. We study this in the context of the gauged linear sigma model, which in their case encompasses Gromov-Witten theory and its three companions (FJRW theory and two mixed theories). For certain Borcea-Voisin orbifolds of Fermat type, we calculate all four genus zero theories explicitly. Furthermore, we relate the I-functions of these theories by analytic continuation and symplectic transformation. In particular, the relation between the Gromov-Witten and FJRW theories can be viewed as an example of the Landau-Ginzburg/Calabi-Yau correspondence for complete intersections of toric varieties.
\end{abstract}

\section{Background}

Mirror symmetry has been a driving force in geometry and physics for more than twenty years.
The first mirror symmetric phenomenon to be discovered mathematically was purely at the level of Hodge numbers. That is, it was noticed that many Calabi-Yau manifolds pair up in such a way that the Hodge diamonds of one is the Hodge diamond of the other rotated by a right angle. One of the earliest sets of examples of this broad phenomenon  was discovered by Borcea \cite{B} and Voisin \cite{V}, now known as Borcea-Voisin manifolds; distinguishing them by their K3 surfaces, these form a class of 92 members closed under this type of cohomological mirror symmetry. These are given as resolutions of certain quotients of products of elliptic curves and certain admissible K3 surfaces $\widetilde{(E\times K)}/\mathbb{Z}_2$. The deeper interest in mirror symmetry on the quantum level originates in the fact that in many cases, the physical observables in string theories defined on mirror background Calabi-Yau manifolds  are identical. The first example of this to be demonstrated mathematically was the use of mirror symmetry to predict the number of rational curves on certain Calabi-Yau threefolds. In particular, Givental \cite{Giv} and Lian, Liu and Yau \cite{LLY} related the Gromov-Witten invariants from enumerative geometry of one manifold (encapsulated in a `J-function') to the periods of the Picard-Fuchs equation of its mirror (encapsulated in an `I-function'). 
The methods used in Givental's formalism have been applied to establish this deeper form of mirror symmetry for several pairs of families of manifolds \cite{Giv}, but not yet the Borcea-Voisin manifolds whose example had in fact preceded it.  We hope to address this question in this and subsequent articles.

One method of attack is another physical duality altogether called Landau-Ginzburg (LG)/Calabi-Yau (CY) correspondence.  Different aspects of this physical correspondence have been formalised in a few ways, but the one of most importance here is that produced by \cite{FJR} between Gromov-Witten theory on the Calabi-Yau side, and Fan-Jarvis-Ruan-Witten (FJRW) theory on the Landau-Ginzburg side, both defined for a hypersurface of weighted projective space. Such a relationship has been established already for a number of examples, including the quintic \cite{ChiRu1} mirror quintic \cite{PS}, general Calabi-Yau hypersurfaces \cite{CIR}, and in a more general form for the classic Calabi-Yau three-fold complete intersections \cite{Cl}, and hypersurfaces of Fano and general type \cite{Ac}. There is a corresponding mirror symmetry for FJRW theory \cite{K}, known as Bergland-H\"ubsch-Krawitz (BHK) mirror symmetry, forming a square of dualities. It appears that FJRW theory is often somewhat easier to compute than the Calabi-Yau theory, and that a promising method of attack to prove Calabi-Yau mirror symmetry may be via BHK mirror symmetry and the Landau-Ginzburg duality. 

In the case of hypersurfaces of weighted projective spaces, there is a direct duality between the Calabi-Yau `phase' and the Landau-Ginzburg phase. Borcea-Voisin orbifolds are not hypersurfaces of weighted projective spaces, however, but can be given as complete intersections in quotients of products of weighted projective spaces. For more general complete intersections of toric varieties, Witten proposed an important set of physical models called \textit{Gauged Linear Sigma Models} (GLSM). A general GLSM has a far more complex and interesting phase structure, divided into several chambers, where wall-crossing can be viewed as a generalisation of the LG/CY correspondence. This has been put on a mathematical footing by Fan, Jarvis and Ruan \cite{FJRnew}. In the case of Borcea-Voisin orbifolds, this produces four different \textit{curve-counting theories}, the original Gromov-Witten and FJRW theories being two of them. This should give a potential approach to finding the quantum mirror structure of Borcea-Voisin orbifolds, via BHK mirror symmetry. This will be left to subsequent articles. Here, we focus on computations in the A-model.

Gromov-Witten theory and its companion theories all come with a state space $\mathcal{H}^{\circ} = \bigoplus_h \phi_h \mathbb{C}$, where $\circ$ stands for the GW, FJRW or mixed theories. It is endowed with an inner product analogous to the Poincar\'e pairing, and a multiplication with identity $\phi_0$. Furthermore they are each assigned a moduli space $\mathcal{M}^{\circ}$ of marked curves endowed with extra structure satisfying certain stability conditions. In Gromov-Witten theory, the curves are endowed with stable maps to $[\mathcal{X}/G],$ where $\mathcal{X} = \{W = 0\}$. In FJRW theory, they are endowed with line bundles satisfying conditions depending on the polynomial $W$ and the group $G$. In the case of interest in this paper, $\mathcal{X}$ is the complete intersection defined by polynomials $W_1, W_2$. The mixed theories come from considering the Gromov-Witten theory of one of the $W_i$ and the FJRW theory of the other, subject to compatibility conditions; the moduli spaces classify marked curves endowed with a stable map to $[\{W_i = 0\}/G]$ and line bundles subject to conditions depending on $W_j$, $j \ne i$, and that stable map. All of these moduli space come equipped with virtual classes.

In all these theories we integrate over the virtual class of the moduli space to define certain intersection numbers $$\langle \tau_{a_1}(\phi_{h_1}), \ldots, \tau_{a_n}(\phi_{h_n})\rangle_{0, n[, \beta]}^{\circ},$$ for $\tau_{a_i}(\phi_{h_i}) = \psi_i^{a_i}\phi_{h_i},$ where the $\psi$-classes are defined in the usual way as the first Chern classes of the Hodge bundle over the moduli space, itself defined fibre-wise as the space of holomorphic differentials over the base curve. The data for Gromov-Witten theory and the mixed theories include stable maps from the source curve, and it is helpful for our purposes to specify the homology class $\beta$ of the image of this stable map; this is not included for the purely FJRW invariants.

We may encapsulate the enumerative information of each theory by defining corresponding \textit{J-functions} over $\mathcal{H}^{\circ}((z^{-1}))$ by setting $J^{\circ}(\sum_ht_0^h \phi_h, z)$ to be
$$z\phi_0 + \sum_h t_0^h \phi_h + \sum_{\substack{n \ge 0 \\(h_1, \ldots, h_n)\\ [\beta \ge 0]}} \sum_{\varepsilon, k} \frac{t_0^{h_1}\ldots t_0^{h_n}}{n!z^{k+1}} \langle \tau_0(\phi_{h_1}), \ldots \tau_0(\phi_{h_n}), \tau_k(\phi_{\varepsilon})\rangle_{0, n+1[, \beta]}^{\circ}.$$ More precisely, we shall consider the \textit{ambient} or \textit{narrow} J-functions, which restrict to the ambient or narrow classes, classes induced from the ambient product of weighted projective spaces, and which have far more manageable enumerative geometry.

The original quantum mirror theorems relate the J-function of an orbifold, from the curve-counting A-side, to its \textit{I-function}, a fundamental solution of the orbifold's corresponding Picard-Fuchs equations, by means of a \textit{mirror map} $\tau(\mathbf{t})$. For $I(\mathbf{t}, z) = f(\mathbf{t})z + \mathbf{g}(\mathbf{t}) + \mathcal{O}(z^{-1}),$ the relation is given by $$J(\tau(\mathbf{t}, z)) = \frac{I(\mathbf{t}, z)}{f(\mathbf{t})},$$ where $\tau(\mathbf{t}) = \frac{\mathbf{g}(\mathbf{t})}{f(\mathbf{t})}.$ The mirror orbifold then swaps the roles of the I- and J-functions. 

In this paper, we will find the J-functions of each theory, but it will be simpler to demonstrate the correspondence in terms of I-functions. However, in this paper we do not address the Picard-Fuchs equations, so our nomenclature for the I-functions originates solely by analogy from its role in the above mirror theorem; our computations are all in terms of the A side.

For certain coefficients $K_{\mathbf{b}}, L_{\mathbf{b}}$ and functions $F_{\mathbf{b}}, G_{\mathbf{b}},$ we find the (narrow, genus zero) Gromov-Witten I-function to be
\begin{align*}
& I_{\mathrm{GW}} (\mathcal{Y}) = ze^{(2D_Et_1 + D_E(t_2+t_3) + \sum_{i=1}^4 w_iD_Kt_{4+i})/z} \times  \nonumber \\
& \sum_{\mathbf{b} \in \mathrm{Box}(\mathcal{Y})} K_{\mathbf{b}} L_{\mathbf{b}} \sum_{\substack{c \in \frac{1}{2}\mathbb{N}_0}}\sum_{\substack{(a, b, k_1, \ldots, k_m) \in \mathbb{N}_0^m\\ a \ge -c/2, b \ge -c/w_0\\v^{S}(a, b, c, k_1, \ldots, k_m) = \mathbf{b}}} (q_1^a q_2^b q_3^c \prod_{j=1}^mx_j^{k_j}) \\ 
& e^{a(2t_1 + t_2 + t_3)+b(\sum_{i=0}^3 w_it_i)+c(t_1 + t_4 + 2t_8)} F_\mathbf{b}(a, c, \mathbf{k}) G_{\mathbf{b}}(b, c, \mathbf{k}) \mathbf{1}_\mathbf{b}.
\end{align*}

For example, for $E = \{X^2 + Y^4 + Z^4 = 0\}$ in $\mathbb{P}(2, 1, 1)$, $K = \{x^2 + y^6 + z^ 6 + w^ 6 = 0\} \subseteq \mathbb{P}(3, 1, 1, 1)$, and $\sigma: (X, x) \mapsto (-X, -x)$, we find:
\begin{align*} 
& I_{\mathrm{GW}}^{[E \times K/\langle \sigma \rangle]}(\mathbf{t}, z) = ze^{(2D_Et_1 + D_E(t_2+t_3) + 3D_Kt_4 + D_K(t_5+t_6+t_7))/z} \times  \nonumber \\
& \bigg(\sum_{c \in \mathbb{N}_0} \sum_{\substack{a, b \in \mathbb{Z}\\ a\ge -c/2\\b \ge -c/3}} q_1^a q_2^b q_3^c e^{(2a+c)t_1 + a(t_2+t_3) + (3b+c)t_4 + b(t_5+t_6+t_7) + 2ct_8} \times \nonumber \\
& \frac{\Gamma(2D_E/z+1) \Gamma(D_E/z+1)^2 \Gamma(3D_K/z+1)\Gamma(D_K/z+1)^3}{\Gamma(2D_E/z+2a+c+1)\Gamma(D_E/z+a+1)^2\Gamma(3D_K/z+3b+c+1)\Gamma(D_K/z+b+1)^3}\times \nonumber \\
& \frac{\Gamma(4D_E/z+4a+2c+1)\Gamma(6D_K/z + 6b+2c + 1)}{\Gamma(2c+1)\Gamma(4D_E/z+1)\Gamma(6D_K/z+1)} \nonumber \\
+& z^{-1}\sum_{c \in \frac{1}{2}\mathbb{N}_0\backslash \mathbb{N}_0} \sum_{\substack{a, b \in \mathbb{Z}\\ a\ge -c/2\\b \ge -c/3}} q_1^a q_2^b q_3^c e^{(2a+c)t_1 + a(t_2+t_3) + (3b+c)t_4 + b(t_5+t_6+t_7) + 2ct_8} \times \nonumber \\
& \frac{\Gamma(2D_E/z+\frac{1}{2}) \Gamma(D_E/z+1)^2 \Gamma(3D_K/z+\frac{1}{2})\Gamma(D_K/z+1)^3 }{\Gamma(2D_E/z+2a+c+1)\Gamma(D_E/z+a+1)^2\Gamma(3D_K/z+3b+c+1)\Gamma(D_K/z+b+1)^3})\times \nonumber \\
 & \frac{\Gamma(4D_E/z+4a+2c+1)\Gamma(6D_K/z + 6b+2c + 1)}{\Gamma(2c+1)\Gamma(4D_E/z+1)\Gamma(6D_K/z+1)}\mathbf{1}_{\sigma}\bigg).
\end{align*}

For our cases of interest we find the I-function for the first mixed theory to be 
\begin{align*}
& I_{\mathrm{FJRW},\,\mathrm{GW}}(\mathbf{t}, z) = ze^{(w_0t_{4} + w_1t_5+w_2t_6+w_3t_7)D_K/z} e^{-z} \times \\
& \sum_{n_{3}, n_{\sigma} \in \mathbb{N}_0^3} 2 \frac{\Gamma(\frac{1}{4}+\frac{n_{3}}{2} + \frac{n_{\sigma}}{4})^2 t_3^{n_3} t_{\sigma}^{n_{\sigma}}}{\Gamma(\frac{1}{4}+\langle\frac{n_3}{2} + \frac{n_{\sigma}}{4}\rangle+1)^2n_3!n_{\sigma}!}z^{\lfloor \frac{n_{3}}{2} + \frac{n_{\sigma}}{4} \rfloor - (n_{3} + n_{\sigma})}  \sum_{\mathbf{b} \in \mathrm{Box}([K/\langle \sigma_K\rangle])}L_{\mathbf{b}}\times\\
& \sum_{\substack{(b, c, \mathbf{k}) \in \Lambda E_{\mathbf{b}}^S([K/\langle \sigma_K\rangle])\\ 2c+ \sum_{j=l+1}^mk_j = n_{\sigma}}} (q_2^b q_3^c \prod_{j=1}^mx_j^{k_j})e^{b(\sum_{i=1}^nw_it_i)+c(t_4+2t_8)}G_{\mathbf{b}}(b, c, \mathbf{k})\phi_{h_{n_3, n_{\sigma}}}\mathbf{1}_{\mathbf{b}}\\ 
=: & \sum_{n_1, n_3, n_{\sigma}} \sum_{\mathbf{b} \in \mathrm{Box}([K/\langle \sigma_K\rangle])}  \omega_{n_3, n_{\sigma}, \mathbf{b}}^{\mathrm{FJRW},\,\mathrm{GW}} z^{\lfloor \frac{n_{3}}{2} + \frac{n_{\sigma}}{4} \rfloor - (n_{1} + n_{3} + n_{\sigma})} \phi_{n_1, n_3, n_{\sigma}} \mathbf{1_b}.
\end{align*}

For our cases of interest we find the I-function for the second mixed theory to be
\begin{align*}
& I_{\mathrm{GW},\,\mathrm{FJRW}}(\mathbf{t}, z) = ze^{(2t_{1} + t_2+t_3)D_E/z} e^{-z}\times \\
& \sum_{n_{3}, n_{\sigma} \in \mathbb{N}_0^3} 2 \frac{\Gamma(\frac{1}{6}+\frac{n_{3}}{3} + \frac{n_{\sigma}}{6})^3t_3^{n_3} t_{\sigma}^{n_{\sigma}}}{\Gamma(\frac{1}{6}+\langle\frac{n_3}{3} + \frac{n_{\sigma}}{6}\rangle+1)^3n_3!n_{\sigma}!}z^{\lfloor \frac{n_{3}}{3} + \frac{n_{\sigma}}{6} \rfloor - (n_{1} + n_{3} + n_{\sigma})} \times\\
&  \sum_{\mathbf{b} \in \mathrm{Box}([E/\langle \sigma_E\rangle])}K_{\mathbf{b}}\sum_{\substack{(a, c) \in \Lambda E_{\mathbf{b}}^S([E/\langle \sigma_E\rangle])\\ 2c = n_{\sigma}}} q_2^a q_3^ce^{a(2t_1+t_2+t_3)+c(t_1+2t_8)}F_{\mathbf{b}}(a, c)\mathbf{1}_{\mathbf{b}}\phi_{h_{n_3, n_{\sigma}}}\\ 
=: & \sum_{n_1, n_3, n_{\sigma}} \sum_{\mathbf{b} \in \mathrm{Box}([E/\langle \sigma_E\rangle])}  \omega_{\mathbf{b}, n_3, n_{\sigma}}^{\mathrm{GW},\,\mathrm{FJRW}} z^{\lfloor \frac{n_{3}}{3} + \frac{n_{\sigma}}{6} \rfloor - (n_{3} + n_{\sigma})} \mathbf{1_b} \phi_{n_1, n_3, n_{\sigma}}. 
\end{align*}

The  FJRW I-function for $W = X^2+Y^4+Z^4+x^2 + y^6 + z^6 + w^6 = 0$ we find to be
\begin{align*}
&I^{\mathrm{nar}}_{FJRW}(\mathbf{t}, z) =  \sum_{M, N, C \ge 0} \frac{\Gamma(\frac{1}{2})\Gamma(\frac{1}{4} + \frac{M}{2} + \frac{C}{4})^2 \Gamma(\frac{1}{2})\Gamma(\frac{1}{6} + \frac{N}{3} + \frac{C}{6})^3}{\Gamma(\frac{1}{2}+1) \Gamma(\frac{1}{4} + \langle \frac{M}{2} + \frac{C}{4}\rangle + 1)^2 \Gamma(\frac{1}{2} + 1) \Gamma(\frac{1}{6} + \langle \frac{N}{3} + \frac{C}{6}\rangle + 1)^3}\\
& \times \frac{T_1^MT_2^NT_3^C}{M!N!C!}z^{1- M - N - C + 2 \lfloor \frac{M}{2} + \frac{C}{4}\rfloor + 3 \lfloor \frac{N}{3} + \frac{C}{6}\rfloor}\phi_{h({M, N, C})}\\
& =: \sum_h \sum_{(M, N, C): h(M, N, C) = h} \omega_h^{\mathrm{FJRW}}z^{1- M - N - C + 2 \lfloor \frac{M}{2} + \frac{C}{4}\rfloor + 3 \lfloor \frac{N}{3} + \frac{C}{6}\rfloor} \phi_h.
\end{align*}
where $h(M, N, C)$ is given in $(\mathbf{C}^*)^7$ by
$$(-1, e^{2\pi i (\frac{M}{2} + \frac{C}{4} + \frac{1}{4})}, e^{2\pi i (\frac{M}{2} + \frac{C}{4} + \frac{1}{4})}, -1, e^{2\pi i (\frac{N}{3} + \frac{C}{6} + \frac{1}{4})}, e^{2\pi i (\frac{N}{3} + \frac{C}{6} + \frac{1}{4})}, e^{2\pi i (\frac{N}{3} + \frac{C}{6} + \frac{1}{4})}).
$$

\subsection{Structure of this Paper and Statement of Main Result}

Borcea-Voisin orbifolds are some of the first examples of non-trivial Calabi-Yau threefolds, each given as $\widetilde{E\times K/\mathbb{Z}_2}$, for $E$ an elliptic curve, $K$ a K3 surface, and $\mathbb{Z}_2$ generated by the product of anti-symplectic involutions on both factors. We consider special cases where both factors can be given as algebraic hypersurfaces in weighted projective space, and the involution acts by negating the first coordinate of each.

In this paper we first give an overview of Gromov-Witten theory, FJRW theory and the hybrid theories as defined in \cite{FJRnew}. Each of these is encapsulated in a state space (which provides cohomological information) and an I-function (which provides enumerative information) defined from a certain moduli space. On the quotient of a variety cut out by polynomials $W_1, \ldots, W_n$ by a group $G$, Gromov-Witten theory gives invariants which count curves going through subvarieties in given homology classes, with certain corrections to allow integration over the moduli space to be well-defined. FJRW theory gives invariants counting curves endowed with line bundles of given multiplicities at certain marked points, subject to conditions depending on the $W_i$ and $G$. In our case, we have two polynomials $W_1, W_2$, and this allows us to define two intermediate theories, similar to the Gromov-Witten theory for one and the FJRW theory for the other. In this paper, we find the state spaces and I-functions for all of these theories in certain cases. We restrict to genus-zero curves, and \textit{ambient} and \textit{narrow} theories, considering only those classes induced from the cohomology of the ambient space.

We end by proving the following:
\begin{thm}

\begin{description}
\item[(1)] For Borcea-Voisin orbifolds of Fermat type with quartic elliptic curve, the state spaces of all four theories are isomorphic as graded inner product spaces.

\item[(2)] For Borcea-Voisin orbifolds of Fermat type, the narrow/ambient state spaces of all four theories are isomorphic as graded inner product spaces.

\item[(3)] For Borcea-Voisin orbifolds of Fermat type, a mirror theorem holds in the sense that the narrow I-functions found in this paper, if written in the form $I^{\circ}(\mathbf{t}, z) = f(\mathbf{t})z + \mathbf{g}(\mathbf{t}) + \mathcal{O}(\mathbf{t}^2),$ are related to the narrow J-functions of their respective theories by $$J^{\circ}(\tau(\mathbf{t}), z) = \frac{I^{\circ}(\mathbf{t}, z)}{f(\mathbf{t})} \text{, where } \mathbf{\tau}(\mathbf{t}) = \frac{\mathbf{g}(\mathbf{t})}{f(\mathbf{t})}.$$

\item[(4)] For Borcea-Voisin orbifolds of Fermat type with quartic elliptic curve, the mixed theory I-functions $I_{\mathrm{FJRW}, \mathrm{GW}} (\mathbf{t}, z)$ are related to $\mathrm{I}_{\mathrm{GW}}(\mathbf{t}, z)$ by analytic continuation and symplectic transformation.

\item[(5)] For Borcea-Voisin orbifolds of Fermat type with K3 surface $\{x^2 + y^6 + z^6 + w^6 = 0\}$, the mixed theory I-functions $I_{\mathrm{GW}, \mathrm{FJRW}} (\mathbf{t}, z)$ are related to $\mathrm{I}_{\mathrm{GW}}(\mathbf{t}, z)$ by analytic continuation and symplectic transformation.

\item[(6)] For $E = \{X^2 + Y^4 + Z^4 = 0\}$ in $\mathbb{P}(2, 1, 1)$, $K = \{x^2 + y^6 + z^ 6 + w^ 6 = 0\} \subseteq \mathbb{P}(3, 1, 1, 1)$, and $\sigma: (X, x) \mapsto (-X, -x)$, $I_{\mathrm{FJRW}}(\mathbf{t}, z)$ is related to $\mathrm{I}_{\mathrm{GW}}(\mathbf{t}, z)$ by analytic continuation and symplectic transformation.
\end{description}
\end{thm}

\subsection{Acknowledgements}

I am grateful to Prof Yongbin Ruan for his incomparable help throughout, to Thomas Coates for his generosity in advice and fruitful discussion, to Nathan Priddis for his taking time to explain certain key concepts to me, and to Pedro Acosta and Emily Clader for helpful discussions.

\section{Preliminaries}
\subsection{Borcea-Voisin Orbifolds}

Any elliptic curve $E$ is endowed with an involution $\sigma_E$ whose induced map on $H^2(E)$ is $-\mathrm{id}$, most simply given as that induced by the map $z \mapsto -z$ in $\C$, if $E$ is considered as the quotient of $\C$ by a lattice.

Similarly, several K3 surfaces $K$ are also endowed with involutions $\sigma_K$ such that the induced map on $H^2(K)$ is also $-\mathrm{id}$. These `anti-symplectic involutions' were explored and mostly classified by Nikulin in \cite{N}.
The fixpoint sets of such involutions are unions of curves, which are either empty, have at least one of genus more than one, or are the union of exactly two curves of genus one. \cite{V}

$E \times K$ is Calabi-Yau, as the product of two Calabi-Yau manifolds. It has an involution $\sigma := \sigma_E \times \sigma_K$, whose induced map on cohomology is now the identity. The quotient $[E\times K/\langle\sigma\rangle]$, in general, has singularities (unless the fixpoint set of $\sigma_K$ was empty, in which case we have the Rodrigues surface). We may resolve these canonically, and the corresponding manifold $\widetilde{E\times K}/\mathbb{Z}_2$, known as a \textit{Borcea-Voisin manifold}, is also Calabi-Yau. To avoid considering the resolution of singularities separately, we treat the quotient itself as a \textit{Borcea-Voisin orbifold} $\mathcal{Y} = [E\times K]/\mathbb{Z}_2$,  the main objects of study of this paper.

\subsection{The Chen-Ruan cohomology of Borcea-Voisin Orbifolds}

 $H_{CR}^*(\mathcal{Y})$ decomposes into two parts: first, there is a part coming from $\sigma$-invariant classes in $H^*(E \times K)$, which in turn decomposes into $(H^+(E) \otimes H^+(K)) \oplus (H^-(E) \otimes H^-(K))$, where $H^\pm$ denotes the eigenspace of $\sigma_E$ or $\sigma_K$ respectively with eigenvalue $\pm1$. Let the fixed point set of $\sigma_K$ be $\Sigma = \coprod_{i=1}^N C_i$, where $C_i$ is connected and has genus $g_i$. Let $N' = \sum g_i$. We compute $\chi(K/\sigma_K) = 12 + N - N'$. From this, and degree considerations, we have
 
 $$H^+(E) = 
 \begin{tabular}{llll}
  &  & 1 & \\
  & 0 & & 0 \\
  & & 1 & 
 \end{tabular}, \quad H^-(E) = \begin{tabular}{llll}
    & & 0 & \\
   & 1 & & 1 \\
   & & 0 & 
 \end{tabular}$$
 
 $$H^+(K) = 
 \begin{tabular}{llllll}
  &  &  & 1 & & \\
  & & 0 & & 0 & \\
  & 0 &  & $a$ &  & 0 \\
  & & 0 & & 0 & \\
  &  &  & 1 & & 
 \end{tabular}, \quad H^-(K) = \begin{tabular}{llllll}
  &  &  & 0 & & \\
  & & 0 & & 0 & \\
  & 1 &  & $b$ &  & 1 \\
  & & 0 & & 0 & \\
  &  &  & 0 & & 
 \end{tabular}$$

where $a = 10+N-N', b = 10-N+N'$. The total invariant part of the cohomology is then $$\begin{tabular}{llllllll}
    & & & 1 & & & \\
   & & 0 & & 0 & & \\
   & 0 &  & $a+1$ &  & 0 & \\
   1 & & $b+1$ & & $b+1$ & & 1 \\
   & 0 &  & $a+1$ &  & 0 & \\
   & & 0 & & 0 & & \\
   & & & 1 & & & \\
 \end{tabular}$$

Second, there is a part coming from the twisted orbifold sectors, or classically from the fixed point locus of the involution. In the Chen-Ruan formalism, this is given by the cohomology of the fixed point sets of the conjugacy classes of the group with the index `twisted' by a number called the age. In our case, we have only one non-trivial conjugacy class $\{\sigma\}$, and $\mathrm{Fix}(\sigma) = 4\coprod_{i=1}^N C_i$. The normal bundle of $\Sigma$ has rank 2, on which the involution acts with eigenvalue $e^{\frac{1}{2}(2\pi i)}.$ We therefore include the cohomology of $\Sigma$:
$$\begin{tabular}{llll}
  &  & $4N$ & \\
  & $4N'$ & & $4N'$ \\
  & & $4N$ & 
 \end{tabular}
 $$
 but shift the degree by 1, ceding
 \begin{equation}\label{eq:coh} H_{CR}^*(\mathcal{Y}) = \begin{tabular}{llllllll}
    & & & 1 & & & \\
   & & 0 & & 0 & & \\
   & 0 &  & $h_{1,1}$ &  & 0 & \\
   1 & & $h_{2, 1}$ & & $h_{2, 1}$ & & 0 \\
   & 0 &  & $h_{1,1}$ &  & 0 & \\
   & & 0 & & 0 & & \\
   & & & 1 & & & \\
 \end{tabular},\end{equation}
 where $h_{1,1} = 11+5N-N'$, $h_{2, 1} = 11+5N'-N$. It turns out that for every Nikulin involution of a K3 surface whose fixpoint set has $N$ components whose genera sum to $N'$, there is another with $N'$ components whose genera sum to $N$, \cite{V} These therefore correspond to mirror Borcea-Voisin orbifolds in the Hodge diamond sense, with $N = 0, N' = 0$ and $N = 2, N' = 2$ corresponding to self-mirror orbifolds.
 
 If $E$ and $K$ can be given by equations inside weighted projective spaces $\mathbb{P}(\mathbf{w}_E), \mathbb{P}(\mathbf{w}_K)$, we define the \textit{ambient space} to be $\mathcal{X} = [(\mathbf{P}(\mathbf{w}_E)\times \mathbf{P}(\mathbf{w}_K)/\tilde{\sigma}]$, where $\tilde{\sigma}$ lifts $\sigma$. We define the \textit{ambient cohomology}  $H^{\mathrm{amb}}(\mathcal{Y})   = i*(H_{\mathrm{CR}}^*(\mathcal{X})) \subseteq H_{\mathrm{CR}}^*(\mathcal{Y})$, induced by the inclusion $i:\mathcal{Y} \hookrightarrow \mathcal{X}$. From basic weighted projective geometry this is always of the form  $$H_{CR}^*(\mathcal{X}) = \begin{tabular}{llllllll}
    & & & 1 & & & \\
   & & 0 & & 0 & & \\
   & 0 &  & $h^{\mathrm{amb}}_{1, 1}$ &  & 0 & \\
   0 & & 0 & & 0 & & 0 \\
   & 0 &  & $h^{\mathrm{amb}}_{1, 1}$ &  & 0 & \\
   & & 0 & & 0 & & \\
   & & & 1 & & & \\
 \end{tabular},$$
 
for $h^{\mathrm{amb}}_{1, 1} \le h_{1, 1}$.
 
 \subsection{Gromov-Witten theory}
 
 A marked (possibly orbifold) curve $(C, p_1, p_2, \ldots, p_n)$ is \textit{stable} if $C$ is connected, compact, at worst nodal, no marked point is a node, and has finitely many automorphisms which fix the marked points. This last condition is equivalent to every genus 0 irreducible component containing at least three marked points and every genus 1 component having at least one marked point.
 
Given an orbifold $\mathcal{Y}$, we call a map $f:(C, p_1, p_2, \ldots, p_n) \to \mathcal{Y}$ \textit{stable} if and only if every component of every fibre is stable; thus, the map can only be constant on an irreducible component of $C$ if that component is stable.

There is a well-defined projective moduli stack $\overline{\mathfrak{M}}_{g, n}(\mathcal{Y}, \beta)$ of stable maps to $\mathcal{Y}$, where the source curves are of genus $g$ with $n$ marked points, and the image of the maps lie in the class $\beta \in H_2(\mathcal{Y})$. \cite{CheRu1} There are subtleties about integrating over this moduli space. Instead of the fundamental class we generally integrate over a specified \textit{virtual fundamental class} $[\overline{\mathfrak{M}}_{g, n}(\mathcal{Y}, \beta)]^{\mathrm{vir}}$ whose definition may be found in \cite{FP}.

An orbifold comes with an \textit{inertial manifold} $I\mathcal{Y} = \coprod_{(g)} Fix(g)$ which serves to decompose the orbifold into parts corresponding to different group conjugacy classes at the orbifold points. This moduli space is endowed with evaluation maps $\mathrm{ev}_i: \overline{\mathfrak{M}}_{g, n}(\mathcal{Y}, \beta) \to I\mathcal{Y}$,  given in the manifold case by $\mathrm{ev}_i:f \mapsto f(p_i)$. In the orbifold case, the inertial orbifold allows us to keep track of which twisted sector the evaluation map sends a marked point to.
There is a well-defined smooth orbibundle $\mathbb{L}_i \to \overline{\mathfrak{M}}_{g, n}(\mathcal{Y}, \beta)$  whose fibre at each curve $f$ in the moduli space is the cotangent space at $f(p_i)$. (At orbifold points this differs from the corresponding tangent space for the underlying space by a factor of the multiplicity of the point.) We define $\psi_i = c_1(\mathbb{L}_i)$. Then the \textit{descendant Gromov-Witten invariants} are given by 
$$\langle \tau_{a_1}(\gamma_1),\ldots, \tau_{a_n}(\gamma_n) \rangle_{g, n}^{\beta}  = \int_{[\overline{\mathfrak{M}}_{g, n}(\mathcal{Y}, \beta)]^{\mathrm{vir}}} \prod_{i=1}^n \mathrm{ev}_i^*(\gamma_i))\psi^{a_i}. $$

All these invariants may be packaged into  $$ \langle \mathbf{t}, \ldots, \mathbf{t}\rangle_{g, n}^{\beta} = \sum_{k_1, \ldots, k_n \ge 0} \langle \tau_{k_1}(t_{k_1}), \ldots, \tau_{k_n}(t_{k_n})\rangle,$$ for $\sum_{i\ge 0} t_iz^i \in H_{\mathrm{CR}}^*(\mathcal{Y})[[z^{-1}]]$. This is the generating function for the genus-$g$ \textit{descendant potential}

$$\mathcal{F}_{\mathcal{Y}}^g (\mathbf{t}) = \sum_{n \ge 0} \sum_{\beta \in H_{\mathrm{CR}}^2(\mathcal{Y})} \frac{q^d}{n!} \langle \mathbf{t}, \mathbf{t}, \ldots \mathbf{t}\rangle_{g, n}^{\beta},$$ which takes values over the Novikov ring $\mathbb{C}[[ H_2(\mathcal{Y})]]$ (or, in our case, the subring involving only effective classes $\mathbb{C}[[ H_2(\mathcal{Y}) \cap NE(\mathcal{Y})]]$.)

As in \cite{CheRu1}, we generally specify a degree-two generating set $\{ \phi_{\alpha}\}$ of $ H_{\mathrm{CR}}^*(\mathcal{Y})$, and write $\mathbf{t} = \sum_{k\ge0}\sum_{\alpha} t_k^{\alpha} \phi_{\alpha} z^k$. 
 
 \subsection{FJRW Theory}

This subsection closely follows \cite{FJR}. A Landau-Ginzburg model is given by  quasi-homogeneous polynomial function $W:\mathbb{C}[x_1, x_2, \ldots x_n] \to \mathbb{C}$ with weights $\bar{w}_1, \bar{w}_2, \ldots, \bar{w}_N$ and degree $d$, and a group $G$ which leaves $W$ invariant.  We denote by $q_i$, $1 \le i \le n$ the \textit{charges} $\frac{\bar{w}_i}{d}$. The hypersurface $X_W$ is Calabi-Yau if and only if $\sum_i q_i = 1$. We shall exclusively work in the \textit{Fermat} case, in which $W = \sum_{i=1}^N x^{a_i}$, where $a_i = \frac{d}{\bar{w}_i} = \frac{1}{q_i}$.

The \textit{maximal group} $G_{\mathrm{max}}$ of all diagonal symmetries leaving $W$ invariant is given by $$\{(e^{\frac{2 \pi i r_1}{a_1}}, \ldots, e^{\frac{2 \pi i r_N}{a_N}})\vert 1 \le r_i a_i\}.$$ There is a special group element $J := (e^{\frac{2 \pi i}{a_1}}, \ldots, \frac{2 \pi i} {a_N})$.

Then $(W, \langle J \rangle)$ corresponds to the hypersurface $X_W = \{ W = 0\} \subseteq (\mathbb{C}^N\backslash \{0\})/\langle J \rangle = \mathbb{P}(w_1, \ldots, w_N)$.

More generally, a well-defined FJRW theory may be given by any group such $G$ such that $\langle J \rangle \subseteq G \subseteq G_{\mathrm{max}}$
(by a result of Krawitz, \cite{K}, these correspond to the \textit{admissible} groups of \cite{FJR}), and this morally corresponds to the orbifold $[X_W/(G/\langle J \rangle)]$). 

It is easy to check that if $W_1, W_2$ share no variables in common, then $X_{W_1} \times X_{W_2}$ corresponds to $(W_1+W_2, \langle J_1, J_2\rangle)$. In our case, we will consider a Landau-Ginzburg theory of the form $(W_1+W_2, \langle J_1, J_2, \sigma \rangle)$.

For $h \in G$, let $N_h$ be the dimension of the fixpoint subspace $\mathrm{Fix}(h) \subseteq \mathbb{C}^N$, let $W_h = W\vert_{\mathrm{Fix}(g)}$, and let $W_h^{+\infty} = (\mathrm{Re} W_h)^{-1}(]\rho, +\infty[)$ for $\rho >> 0$. 

The state space is given by $$\mathcal{H}_{\mathrm{FJRW}}(W, G) = \bigoplus_{h \in G} \mathcal{H}_{h},$$ where $$\mathcal{H}_h = H^{N_h}(\mathrm{Fix}(h), W_h^{\infty}; \mathbb{C})^G.$$

The sectors corresponding to $h$ for which $\mathrm{Fix}(h) = \{ 0 \}$ are termed \textit{narrow} sectors, and we denote their union $\mathcal{H}_{\mathrm{nar}}$. These will be seen to correspond to the ambient classes in Gromov-Witten theory. All other sectors are termed $\textit{broad}$.

Analogously to Chen-Ruan cohomology, we define the \textit{age} of the action of $g$ with eigenvalues $e^{2\pi i \Theta_k(h)}$ to be $\sum_k \Theta_k(h)$, and the Hodge bidegree of $H_g$ is then shifted by $(\mathrm{age}(g) - 1, \mathrm{age}(g) -1)$, giving total degree $\mathrm{deg}_W (\alpha) = N_h + 2(\mathrm{age}(g) - 1)$.

Analogously to Chen-Ruan cohomology, we define the \textit{age} of the action of $g$ with eigenvalues $e^{2\pi i \Theta_k(h)}$ to be $\sum_k \Theta_k(h)$, and the Hodge bidegree of $H_g$ is then shifted by $(\mathrm{age}(g) - 1, \mathrm{age}(g) -1)$, giving total degree $\mathrm{deg}_W (\alpha) = N_h + 2(\mathrm{age}(g) - 1)$.

An alternative construction can be given by the theory of the Milnor ring $$\mathcal{Q}_W := \mathbb{C}[x_1, x_2, \ldots, x_N]\mathcal{J}_W$$ where $\mathcal{J}_W$ is the Jacobian ideal $$(\frac{\partial W}{\partial x_1}, \frac{\partial W}{\partial x_2}, \ldots \frac{\partial W}{\partial x_N}).$$ This is a complex vector space of dimension $$\mu = \prod_{i=1}^N (\frac{1}{q_i} - 1),$$ where it is a straightforward exercise to show that the highest degree attained is that of $\mathrm{Hess}(W)$, equal to $\sum_{i=1}^N (1-2q_j)$, which we label the \textit{central charge}, denoted by $\hat{c}$.

Then we can alternatively define the FJRW state space sector-wise by $$\mathcal{H}_g  = \Omega^{N_g}(\mathrm{Fix}(g))/(\mathrm{d},W \vert_{\mathrm{Fix}(g)}\wedge \Omega^{N_g-1}) \cong \mathcal{Q}_{W\vert _{\mathrm{Fix(g)}}}\cdot \mathrm{d}\,x_1 \wedge \ldots \wedge \mathrm{d}\,x_N.$$

Let $\phi_g$ be the fundamental class in $\mathcal{H}_g$, and write $\phi^g = \phi_{g^{-1}}$.

There is a natural pairing $\eta(\cdot, \cdot)$ given in the context of the Milnor ring, given by $$fg = \frac{\langle f, g\rangle}{\mu} \mathrm{Hess}(W) + \mathrm{lower\,terms}.$$

FJRW theory associates a moduli stack $\mathcal{W}_{g, n}^{(W, G)}$ of curves endowed with line bundles and some further structure to each Landau-Ginzburg model $(W, G)$. It was originally defined to give an appropriate context for the solution of the \textit{Witten equation} defined in \cite{Wit} $$\overline{\partial} u_i + \frac{\overline{\partial W}}{\partial u_i},$$ where $W$ is a quasi-homogeneous polynomial and and $u_i$ is a section of a line bundle over some complex curve $C$. a full treatment can be found in \cite{FJR}.

Instead of considering maps from the curves to an ambient space, we consider specified sets of line bundles which each loosely correspond to coordinates. 

A \textit{$d$-stable $W$-spin orbicurve} is a marked orbicurve $(C, p_1, \ldots, p_n)$ with at worst nodal singularities endowed with line bundles $\mathcal{L}_1, \mathcal{L}_2, \ldots, \mathcal{L}_N$ such that for each monomial term $W_i = \prod_{j=1}^N x_j^{a_{ij}}$ of $W$ we have $$\bigotimes_{j=1}^N\mathcal{L}_j^{\otimes a_{ij}} \cong \omega_{\mathrm{log}} := \omega_C \otimes (\sum_{i=1}^n p_i),$$ and isomorphisms $$\phi_i: \mathcal{L}_i^{\otimes d} \to \omega_{\log}.$$

There is a proper Deligne-Mumford stack $\mathcal{W}_{g, n}$ of $n$-marked $W$-spin curves of genus $g$. \cite{FJR}

We represent $h \in G_{\mathrm{max}} $ by $(e^{2 \pi i\Theta_1(h)}, \ldots, e^{2 \pi i \Theta_N(h)})$. Then the moduli space decomposes into $\coprod_{\mathbf{h} \in G^n} W_{g, n}(\mathbf{h})$, where $\Theta_k(h_i) = \mathrm{mult}_{p_i}(\mathcal{L})/d$. For a curve in $\mathcal{W}_{g, n}(\mathbf{h})$, the corresponding line bundle over the coarse space $\vert \mathcal{L}_k \vert$ has degree $$q_k (2g-2+n) - \sum_{i=1}^n \Theta_k(h_i),$$ which must be an integer - this must be true for all $k$ for $\mathcal{W}_{g, n}(\mathbf{h})$ to be non-empty.

For admissible \cite{FJR, K} $G \subseteq G_{\mathrm{max}}$, there is some quasi-homogeneous polynomial $Z$ such that $G_{\max}(W+Z) = G$. Then $\mathcal{W}_{g, n, G} \subseteq \mathcal{W}_{g, n}$ is the (proper \cite{FJR}) substack of $(W+Z)$-orbicurves.

There is a virtual cycle $[W_{g,n,G}]^{\mathrm{vir}}$ of degree $2(\hat{c}-3)(1-g) + n - \sum_{i=1}^nn \iota (h_i)$, and there are $\psi$-classes defined similarly to the Gromov-Witten case. \cite{FJR} 

The FJRW invariants are given by $$\langle \psi^{a_1} \phi_{h_1}, \psi^{a_2}\phi_{h_2}, \ldots, \psi^{a_n}\phi_{h_2}\rangle_{g, n}^{(W, G)} = (\prod_k q_k )^{g-1} \int_{[\mathcal{W}_{g, n, G}(h_1, \ldots, h_n) ]^{\mathrm{vir}}} \prod_{i=1}^n \psi_i^{a_i}.$$ This is defined via the virtual class of the moduli space, which is given by $-{R\pi_*(\bigoplus_{k=1}^N \mathcal{L}_k)}^{\vee}$.


We also have a product structure $\cup$ given by $$\eta_(\alpha \cup \beta, \gamma) = \langle \alpha, \beta, \gamma \rangle.$$

For several reasons it is less complicated to consider only the FJRW theory of the narrow sectors, and this paper will only compute the FJRW invariants involving $\phi_h \in \mathcal{H}_{\mathrm{nar}}$. However, it will also help to be able to use the whole group $G$. We define the \textit{extended narrow state space} $$\mathcal{H}^{\mathrm{ext}} := \mathcal{H}^{\mathrm{nar}} \oplus \bigoplus_{h \in G\backslash \mathcal{N}} \mathbb{C}\phi_h,$$ where $\mathcal{N}$ is the set of $h \in G$ giving narrow sectors.

Finally, we define our twisted invariants, corresponding to an integral is not over the whole ambient space, but rather the sub-variety of interest. This is also given in \cite{FJR}, based on the formalism given for Gromov-Witten invariants in \cite{Giv}.





\subsection{The Fan-Jarvis-Ruan GLSM}

This subsection is a very quick overview of the required vocabulary of the theory of the Gauged Linear Sigma Model defined in \cite{FJRnew}. Let $G$ be a reductive group in $\mathrm{GL}(V)$, for $V$ some vector space. Let $\theta \in \hat{G}$ be some character of $G$, and let $L_{\theta}$ be the induced line bundle over $V$, given by $V \times \mathbb{C}$ with $G$ acting by $g: (\mathbf{v}, \tilde{z}) \mapsto (g\cdot \mathbf{v}, \theta(g)\tilde{z}).$  A point $\mathbf{v}$ of $V$ is \textit{semistable} under the action of $G$ and a choice of $\theta$ if there is some positive integer $m$ and some $G$-invariant section $f \in H^0(V, L^{\otimes m}; \mathbb{C})^G$ for which $f(\mathbf{v}) \ne 0$. We let $V^{ss}(\theta)$ be the set of $(G, \theta)$-semistable points. A semistable point is \textit{stable} if it has finite stabiliser and closed $G$-orbit, and we label the set of stable points $V^s(\theta)$. In general, this removes a zero-standard-measure set of `bad points' which cause the quotient $[V/G]$ to be non-separated.  In our cases of interest, $V^s(\theta) = V^{ss}(\theta)$. We define the GIT quotient $[V//_{\theta}G]$ to be $[V^{ss}(\theta)/G]$. This plays the role of the `ambient space' of the theory.

This theory does not solely consider the action of $G$, but rather the action of an extension of $G$ which generalises the weights. Let $\mathbb{C}^*_R$ be $\mathbb{C}^*$ acting on $V$ by $$\lambda:(v_1, \ldots, v_n) \mapsto (\lambda^{c_1}v_1, \ldots, \lambda^{c_n}v_n)$$. The $c_i$ are not necessarily positive. If the action of $\mathbb{C}^*_R$ commutes with the action of $G$ and $G \cap \mathbb{C}^*_R = \langle J \rangle,$ where $J = (e^{2\pi i c_1/d}, \ldots, e^{2\pi i c_n/d}).$ Let $\Gamma = $. Let $\vartheta$ be a lift of $\theta$ from $G$ to $\Gamma$. It is cleart that $V^{ss}(\vartheta) \subset V^{ss}(\theta).$ We equality holds, we call such a lift $\vartheta$ \textit{good}. In our case, we shall only be considering the trivial lift.

For the last piece of input data, fix a non-degenerate $G$-invariant polynomial $W$ defined on $V$ of degree $d$ and charges $q_i = c_i/d$. This shall be the \textit{superpotential} of the theory. The image of the critical set $\mathcal{CR}$ of $W$ in the GIT-quotient, $[\mathcal{CR}//_{\theta}G]$, is of special interest. If $W$ is non-degenerate, then $\mathcal{CR}$ is compact.

The theory is endowed with a state space directly generalising that of FJRW theory. When $G$ is abelian, we define $\mathcal{H}_{\mathrm{GLSM}} = \bigoplus_{\alpha \in \mathbb{Q}} H_{\mathrm{CR}}^{\alpha+2q}([V//_{\theta} G], W^{\infty}; \mathbb{C}),$ where $W^{\infty} = (\mathrm{Re}W)^{-1}([0, \rho])$ for $\rho >>0$, and $q = \sum_i q_i$. This plays the role of the ambient state space.

If an element $g \in G$ has compact inertia stack component (that is, fixpoint set) then it is \textit{narrow}. $\mathcal{H}_{\mathrm{nar}}  = \bigoplus_{g \text{, narrow}} \mathcal{H}_{\mathrm{CR}}^{(g)}.$ If an element of $\mathcal{H}_{\mathrm{GLSM}}$ is Poincar\'e-dual to a substack of $\mathcal{CR},$ then it is \textit{critical}. Let $\mathcal{H}_{\mathrm{GLSM}, \mathrm{comp}}$ be the span of the narrow and critical elements. This will encompass the narrow and ambient sectors in FJRW and GW theory, respectively.

The moduli space of the theory classifies \textit{stable Landau-Ginzburg quasimaps}. These are given by tuples $$(\mathcal{C}, y_1, \ldots, y_n, \mathcal{P}, \sigma, \kappa),$$ such that $(\mathcal{C}, y_1, y_2, \ldots, y_n)$ is an $n$-pointed orbicurve, $\mathcal{P}:\mathcal{C} \to B\Gamma$ is a representable principal $\Gamma$-bundle, $\sigma:\mathcal{C} \to \mathcal{E} = \mathcal{P}\times_{\Gamma} V$ is a global section, and $\kappa:\zeta^*\mathcal{P} \to \omega{\mathrm{log}, \mathcal{C}}^{\circ}$ is an isomorphism of principal $\mathbb{C}^*$-bundles, where in turn $\zeta:\Gamma \to \mathbb{C}^*_{\mathrm{R}}$ is the group homomorphism sending $G$ to 1 and $(\lambda^{c_1}, \ldots, \lambda^{c_n})$ to $\lambda^d$, and $\omega{\mathrm{log}, \mathcal{C}}^{\circ}$ is the induced principal $\mathbb{C}^*$-bundle associated to $\omega_{\mathrm{log}, \mathcal{C}}.$ They moreover satisfy certain technical stability conditions detailed in \cite{FJRnew} which depend on the lift $\vartheta$ and a rational number $\epsilon$, which restricts the behaviour of $\sigma$ on certain points of $mathcal{C}$. In our case, we shall also require that $\sigma$ induces a map $[\sigma]: \mathcal{P} \to V$ with image in $\mathcal{CR}(W)$. They define the notion of the \textit{degree} of a quasi-map, which coincides with the degree of the image in the Gromov-Witten case. We let $\mathrm{LGQ}_{g, n}^{\epsilon, \vartheta}(\mathcal{[\mathcal{CR}^{ss}/G]}, \beta)$ be the moduli space of such stable $n$-pointed genus-$g$ LG quasimaps for given $\vartheta, \epsilon$ and degree $\beta$.

By the main theorems of \cite{FJR}, This moduli space is a Deligne-Mumford stack.

This moduli space has a well-defined virtual class, $\psi$-classes, and evaluation maps $ev_i:\mathrm{LGQ}_{g, n}^{\epsilon, \vartheta}(\mathcal{[\mathcal{CR}^{ss}/G]}) \to \mathbb{I}[\mathcal{CR}//_{\theta}G]$ which allow us to define invariants 
$$\langle \tau_{l_1}(\alpha_1), \ldots, \tau_{l_n}(\alpha_n) \rangle = \int_{[\mathrm{LGQ}_{g, n}^{\epsilon, \vartheta}(\mathcal{[\mathcal{CR}^{ss}/G]}, \beta)]}^{\mathrm{vir}} \prod_{i=1}^n ev_i^*(\alpha_i) \psi_i^{l_i}.$$

In our case of interest, the possible weights for the character $\theta$ will divide the GLSM into four chambers, one of which will correspond to  Gromov-Witten theory, one to FJRW theory, and the other two to certain `mixed' theories.
 
\subsection{Givental's Formalism for Mirror Symmetry}
 
The definitions of Gromov-Witten theory, FJRW theory and the mixed theories have already been analogous in several ways. For this section, the treatment is identical and follows for all \textit{generalised} Gromov-Witten theories (see \cite{CPS}). Let $^{\circ}$ stand for any GW, FJRW or mixed theory under discussion. So far, each have a state space endowed with an inner product and invariants which can be compiled into  genus-$g$ generating functions $\mathcal{F}_g^{\circ} = \sum_{n\ge 0} \frac{1}{n!}\langle \mathbf{t}(\psi), \ldots, \mathbf{t}(\psi)\rangle.$ Let the \textit{total genus descendant potential} be $$\mathcal{D}^{\circ} :=e^{\sum_{g\ge0} \hbar^{g-1} \mathcal{F}_g^{\circ}}.$$

Define the symplectic vector space $\mathcal{V}^{\circ}$ to be $\mathcal{H}^{\circ}((z^{-1}))$ endowed with the symplectic form $$\Omega^{\circ} (f, g) = \mathrm{Res}_{z = 0} \langle f(-z), g(z) \rangle^{\circ},$$ defined via the induced inner product. We choose a polarisation $\mathcal{V}_+^{\circ} = H^{\circ}[z], \mathcal{V}_-^{\circ} = z^{-1}H^{\circ}[[z^{-1}]].$ This gives us Darboux coordinates $(\mathbf{p}, \mathbf{q})$ corresponding to $$\sum_{k\ge 0} \sum_{i \in I} q_{k}^i \phi_i z^k + \sum_{k\ge 0} \sum_{i \in I} p_{k, i} \phi^i (-z)^{-k-1}.$$

The dilaton shift is a slight adjustment fitting the $\mathbf{t}$-coordinates to our framework: \begin{equation*}\label{dilaton}q_1^0 = t_1^0 -1, q_k^i = t_k^i.\end{equation*}

We define the \textit{Givental Lagrangian cone} $$\mathcal{L}_{\circ} = \{\mathbf{p} = \mathrm{d}_{\mathbf{q}} \mathcal{F}_0^{\circ}\},$$ which can be shown to be Lagrangian with respect to $\Omega^{\circ}$ by some basic generalised Gromov-Witten theory \cite{CPS}, and is clearly a cone. Its elements may be written $$-\phi_0z + \sum_{\substack{k\ge 0\\ i \in I}} t_{k}^i \phi_i z^k + \sum_{\substack{a_1, \ldots, a_n, a \ge 0\\i_1, \ldots, i_n, i \in I}} \frac{t_{a_1}^{i_1} \ldots t_{a_n}^{i_n}}{n!(-z)^{a+1}} \langle \psi^{a_1} \phi_{i_1}, \ldots, \psi^{1_n}\phi_{i_n}, \psi^a\phi_i\rangle^{\circ}_{0, n+1} \phi^i.$$

By some further Gromov-Witten theory \cite{CPS} it can be shown that $$\mathcal{L} \cap T_f\mathcal{L} = z T_f\mathcal{L},$$
where $\mathcal{L}$ is ruled by the $zT_f\mathcal{L}$ over all $f$, and we have a filtration $$T_f\mathcal{L} \supset zT_f\mathcal{L} \supset z^2 T_f\mathcal{L} \supset \ldots$$

The image of a function $f(\mathbf{t})$ (with $\mathbf{t} \in \mathcal{H}^{\circ}$) corresponds to a `slice' of $\mathcal{L}$. In this way $\mathcal{L}^{\circ} \cap -\phi_0z \oplus \mathcal{H} \oplus \mathcal{V}_-$ corresponds to the J-function: 

$$ J^{\circ}(\mathbf{t}, z) = \phi_0z + \mathbf{t} + \sum_{n \ge 0}\sum_{a \ge0, i \in I} \frac{1}{n!z^{a+1}}\langle \mathbf{t}, \ldots, \mathbf{t},\phi_i \psi^a\rangle_{0, n+1}^{\circ} \phi^i.$$  
 
Givental's version of mirror symmetry relates the J-function to the \textit{I-function} $I(\mathbf{t}, z)$, which in the Calabi-Yau sense provides solutions to the Picard-Fuchs equations of the Calabi-Yau family. Mirror symmetry \textit{in the sense of Givental} insists there is a \textit{mirror map} $\tau:\mathcal{H} \to \mathcal{H}^{\vee}$ so that $$J^{\circ}(\tau(\mathbf{t}), -z) = I^{\circ}(\mathbf{t}, -z),$$ where on the Calabi-Yau side, $\mathcal{H}^{\vee}$ is  the state space of the mirror orbifold. That is, there is an invertible mirror map between the cohomologies of two mirror manifolds that not only rotates the Hodge diamond by 90$^{\circ}$ but also swaps their I- and J-functions. Based on the types of string theory to which they are associated, it is common to refer to the K\"ahler geometry relating to the J-function and  the (1, 1)-sector of the Hodge diamond relating to it as the `A-side', and the complex geometry relating to the I-function, Picard-Fuchs equations and the (2, 1)-sector as the `B side'. Thus a mirror pair have their A- and B-sides swapped.

In order to calculate these J-functions (on either side), however, we define an equivariant theory first and take the non-equivariant limit.  $\mathbb{C}^*$ acts on each $\mathcal{L}_k$ by multiplication on each fibre, and on each moduli space by performing this action pointwise. The equivariant theories are defined over the ground ring $R = H^*_{\mathbb{C}^*}(pt, \mathbb{C})[[s_0, s_1, \ldots]]$, and the invariants are given by cupping the integrand with (\textit{twisting by}) the multiplicative characteristic class $$\mathbb{c}(E) := \mathrm{exp}(\sum_k s_k \mathrm{ch}_k(E)),$$ where in our case $E$ will be the standard virtual class of either moduli space. 

If $s_d = 0$ for all $d$, this class is zero and we still have the \textit{untwisted theory}. For
$$s_d = \left\{
\begin{matrix}
-\mathrm{ln}\lambda, & d=0\\
\frac{(d-1)!}{\lambda^d}, & d>0
\end{matrix}\right.
$$ 
we will recover the dual of the equivariant Euler class of the virtual bundle in the non-equivariant limit.

\begin{proof}
For a line bundle $\mathcal{L}$, we have
\begin{align*}
& \mathrm{exp}(\sum_{d\ge 0}s_d\mathrm{ch}_d(-\mathcal{L}) = \mathrm{exp}(\mathrm{ln}(\lambda)\, \mathrm{ch}_0(\mathcal{L}) - \sum_{d>0} \frac{(d-1)!}{\lambda^d} \mathrm{ch}_d (\mathcal{L})\\
& =  \mathrm{exp}(\mathrm{ch}_0(\check{\mathcal{L}}) - \sum_{d>0} (-1)^{d-1} \frac{(d-1)!}{\lambda^d}\mathrm{ch}_d(\check{\mathcal{L}}))\\
& =  \lambda \mathrm{exp} (\mathrm{ln}(1+\frac{c_1(\check{\mathcal{L}})}{\lambda}))\\
& =  \lambda + c_1(\check{\mathcal{L}}).
\end{align*}
Taking $\lambda \to 0$ gives the desired result for $\mathcal{L}$; as we are working over $\mathbb{C}$, extending by the splitting principle will give us the same result for all vector bundles.
\end{proof}

We use the twisted invariants to define twisted generating functions $\mathcal{F}^{\mathrm{tw}}$, twisted potential $\mathcal{D}^{\mathrm{tw}}$,  a twisted Lagrangian cone $\mathcal{L}^{\mathrm{tw}}$ and a twisted J-function $J^{\mathrm{tw}}$. When appropriate we will denote their untwisted analogues $\mathcal{F}^{\mathrm{un}}, \mathcal{D}^{\mathrm{un}}, \mathcal{L}^{\mathrm{un}}, J^{\mathrm{tw}}$.


\subsection{The Landau-Ginzburg/Calabi-Yau Correspondence and Mirror Symmetry}

Beyond analogous definitions, mirror symmetry is a non-trivial duality relating the I-function and J-function. The LG/CY correspondence furthermore relates the state spaces and I-functions and J-functions of GW theory and FJRW theory:

\begin{center}
Calabi-Yau side
\end{center}
$$
\xymatrix{
\mathcal{H}^{GW}(\mathcal{Y}), J_{\mathrm{GW}}(\mathcal{Y}) \ar@{<->}[r] \ar@{<->}[d] & \mathcal{H}^{GW}(\check{\mathcal{Y}}), I_{\mathrm{GW}}(\mathcal{Y}) \ar@{<->}[d] \cong J_{\mathrm{GW}}(\check{\mathcal{Y}})\\
\mathcal{H}^{FJRW}(W, G), J_{\mathrm{GW}}(\mathcal{X}) \ar@{<->}[r]  & \mathcal{H}^{FJRW}(\check{(W, G)}), I_{\mathrm{FJRW}}(W, G) \cong J_{\mathrm{FJRW}}(\check{W}, \check{G}) }
$$
\begin{center}
Landau-Ginzburg side
\end{center}

We shall restrict to the narrow FJRW sectors of $(W, G)$, which correspond to the cohomology sectors induced by the ambient space of $\mathcal{Y}$ on the GW side, and then relate $I_{\mathrm{GW}}^{\mathrm{amb}}(\mathbf{t}, -z)$ and $I_{\mathrm{FJRW}}^{\mathrm{nar}}(\mathbf{t}, -z)$, which can both be found by \textit{twisting} the actual I-functions of the ambient space or its LG dual. On the Gromov-Witten B-side we take the I-function to provide solutions to the Picard-Fuchs equations of a family parametrised by $\psi$ around the point $\psi = 0$. Then the Landau-Ginzburg I-function is taken at $\psi = \infty$. The actual LG/CY correspondence is thus given as the composition of an analytic continuation from $\psi = 0$ to $\infty$ and a symplectic transformation between $\mathcal{V}^{\mathrm{GW}}$ and $\mathcal{V}^{\mathrm{FJRW}}$.
 
As a final note, the bottom arrow represents a neat formulation of mirror symmetry on the LG side due to Krawitz \cite{K} known as Bergland-Huebsch-Krawitz mirror symmetry. Where aspects of mirror symmetry are themselves difficult to prove on the Calabi-Yau side, the hope is thus that the FJRW theory will be easier to compute.

\section{The Gromov-Witten side}

\subsection{The State Space}

There are two elliptic curves of Fermat type $X^2 + P(Y, Z)$ with anti-symplectic involution given by $X \mapsto -X,$ tabulated below.
$$
\begin{tabular}{|l|l|}
\hline
$P(X, Y, Z)$ & Ambient space $\mathbb{P}(v_0, v_1, v_2)$ \\
\hline
$Y^4+Z^4$ & $\mathbb{P}(2, 1, 1)$\\
$Y^3+Z^6$ & $\mathbb{P}(3, 2, 1)$\\
\hline
\end{tabular}$$

We will assume that $E = \{X^2 + Y^4+Z^4=0\}$ throughout. The second curve produces complications, to be discussed later. 

For the sake of simplifaction, from here on $E$ will be the quartic curve $\{X^2+Y^4+Z^4 = 0\}$ in $\mathbb{P}(2,1,1)$,with the corresponding involution  $\sigma_E:X \mapsto -X$ unless otherwise specified. The choice of elliptic curve does not change the full state space, and we shall see that the changes to the narrow part of the state space are minor. Explicit equations and Nikulin involutions for K3 surfaces in weighted projective space given by polynomials of the form $x^2+P(x, y, z)$ are tabulated in \cite{GoLiYu} by invariants $(r, a)$, where $(N, N') = (1+\frac{r-a}{2}, 11-\frac{r+a}{2})$. Those of Fermat type are
\begin{center}
\addvbuffer[8pt, 16pt]{
\begin{tabular}{|l|l|ll|}
\hline
$P(x, y, z)$ & Ambient space $\mathbb{P}(w_0, w_1, w_2, w_3)$ & $N$ & $N'$\\
\hline
$y^6 + z^6 + w^6$ & $\mathbb{P}(3,1,1,1)$ & 1 & 10 \\
$y^5+z^5+w^{10}$ & $\mathbb{P}(5,2,2,1)$ & 2 & 6 \\
$y^3 + z^{10} + w^{15}$ & $\mathbb{P}(15, 10, 3, 2)$ & 4 & 4 \\
$y^3 + z^7 + w^{42}$ & $\mathbb{P}(21, 14, 6, 1)$ & 6 & 6 \\
$y^3 + z^9 + w^{18}$ & $\mathbb{P}(9, 6, 2, 1)$ & 3 & 7\\
$y^4+z^8 + w^8$ & $\mathbb{P}(4, 2, 1,1)$ & 1 & 9 \\
$y^4 + z^5 + w^{20}$ & $\mathbb{P}(10, 5, 4, 1)$ & 2 & 6 \\
$y^4 + z^ 6 + w^{12}$ & $\mathbb{P}(6, 3, 2, 1)$ & 1 & 7 \\
$y^3 + z^{12} + w^{12}$ & $\mathbb{P}(6, 4, 1, 1)$ & 2 & 10 \\
$y^3 + z^8 + w^{24}$ & $\mathbb{P}(12, 8, 3, 1)$ & 3 & 7 \\
\hline
\end{tabular}}
\end{center}

Note that in all of the above cases, only the first has all $w_i$ pairwise relatively prime, and in all other cases any common factor for any two $w_i, w_j$ divides only those two, and all such common factors are prime. Define $d = \mathrm{lcm}(w_0, w_1, w_2, w_3)$.

We consider the K3 surface $\{x^2 + y^6 +z^2 + w^2 = 0\} \subseteq \Proj(3, 1, 1, 1)$, with involution $\sigma_K: x \mapsto -x$. (We use upper case for the coordinates corresponding to the elliptic curve, and lower case for those corresponding to the K3 surface).

We shall also denote the ambient space $[(\Proj_{2, 1, 1} \times \Proj_{w_0, w_1, w_2, w_3})/\mathbb{Z}_2]$ by $\mathcal{X}$.

In this case, from \eqref{eq:coh}, we have $$H_{\mathrm{CR}}^*(\mathcal{Y}) =
 \begin{tabular}{llllllll}
    & & & 1 & & & \\
   & & 0 & & 0 & & \\
   & 0 &  & 6 &  & 0 & \\
   1 & & 60 & & 60 & & 0 \\
   & 0 &  & 6 &  & 0 & \\
   & & 0 & & 0 & & \\
   & & & 1 & & & \\
 \end{tabular}$$
We will be specifically considering the \textit{ambient classes}, that is those induced from classes from $Y$.

We write the ambient space $\X = (\Proj_{2, 1, 1} \times \Proj_{3, 1, 1, 1})/\mathbb{Z}_2$ as  $((\C^3\backslash\{0\}) \times (\C^4 \backslash \{0\}) \times \C^*)/(\C^*)^3$ where $(\C^*)^3$ acts by $$(\lambda_1, \lambda_2, \lambda_3): (X, Y, Z, x, y, z, w, \alpha) \mapsto (\lambda_1^2\lambda_3 X, \lambda_1 Y, \lambda_1 Z, \lambda_2^3{w_0}\lambda_3 x, \lambda_2^{w_1} y, \lambda_2^{w_2} z, \lambda_2^{w_3} w, \lambda_1^2 \alpha).$$ Thus if $\alpha$ is chosen to be 1, then $\lambda_3 = \pm 1$, and the action is that of $\sigma$.

We consider which $(\lambda_1, \lambda_2, \lambda_3)$ have fixpoints - each of these group elements corresponds to a non-empty component of the inertia stack. For example, for $\mathbb{P}(3, 1, 1, 1)$, we consider whether or not $Y, Z = 0$, or $y = z = w = 0$, and find 18 elements: $$(\pm 1, \sqrt[3]{1}, 1), (\pm i, \sqrt[3]{-1}, -1), (\pm i, 1, -1), (1, \sqrt[3]{-1}, -1), (1, 1, -1).$$

However, those which require $Y = Z = 0$ or $y = z= w = 0$ do not intersect $E \subseteq \Proj_{2, 1, 1}$ or $K \subseteq \Proj_{w_0, w_1, w_2, w_3}$, as this would require $X = 0$ or $x = 0$. Only two group elements remain: $(1, 1, 1)$, corresponding to the identity, and $(1, 1, -1)$, corresponding to the involution itself. 

Thus the untwisted and $\sigma$-sectors in $H_{\mathbf{CR}}^{\mathrm{amb}}(\mathcal{Y})$ are generated by the Poincar\'e-dual classes to
$$
\begin{matrix}
\{pt\} \\
E, \, H_K,\, pt_{\sigma}\\
E\times H_K,\, K, \, \Sigma\\
\mathcal{Y}
\end{matrix}
$$
where $H_K$ is the intersection of a hyperplane in $\mathbb{P}_{w_0, w_1, w_2, w_3}$ with $K$, and $\Sigma$ is the fixpoint set of $\sigma$, isomorphic to four copies of $\Sigma_{\frac{\mathrm{lcm}(w_1, w_2, w_3)-1)}{2}\frac{\mathrm{lcm}(w_1, w_2, w_3)-2)}{2}}$, by the degree-genus formula. We will find it useful to write the classes in a slightly different way: Let $D_i = \{ X_i = 0\}$ be the $i$-th coordinate $X_i$, $i = 1, \ldots, 8$. We will express these divisors in terms of the toric divisors $D_E = H_E \times K$ and $D_K = E \times H_K$. We have $D_E = K, D_K = E \times H_K, D_K\cup D_K = E, D_E \cup D_K = H_K, D_E \cup D_K \cup D_K = \{ pt \}, D_E^2 = D_K^2 = 0$. We write $\mathbf{1}_{\sigma}$ for the identity class on the $\sigma$-sector, and the point class on that sector is given by $D_K \cup \mathbf{1}_{\sigma}$ after scaling (which can be seen from the intersection form on a resolution of the orbifold, and B\'ezout's theorem). However, $D_E \cup \mathbf{1}_{\sigma}$ can be shown in both of these ways to be zero, so that $D_E, D_K, \mathbf{1}_{\sigma}$ are multiplicative generators of a basis of the cohomology. For the untwisted sector it is the same as the usual intersection product, and for the twisted sector we have $\mathbf{1}_{\sigma} D_E  = 0$, so that we may take the (one-dimensional) twisted part of degree 4 to be generated by $\mathbf{1}_{\sigma} D_K.$ For $\mathbb{P}(3, 1, 1, 1)$, the story ends here.

For the other cases, when $w_i, w_j$ share a common prime factor $p$, there is also a fixed subspace $\mathbb{P}(\frac{w_i}{p}, \frac{w_j}{p})$, one copy each corresponding to elements of a subgroup $\mathbb{Z}_p$. There are two cases: first, one of $i, j$, WLOG $i = 0$, in which case the intersection with $K$ is $\frac{d}{w_j}$ $\mathcal{B}\mathbb{Z}_p$, but does not intersect $\mathrm{Fix}(\sigma)$; we thus have $(p-1)$ dimensions of ambient classes of degree $2$ and, multiplying by $H_E \times K$, another $(p-1)$ dimensions of ambient classes of degree $4$. Otherwise, neither of $i, j = 0$, in which case we have that $\mathbb{P}(\frac{w_i}{p}, \frac{w_j}{p}) \cap K \subseteq \mathrm{Fix}(\sigma)$, contributing the non-trivial degree-2 sectors from $\mathcal{B}\mathbb{Z}_p$, which, cupping with $D_E = H_E \times K$, contributes just as many of degree 4. Furthermore, multiples of $H_E \times K$ are the only classes that cup with these classes non-trivially.


For $p_{i, j}>1$ for distinct $i, j$, we must include all $(a, b, c) \in \mathbb{Z} \times  \frac{1}{p_{i, j}}\mathbb{Z} \times \mathbb{Z}$. Under the valuation map such $(a, b, c)$ correspond to the sectors generated by $\mathbf{1}_{{g_{i, j}}^r}$ where $b \equiv_1 \frac{r}{p_{i, j}}$. These all have degree 2, and we have as many classes again given by $D_E\mathbf{1}_{g_{i, j}^r}$.

For $p_{0, j} >1$ for $j>0$, $\mathrm{Fix}((\lambda_1, \lambda_2, -1))$ for non-trivial $\lambda_2$ may also be non-trivial: if $w_i = 2p \not \vert w_0$, we have $(\zeta_{2p}^{k})^{2p} =1$ for all $k$ and  $(-1)(\zeta_{2p}^k)^{p\frac{w_0}{p}} = 1$ for \textit{odd} $k$. This means that both $x$ and the coordinate corresponding to $w_i$ are both fixed, and so by B\'ezout's theorem we have a non-empty zero-dimensional intersection of this one-dimensional weighted projective subspace with $K$, adding a new point twisted sector for each odd $k$ from $1, \ldots, 2p$: there are $p$ of these, denoted $\mathbf{1}_{g'_{\sigma \frac{r}{2p_{0, i}}}}$.

The final possible situation giving non-empty twisted sectors arises when $w_i, w_j$ do not divide $w_0$ but share a common factor for $i \ne 0$ (which must be 2): this only occurs for $(5, 2, 2, 1)$, $(21, 14, 6, 1)$ and $(15, 10, 3, 2)$. We have $\mathrm{Fix}(1, -1, -1) = \{w = 0\}$, which also has non-trivial intersection with $K$ and gives $(2-1) = 1$ twisted copy of $\mathbb{P}(\frac{w_i}{2}, \frac{w_j}{2})$. This contributes just as many new sectors: $(p-1) + (q-1)$ point sectors of the form $\mathbf{1}_{{\sigma g'}_i^{\frac{r}{2p_{0, i}}}}, \mathbf{1}_{{\sigma g'}_{\frac{r}{2p_{0, i}}}}$, and one 1-dimensional sector $\mathbf{1}_{{\sigma g'}_{\frac{r}{2p_{0, i}}}} = \mathbf{1}_{{\sigma g'}_{\frac{r}{2p_{0, j}}}}$, which we shall denote $\mathbf{1}_{\tilde{g}}$, and the generic point in the same sector equal to $D_K\mathbf{1}_{\tilde{g}}$.  Half of these classes, including $\mathbf{1}_{\tilde{g}}$, have degree 2, and the rest have degree 4.

Let $$S_j^{\alpha} = \{\frac{r}{2p_{0, j}}: 2\not\vert r\, r \le 2p_{0,j} - 1, \, \mathrm{deg}_{\mathrm{CR}}(-1, \zeta_{2p}^r, -1) = \alpha\},$$ and let $$S(\alpha) = \bigcup_{j: \,w_j \not\vert w_0} S_j^{\alpha}.$$


Under the CR-pairing, $(\mathbf{1}_g)^{-1} = D_E\mathbf{1}_{g^{-1}}, \; (\mathbf{1}_{\sigma g})^{-1} =  \mathbf{1}_{\sigma g^{-1}}$ for $g \ne \tilde{g}$, and $(\mathbf{1}_{\sigma \tilde{g}})^{-1} = D_K\mathbf{1}_{\sigma \tilde{g}}$.

Then the ambient part of $\mathcal{H}^{\mathrm{amb}}(\mathcal{Y})$ is in general given by the following sets of generators, ordered by degree:
\begin{equation*}
\label{eq:cramb}
\begin{matrix}
\{\mathbf{1}\} \\
\{D_E, D_K, \mathbf{1}_{\sigma}\} \cup \bigcup_{\substack{p_{i, j} >1\\ 1\le r\le p_{i, j}-1}} \{\mathbf{1}_{g_{i, j}^r}\} \cup \bigcup_{\frac{r}{p} \in S(2)} \{\mathbf{1}_{g_{\frac{r}{2p}}}\} \\
\{D_ED_K, D_K^2, D_K\mathbf{1}_{\sigma}\} \cup\bigcup_{\substack{p_{i, j} >1\\ 1\le r\le p_{i, j}-1}} \{D_E\mathbf{1}_{g_{i, j}^r}\} \cup \bigcup_{\frac{r}{p} \in S(4)} \mathbf{1}_{\sigma g'_{\frac{r}{p}}}   [\cup  \{ D_K\mathbf{1}_{\sigma \tilde{g}}\}] \\
\{u\}
\end{matrix}
\end{equation*}
where the sector $D_K\mathbf{1}_{\sigma \tilde{g}}$ is included when $w_1 = \frac{d}{3}$. 

For $E = \{X^2+Y^3+Z^6 = 0\},$ similar arguments hold as for $K$, and there is an extra $\sigma_E$-twisted sector, doubling the number of $\sigma$-twisted sectors. 

\subsection{Enumerative Geometry}

We treat $\mathcal{X}$ as a Deligne-Mumford toric stacks in the sense of \cite{BCS}.

$\mathcal{X}$ corresponds to the stacky fan $(N, \mathbf{\Sigma}, \rho)$ where $$N = \mathbb{Z}^5 + \langle (\frac{1}{2}, \frac{1}{2}, 0, 0, 0)\rangle + \langle (0, 0, \frac{w_1}{w_0}, \frac{w_2}{w_0}, \frac{w_3}{w_0})\rangle + \langle (\frac{1}{4}, \frac{1}{4}, \frac{w_1}{2w_0}, \frac{w_2}{2w_0}, \frac{w_3}{2w_0})\rangle \subset \mathbb{Z}^5 \otimes \mathbb{Q},$$ with $\rho:\mathbb{Z}^8 \to N$ given by $$\begin{pmatrix}
-\frac{1}{2} & 1 & 0 & 0 & 0 & 0 & 0 & \frac{1}{4}\\
-\frac{1}{2} & 0 & 1 & 0 & 0 & 0 & 0 & \frac{1}{4}\\
0 & 0 & 0 & -\frac{w_1}{w_0} & 1 & 0 & 0 & \frac{w_1}{2w_0}\\
0 & 0 & 0 & -\frac{w_2}{w_0} & 0 & 1 & 0 & \frac{w_2}{2w_0} \\
0 & 0 & 0 & -\frac{w_3}{w_0} & 0 & 0 & 1 & \frac{w_3}{2w_0}
\end{pmatrix}$$
where the columns $\rho_1, \rho_2, \ldots, \rho_8$ and the maximal cones of $\Sigma$ are those generated by all but one of $\{\rho_1, \rho_2, \rho_3\}$, and all but one of $\{\rho_4, \rho_5, \rho_6, \rho_7\}$ (and excludes $\rho_8$, which appears only due to the factor $\mathbb{C}^*$ which is solely included to represent the $\mathbb{Z}_2$-action as toric).

The \textit{box} of $\mathbf{\Sigma}$ is $$\mathrm{Box}(\mathcal{X}) := \{\sum_{i: \rho_i \in \kappa} a_i \rho_i \text{ for some } \kappa \in \mathbf{\Sigma}, \, 0 \le a_i < 1  \},$$
and is in one-to-one correspondence with the set of sectors of the inertial manifold of $\mathcal{X}$. We are interested in the subset of those which induce non-empty sectors in $\mathcal{Y}$, which we shall denote $\mathrm{Box}(\mathcal{Y})$. We will also denote the elements of the box corresponding to $\mathbf{1}_{g_{i, j}^r}$ and $\mathbf{1}_{g_{\sigma\frac{r}{2p}}}$ by $\mathbf{b}_{g_{i, j}^r}$ and  $\mathbf{b}_{\sigma g'_{\frac{r}{2p}}}$ respectively. 

The twisted sectors coming from $\mathbb{P}(w_0, w_1, w_2, w_3)$ are given by the following. We note that from the construction of the toric fan of weighted projective space, $w_0\rho_4 + w_1\rho_5 + w_2\rho_6+w_3\rho_7 = 0$, so that for each $p_{i, j} := \mathrm{gcd}(w_i, w_j) >1$, $$\frac{rw_i}{p_{i, j}}\rho_{4+i} + \frac{rw_j}{p_{i, j}}\rho_{4+j} + \sum_{k \ne i, j} \lceil \frac{rw_k}{p_{i, j}}\rceil\rho_{4+k} = \sum_{k \ne i, j} \langle - \frac{rw_k}{p_{i, j}} \rangle\rho_{4+k}$$ is an element of the box, since it is given as a member of $N$ on the left and its $\rho_i$-coordinates are strictly bounded by 1 in the expression on the right.


For $p_{0, i}>1$, we refer to \cite{BCS} to note that when we have $c_1, c_2 \in \kappa \in \mathbf{\Sigma}$, then $\mathbf{1}_{c_1}\mathbf{1}_{c_2} = \mathbf{1}_{c_1 + c_2}$. $\mathbf{1}_{\sigma}$ is represented by $\rho_8 = \frac{1}{2}(\rho_1+\rho_4)$, and so the sectors generated by elements of the form  $\mathbf{1}_{\sigma g'_{\frac{r}{2p}}}$ are given by $\frac{1}{2}(\rho_1+\rho_4) + [g'_{\frac{r}{2p}}]$, the latter being the fan representation given above; we consider those inducing non-zero sectors in $\mathrm{Box}(\mathcal{Y})$ as before.

We extend $\mathbf{\Sigma}$ by the sectors represented by $s_1, \ldots, s_l, s_{l+1}, \ldots, s_m$, where the sectors $s_1, \ldots, s_l$ represent the $\mathbf{1}_{g_{i, j}^r}$, and $s_{l+1}, \ldots, s_{m}$ correspond to the sectors given by the elements of $S(2)$.  The representations found above give coefficients $s_{i, j}$, $0 \le s_{i, j}<1$ such that $\sum_{i: \rho_i \in \sigma(j)} s_{i,j}\rho_i = s_j$, where $\sigma(j)$ is the cone containing $s_j$ (we set all other coefficients for each $j$ to be zero). Furthermore, such $s_{i, j}$ are unique. Note that $s_{1, j} = s_{4, j} = 0$ for $1 \le j \le l$ and $s_{1, j} = s_{4, j} = -\frac{1}{2}$ for $l+1 \le j \le m$.

Let $1 \le m \le M = := \vert \mathrm{Box}(\mathcal{Y}) \vert$ and $S = \{1, 2, \ldots, m\}$. Then choose an injective function $S \mapsto \mathrm{Box}(\mathcal{Y})$, given by $i \mapsto s_i$, $1 \le i \le m$. We define the \textit{$S$-extended stacky fan} $\mathbf{\Sigma}^S$ by extending $\rho:\mathbb{Z}^8 \to N$ to $\rho^S: \mathbf{Z}^{8+m} \to N$ by setting $\rho_{8+i} = S(i)$. It represents the same orbifold as $\Sigma$ \cite{CCIT1}. 
Following \cite{BCS}, we  tautologically lift $\rho$ to $R = I$ with projection $\rho$:
$$\xymatrix{
& \mathbb{Z}^8 \ar[d]^{\pi = \rho} \\
\mathbb{Z}^8 \ar[ur]^{R = I} \ar[r]^{\rho} & N
}$$
with kernel $\xymatrix{\mathbb{L} \ar@{^{(}->}[r]^Q & \mathbb{Z}^8}$, where $$Q = \begin{pmatrix}
2 & 0 & 1\\
1 & 0 & 0\\
1 & 0 & 0\\
0 & w_0 & 1\\
0 & w_1 & 0\\
0 & w_2 & 0\\
0 & w_3 & 0\\
0 & 0 & 2
\end{pmatrix}.$$
Taking the Gale dual we find $\check{\mathbb{L}} \cong \mathbb{Z}^3$ and  $$\check{\rho} : \mathbb{Z}^{3+8}/\mathrm{ker}\,\mathrm{Im}\, ([RQ]^*) \to \mathrm{ker}\,\mathrm{Im}\, ([RQ]^*)$$ is given by the same matrix, so that the group $$G = \mathrm{Hom}\,(\check{\mathbb{L}},\mathbb{C}^*) =  (\mathbb{C}^*)^3$$ acts by $$\alpha = \begin{pmatrix}
2 & 1 & 1 & 0 & 0 & 0 & 0 & 0\\
0 & 0 & 0 & w_0 & w_1 & w_2 & w_3 & 0\\
1 & 0 & 0 & 1 & 0 & 0 & 0 & 2
\end{pmatrix},$$ the desired weights. The columns here correspond to $D_i$ in $\mathbb{L}$ \cite{CCIT1}. From this the Chen-Ruan cohomology may also be computed from this construction via \cite{BCS}: this agrees with our outline in the previous subsection.

From \cite{BCS} we can relate the Chen-Ruan cohomology to the stacky fan construction. The twisted sector generated by $\mathbf{1}_{\sigma}$ is given by $\rho_8 = \frac{w_1}{2w_0}\rho_1 + \frac{w_2}{2w_0}\rho_2 + \frac{w_3}{2w_0}\rho_3$.

The corresponding kernel $\mathbb{L}^S$ under $S$-extension fits into the exact sequence $$0 \to \mathbb{L} \to \mathbb{L}^S \to \mathbb{Z}^m.$$ Considering these inside their tensor products with $\mathbb{Q}$, this sequence splits via the map $e_j \mapsto e_{8+j} - \sum_i s_{j,i} e_{i}$. Therefore $\mathbb{L}^S \otimes \mathbb{Q}$ is the image of
\begin{align*}
(a, b, c, k_1, \ldots, k_m) \mapsto & a(2\rho_1 + \rho_2 + \rho_3) + b(\sum_{i=1}^4 w_i \rho_{4+i} + c (\rho_1 + \rho_4 + 2\rho_8) + \sum_{j=1}^m k_j(s_j - \sum_{i=0}^8 s_{j, i}\rho_{i}).
 \end{align*}

The Mori cone $\mathrm{NE}(\mathcal{X}) \subset \mathbb{L}$ is $\sum_{\kappa \in \mathbf{\Sigma}} \mathbb{R}_{\ge 0} \check{C_{\tau}}$, where $\check{C_{\kappa}}$ is the dual cone of $C_{\kappa} = \sum_{i \notin \kappa} \mathbb{R}_{\ge 0} \rho_i$. Sorting through the cases, we get $$\mathrm{NE}(\mathcal{X}) = \{(a, b, c) \in \mathbb{L} \,\vert\, c \ge 0, 2a+c \ge 0, w_0b+c \ge 0\}.$$ This is generated by rays
\begin{align*}
& \{ a + \frac{c}{2} \ge 0, b + \frac{c}{w_0} = 0, c = 0 \} = \langle (1, 0, 0)\rangle,\\
& \{a + \frac{c}{2} = 0, b + \frac{c}{w_0} \ge 0, c = 0 \} = \langle (0, 1, 0)\rangle,\\
& \{a + \frac{c}{2} = 0, b + \frac{c}{w_0} = 0, c \ge 0 \} = \langle (-\frac{1}{2}, -\frac{1}{w_0}, 1)\rangle.\\
\end{align*}

The $S$-extended Mori cone is given by $NE^S(\mathcal{X}) := NE(\mathcal{X}) \times \mathbf{R}_{\ge0}^m$.

Using the notation of \cite{CCIT1}, we define $\Lambda_{\kappa}^S = \{ \sum_i^{8+m} \lambda_i e_i \in \mathbb{L} \otimes \mathbb{Q} \, \vert \, i \not\in \kappa \implies \lambda_i \in \mathbb{Z}; \lambda_i \in \mathbb{Z}, i > 8 \}$, and $\Lambda^S = \bigcup_{\kappa \in \Sigma} \Lambda_{\kappa}^S$. That is, those elements of $\mathbf{L}^S$ for which the following must hold: $2c \in \mathbb{Z}$, one of $2a+c - \sum_j k_j s_{j, 1} \in \mathbb{Z}$ or $a \in \mathbb{Z}$, and one of $w_0b + c - \sum_j k_j s_{j, 4} \in \mathbb{Z}$ or $w_ib - \sum_j k_j s_{j, i}\in \mathbb{Z}$, for some $i = 5, 6, 7$. The \textit{valuation map} is then defined to be $$v^S: \Lambda \to \mathrm{Box}(\mathbf{\Sigma})$$ by $$\lambda \mapsto \sum_{i=1}^8 \lceil \lambda_i \rceil \rho_i = \sum_{i=1}^8 \langle -\lambda_i \rangle \rho_i,$$ the latter inequality holding by equations defining $\mathbb{L}$. We set $\Lambda_{\mathbf{b}} = (v^S)^{-1}(\mathbf{b})$, $\Lambda E (\mathcal{X}) =  \Lambda \cap \mathrm{NE}(\mathcal{X})$ and $\Lambda E_{\mathbf{b}} (\mathcal{X}) = \Lambda_{\mathbf{b}} \cap \mathrm{NE}(\mathcal{X}).$

The untwisted I-function is given in \cite{CCIT1} by
$$I_{\mathrm{GW}}(\mathbf{t}, z) = ze^{\sum_i D_it_i/z}\sum_{\mathbf{b} \in \mathrm{Box}(\mathbf{\Sigma})} \sum_{\lambda \in \Lambda E_{\mathbf{b}}} q^{\lambda} \prod_i^{8+m} e^{(D_i \cdot \lambda) t_i} \frac{\prod_{\langle d \rangle = \langle \lambda_i \rangle, \, d \le 0} (D_i+dz)}{\prod_{\langle d \rangle = \langle \lambda_i \rangle, d \le \lambda_i} (D_i+dz)}\mathbf{1}_b,$$
where $q^{\lambda}$ are the Novikov variables recording the class $\lambda$, and the $D_i$ are the divisor classes given above for $i\le 8$, and zero for $i>8$.

Denote each term corresponding to $(\mathbf{b}, \lambda)$ by $I_{\mathbf{b}, \lambda}(t, z)$. $\mathcal{Y}$ is the generic zero section of the bundle $\mathcal{E} = \mathcal{O}(4,0) \oplus \mathcal{O}(0,2w_0)$. The \textit{twisted} I-function of this bundle is then given in \cite{CCIT2} by $$\sum_{\mathbf{b}, \lambda} I_{\mathbf{b}, \lambda}(t, z) M_{\mathbf{b}, \lambda}(t, z),$$ via the modification factors $$M_{\mathbf{b}, \lambda}(t, z) = \prod_{k=1}^{2\cdot(2a+c+\frac{1}{2}\sum_j={l+1}^m k_j)}(4D_E + kz)\prod_{l=1}^{2\cdot(w_0b+c+\frac{1}{2}\sum_{j=l+1}^m k_j)} (2w_0D_K+lz).$$

Note that considering the cases when $\lambda_i \in \mathbb{Z}$ and $\lambda_i \not\in \mathbb{Z}$ separately, we may write $$\frac{\prod_{\langle d \rangle =  \langle \lambda_i \rangle, \, d \le 0} (D_i+dz)}{\prod_{\langle d \rangle = \langle \lambda_i \rangle, d \le \lambda_i} (D_i+dz)}$$  as $$ z^{-\lceil \lambda_i \rceil} \frac{\Gamma(D_i/z+ \langle \langle \lambda_i \rangle\rangle)}{\Gamma(D_i/z+\lambda_i +1)},$$ where $\langle \langle \lambda \rangle \rangle := 1- \langle 1- \lambda \rangle$. We rewrite the modification factor similarly. For $i=8$, this gives a factor of $\frac{1}{(2c)!}$, and for $i = 8+j$, this gives a factor of $\frac{1}{k_j!}$. This re-expression will allow us to extend the function analytically.

We now have all the ingredients to write the I-function in our case explicitly. Collect all the factors constant in $m, n, c, \mathbf{k}$ as $K_{\mathbf{b}}, L_{\mathbf{b}}$ for each $\mathbf{b} \in \mathrm{Box}(\mathcal{Y})$:
\begin{align*}
& K_0 = K_{{g}_{i, j}^r} = \frac{\Gamma(2D_E/z+1)\Gamma(D_E/z+1)^2}{\Gamma(4D_E/z + 1)}\\
& K_{\sigma} = K_{\sigma{g'}_{\frac{r}{2p}}} = z^{-1/2}\frac{\Gamma(2D_E/z+\frac{1}{2})\Gamma(D_E/z+1)^2}{\Gamma(4D_E/z + 1)}\\
&L_0 = \frac{\Gamma(w_0D_K/z+1)\prod_{\nu=1}^3\Gamma(w_{\nu} D_K/z+1)}{\Gamma(2w_0D_K/z + 1)},\\
&L_{\sigma} = z^{-1/2}\frac{\Gamma(w_0D_K/z+\frac{1}{2})\prod_{\nu=1}^3\Gamma(w_{\nu} D_K/z+1)}{\Gamma(2w_0D_K/z + 1)},\\
&L_{g_{i, j}^r} = z^{\sum_{\nu} ((\mathbf{b}_{g_{i, j}^r})_{4+\nu} - \lceil (\mathbf{b}_{g_{i, j}^r})_{4+\nu} \rceil)}\frac{\prod_{i=0}^3\Gamma(w_iD_K/z +(\mathbf{b}_{g_{p_{i, j}^r}})_{4+\nu} +1)}{\Gamma(2w_0D_K/z+1)},\\
&L_{\sigma{g'}_{\frac{r}{2p}}} = z^{\sum_{\nu} ((\mathbf{b}_{g'_{\frac{p}{2r}}})_{4+\nu} - \lceil ((\mathbf{b}_{g'_{\frac{p}{2r}}})_{4+\nu} \rceil)}\frac{\prod_{\nu=0}^3\Gamma(w_{\nu}D_K/z + (\mathbf{b}_{\sigma g'_{\frac{r}{2p}}})_{4+\nu})}{\Gamma(2w_0D_K/z+1)}.
\end{align*}
Then collect the factors depending on $m, n, c, \mathbf{k}$ as $F(a, c, \mathbf{k}), G(b, c, \mathbf{k})$:
\begin{align*}
 F(a, c, \mathbf{k}) & =  \frac{\Gamma(4D_E/z + 4a+2c-2\sum_{\mu=1}^mk_{\mu}s_{\mu, 1} +1)}{\Gamma(2D_E/z+ 2a+c-\sum_{\mu=1}^m k_\mu s_{\mu, 1} +1)\Gamma(D_E/z+a+1)^2(2c+1)!},\\
G(b, c, \mathbf{k}) &= \frac{\Gamma(2w_0D_K/z +2w_0b+2c-2\sum_{\mu=1}^mk_{\mu}s_{\mu, 4}+ 1)}{\Gamma(w_0D_K/z+w_0b+c-\sum_{\mu=1}^mk_\mu s_{\mu, 4}+1)}\\
&\times\frac{1}{\prod_{\nu=1}^3\Gamma(w_{\nu} D_K/z+w_{\nu}b-\sum_{\mu=1}^mk_{\mu} s_{\mu, 4}+1)\prod_{\mu=1}^mk_{\mu}!}.
\end{align*}
Then we have
\begin{align*}
& I_{\mathrm{GW}} (\mathcal{Y}) = ze^{(2D_Et_1 + D_E(t_2+t_3) + \sum_{i=1}^4 w_iD_Kt_{4+i})/z} \times  \nonumber \\
& \sum_{\mathbf{b} \in \mathrm{Box}(\mathcal{Y})} K_{\mathbf{b}} L_{\mathbf{b}} \sum_{\substack{c \in \frac{1}{2}\mathbb{N}_0}}\sum_{\substack{(a, b, k_1, \ldots, k_m) \in \mathbb{N}_0^m\\ a \ge -c/2, b \ge -c/w_0\\v^{S}(a, b, c, k_1, \ldots, k_m) = \mathbf{b}}} (q_1^a q_2^b q_3^c \prod_{j=1}^mx_j^{k_j}) \\ 
& e^{a(2t_1 + t_2 + t_3)+b(\sum_{i=0}^3 w_it_i)+c(t_1 + t_4 + 2t_8)} F_\mathbf{b}(a, c, \mathbf{k}) G_{\mathbf{b}}(b, c, \mathbf{k}) \mathbf{1}_\mathbf{b}.
\end{align*}
Note that all other powers of $z$ cancel, and that this is a direct consequence of the Calabi-Yau condition that the charges sum to 1 in each factor space (the $(4a+2c) + (6b+2c)$ coming from the vector bundle precisely matches the $a + a + 2a + c + b + b + b + 3b + c + 2c$ coming from the ambient space). Note also that setting all the divisors to have degree 1 and all variables other than $z$ to have degree $0$, our function is homogeneous, and only the first multiple sum contributes the lowest powers of $z$. Upon expanding in $z$ we find our I-function to be of the form $$z + (2t_1+t_2+t_3)D_E + (3t_4+t_5+t_6+t_7)D_K + \sum_{\mu=1}^m L_\mu x_{\mu} \mathbf{1}_{\mu} + O(z^{-1});$$ thus the condition $S\sharp$ as defined in \cite{CCIT1} holds, so that the mirror theorem found there applies, as follows.
\begin{prop}
$$J_{\mathrm{GW}}^{\mathcal{Y}}((2t_1+t_2+t_3)D_E + (3t_4+t_5+t_6+t_7)D_K + \sum_{\mu=1}^m L_\mu x_{\mu} \mathbf{1}_{\mu}, z) = I_{\mathrm{GW}}^{\mathcal{Y}}(\mathbf{t}, x_1, \ldots, x_m, z).$$
\end{prop}

It will be simpler to demonstrate the correspondence in terms of these I-functions defined above, but in this paper we do not address the Picard-Fuchs equations, so our nomenclature for the I-functions originates solely by analogy from its role in the above mirror theorem; our computations are all in terms of the A side.

For clarity, in our first example $K = \{x^2 + y^6 + z^ 6 + w^ 6 = 0\} \subseteq \mathbb{P}(3, 1, 1, 1)$, the I-function reduces to:
\begin{align*} 
& I_{\mathrm{GW}}^{\mathcal{Y}}(\mathbf{t}, z) =  ze^{(2D_Et_1 + D_E(t_2+t_3) + 3D_Kt_4 + D_K(t_5+t_6+t_7))/z} \times  \nonumber \\
& \bigg(\sum_{c \in \mathbb{N}_0} \sum_{\substack{a, b \in \mathbb{Z}\\ a\ge -c/2\\b \ge -c/3}} q_1^a q_2^b q_3^c e^{(2a+c)t_1 + a(t_2+t_3) + (3b+c)t_4 + b(t_5+t_6+t_7) + 2ct_8} \times \nonumber \\
& \frac{\Gamma(2D_E/z+1) \Gamma(D_E/z+1)^2 \Gamma(3D_K/z+1)\Gamma(D_K/z+1)^3}{\Gamma(2D_E/z+2a+c+1)\Gamma(D_E/z+a+1)^2\Gamma(3D_K/z+3b+c+1)\Gamma(D_K/z+b+1)^3}\times \nonumber \\
& \frac{\Gamma(4D_E/z+4a+2c+1)\Gamma(6D_K/z + 6b+2c + 1)}{\Gamma(2c+1)\Gamma(4D_E/z+1)\Gamma(6D_K/z+1)} \nonumber \\
+& z^{-1}\sum_{c \in \frac{1}{2}\mathbb{N}_0\backslash \mathbb{N}_0} \sum_{\substack{a, b \in \mathbb{Z}\\ a\ge -c/2\\b \ge -c/3}} q_1^a q_2^b q_3^c e^{(2a+c)t_1 + a(t_2+t_3) + (3b+c)t_4 + b(t_5+t_6+t_7) + 2ct_8} \times \nonumber \\
& \frac{\Gamma(2D_E/z+\frac{1}{2}) \Gamma(D_E/z+1)^2 \Gamma(3D_K/z+\frac{1}{2})\Gamma(D_K/z+1)^3 }{\Gamma(2D_E/z+2a+c+1)\Gamma(D_E/z+a+1)^2\Gamma(3D_K/z+3b+c+1)\Gamma(D_K/z+b+1)^3})\times \nonumber \\
 & \frac{\Gamma(4D_E/z+4a+2c+1)\Gamma(6D_K/z + 6b+2c + 1)}{\Gamma(2c+1)\Gamma(4D_E/z+1)\Gamma(6D_K/z+1)}\mathbf{1}_{\sigma}\bigg).
\end{align*}

Finally, for $E = \{X^2+Y^3+Z^6 = 0\}$, we have double the number of $\sigma$-twisted sectors, and the I-function may be found in precisely the same way, but we shall omit explicit expressions for simplicity.

\section{The Pure Landau-Ginzburg Side}

\subsection{The FJRW State Space}

In this subsection we compute the narrow FJRW state spaces for $W = X^2+Y^4+Z^4+x^2+P(y, z, w)$, where $P(y, z, w)$ has degree $d$ and weights $w_0, w_1, w_2, w_3$ and $G = \langle J_1, J_2, \sigma\rangle$. We let $q_k$ be the associated charges of $W$. (Note that $w_i = \bar{w}_{4+i}$ in our notation from \S 1). First, it is instructive to compute the full state space for an example. Below we find the full state space for $P(y, z, w) =  y^6+z^6+w^6$. 
\newline


$$\begin{tabular}{|l|lll|ll|}
\hline
 & $g$ & & & $\sigma g$ & \\
\hline
$g$ & $N_g$ & $\mathrm{deg}_W(g)$  & Case & $N_{\sigma g}$ & Case\\
\hline
 $e$ & 7 & 3 & I & 5 & I*\\
 $J_1$ & 4 & 2 & II & 4 & II*\\
 $J_1^2$ & 5 & 3 & III & 3 & III*\\
 $J_1^3$ & 4 & 4 & II & 4 & II*\\
 \hline
 $J_2$ & 3 & 1 & IV  & 3 & IV*\\
 $J_1J_2$ & 0 & 0 & V & 2 & V*\\
 $J_1^2J_2$ & 1 & 1 & IV & 1 & VI*\\
 $J_1^3J_2$ & 0 & 2 & V & 2 & V*\\
 \hline
 $J_2^2$ & 4 & 2 & VII & 2 & VII*\\
 $J_1J_2^2$ & 1 & 1 & VI* & 1 & VI\\
 $J_1^2J_2^2$ & 2 & 2 & V* & 0 & V\\
 $J_1^3J_2^2$ & 1 & 3 & VI* & 1 & VI\\
 \hline
 $J_2^3$ & 3 & 3 & IV & 3 & IV*\\
 $J_1J_2^3$ & 0 & 2 & V & 2 & V*\\
 $J_1^2J_2^3$ & 1 & 3 & VI & 1 & VI*\\
$J_1^3J_2^3$ & 0 & 4 & V & 2 & V*\\
\hline
$J_2^4$ & 4 & 4 & VII & 2 & VII*\\
$J_1J_2^4$ & 1 & 3 & VI* & 1 & VI\\
$J_1^2J_2^4$ & 2 & 4 & V* & 0 & V\\
$J_1^3J_2^4$ & 1 & 5 & VI* & 1 & VI\\
\hline
$J_2^5$ & 3 & 5 & IV & 3 & IV*\\
$J_1J_2^5$ & 0 & 4 & V & 2 & V*\\
$J_1^2J_2^5$ & 1 & 5 & VI & 1 & VI*\\
$J_1^3J_2^5$ & 0 & 6 & V & 2 & V*\\
\hline
\end{tabular}$$

\begin{itemize}
\item \textbf{Case I.} $W\vert_g = X^2 + Y^4 + Z^4 + x^2 + y^6 + z^6 + w^6$. For $e$, the fixpoint set is the whole space, and we have $$\mathcal{H}_{\mathrm{FJRW}}^{\mathrm{I}} = \mathcal{Q}_{W_{\mathrm{I}}}\cdot \mathrm{d} X \wedge \mathrm{d} Y\wedge\mathrm{d} Z\wedge\mathrm{d} x\wedge\mathrm{d} y\wedge\mathrm{d} z\wedge\mathrm{d} w = \mathbb{C}[X, Y, Z, x, y, z, w]/\langle X, Y^3, Z^3, x, y^5, z^5, w^5\rangle.$$ This has generators $$Y^bZ^cy^ez^fw^g\mathrm{d}Y\wedge\mathrm{d}Z\wedge\mathrm{d}y\wedge\mathrm{d}z\wedge\mathrm{d}w,$$ for $0 \le b, c\le 2$ and $0 \le e, f, g\le 4.$ These are all invariant under $\sigma$. We find that such a form is invariant under $J_1$ iff $4 \, \vert \, b+c$ and under $J_2$ iff $6 \, \vert \, e+f+g$. We find that $\mathcal{H}_{\mathrm{FJRW}}^{\mathrm{I}} \cong \mathbb{C}^{42}$.
\item \textbf{Case I*.} $W\vert_g = Y^4 + Z^4 + y^6 + z^6 + w^6$. The standard generators of the Milnor ring are given by $$Y^bZ^cy^ez^fw^g\mathrm{d}Y\wedge\mathrm{d}Z\wedge\mathrm{d}y\wedge \mathrm{d}z\wedge\mathrm{d}w,$$ such that $4\,\vert\, b+c+2$ and $6 \, \vert \, e+f+g + 3$. This gives $3 \times 20$ possibilities, so that $\mathcal{H}_{\mathrm{FJRW}}^{\mathrm{I}*} \cong \mathbb{C}^{60}$.
\item \textbf{Cases II, II*, IV, IV*, VI, VI*.} Here $W_g$ has only one $X^2$ or $x^2$ appearing, and so to be invariant under $\sigma$ we would require the forms over $\mathbb{C}$ to be 0. Thus these cases do not contribute to $\mathcal{H}_{\mathrm{FJRW}}$.
\item \textbf{Cases III, V*, VII.} These $W_g$ have only one of $Y^4+Z^4$ or $y^6+z^6+w^6$ appearing (WLOG $Y^4+Z^4$), but also have $X^2+x^2$ appearing. Thus, invariance under both $J_2$ (otherwise $J_1$) will require the forms to be zero. Again, there is no contribution.
\item \textbf{Case III*.} Here $W\vert_g = y^6 + z^6 + w^6$, with Milnor generators $y^ez^fw^g\mathrm{d}y\wedge\mathrm{d}z\wedge\mathrm{d}w$. We require $6 \vert e + f + g + 3$, giving $\mathcal{H}_{\mathrm{FJRW}}^{\mathrm{III}*} \cong \mathbb{C}^{20}$.
\item \textbf{Case V.} These are the narrow sectors, each isomorphic to $\mathbb{C}$.
\item \textbf{Case VII*.} $W_g = Y^4+ Z^4$ Similarly to case III*, we find that $\mathcal{H}_{\mathrm{FJRW}}^{\mathrm{VII}*} \cong \mathbb{C}^3. $
\end{itemize}

Putting this together and sorting by bi-degree, we find that $$\bigoplus_{g \in G} \mathcal{H}_{\mathrm{FJRW}}^{g} \cong \mathcal{H}_{\mathrm{GW}}(\mathcal{Y})$$ as desired.

We similarly find that $\mathcal{H}_{\mathrm{FJRW}}^{\mathrm{nar}} \cong \mathcal{H}_{\mathrm{GW}}^{\mathrm{amb}}$. The narrow sectors have generators corresponding to $g \in G$ for which $\mathrm{Fix}(g) = \{0\}$ (those of type $V$ above). Ordered by degree, these are: $$ \begin{matrix}
J_1J_2\\
J_1^3J_2, \, J_1J_2^3, \, \sigma J_1^2J_2^2\\
J_1^3J_2^3, \, J_1J_2^5, \, \sigma J_1^2 J_2^4\\
J_1^3J_2^5
\end{matrix}$$


\subsection{The Narrow Sectors}

In general, if $w_0, w_1, w_2, w_3$ are all relatively prime, our group has elements $\sigma^tJ_1^rJ_2^s$, for $t = 0, 1, \; 0 \le r \le 3, \; 0 \le s \le  d-1$. We wish to find the narrow sectors, which correspond to those group elements $h$ for which  $\Theta_k(h) \ne 0$ for all $k$. First we consider the case where $t = 0$. Then we must have $r = 1, 3$, and $s$ must not be divisible by any $\frac{d}{w_i}$. For each i, there are $w_i-1 = (\frac{d}{d/w_i}-1)$ values of $s$ which are divisible by $d/w_i$ not counting 0 (giving broad sectors), and these are the ones we exclude. The Gorenstein condition for the K3 surface tells us that $d : = \mathrm{gcd}(w_0, w_1, w_2, w_3) = \sum_i w_i$, so it follows that there are in total $3 = d - \sum_{i=0}^3(\frac{d}{d/w_i} - 1) -1$ possible $s$. (Two of these will always be precisely $s = 1, d-1$; from the bounds given by the degree formula, these are the only ones which can ever have degree 0 or 6 respectively.) This gives at least 6 sectors, as above. If $t = 1$, then $r = 2$, and we require $s$ to be even, and then we argue similarly, but just for $w_1, w_2, w_3$: this gives two possible $s$ giving narrow sectors. The sectors $\sigma J_1^2 J_2^2$ and $\sigma J_1^2 J_2^{d-2}$ are included among these, so this completes the set.

However, when $w_i, w_j$ have a common factor $p_{i, j}$, we note that there are $p_{i, j}-1$ broad sectors corresponding to $k\frac{d}{p_{i, j}}$ for $k = 1, \ldots, p_{i, j} -1$  which have been excluded twice in the above method, since they are divisible by both $\frac{d}{w_i}$ and $\frac{d}{w_j}$., So to avoid double-counting, we must exclude $p_{i, j}-1$ fewer sectors above when $t=0$; that is, we add another $p_{i, j}-1$ possible $s$. All of the corresponding $J_1J_2^s$ have degree 2 (that is, excluding $s = 1, d-1$), and all the corresponding $J_1^3J_2^s$ have degree 4. Furthermore, $J_1J_2^s$ is inverse to $J_1J_2^{d-r}$ with respect to the FJRW-pairing. 

When $t=1$, for the sectors to be narrow we require $s$ to be even. Thus for each $i$, we must find $k$ such that $2k(\frac{w_i}{d}) = k\frac{w_i}{w_0}$ is not an integer for any $i$. We count  first those which give integers for some $i$. Since $w_1 + w_2 + w_3 = w_0 = \frac{d}{2}$, it follows that each $w_i = p_{0, i}$ (either 1 or a prime) or $2p_{0, i}$. If $w_j = p_{0, j}$ (in which case $w_j \vert w_0$), or $w_j = 2p_{0, j}$ and $w_i\vert w_0$, then $k\frac{w_j}{w_0}$ is an integer for any $k$: there are then $\frac{w_0}{w_j}$ possible $k$ for each $w_i$, $i = 1, 2, 3$, and if they are all relatively prime, then this gives a total of $w_0 - \sum_{i=1}^3(w_i-1) -1 = 2$ narrow sectors accounting for $k = 0$.   If $w_j = 2p_{0, j} \not\vert w_0$, then $k\frac{w_j}{w_0}$ is only an integer for \textit{even} $k$: thus we have  $p_{0, j}$ fewer broad sectors, so for each such case we must add a further $p_{0, j}$ broad sectors. If $w_i, w_j$ share a common factor for $i = 1, 2, 3$, then that common factor is 2, and it follows that this situation adds no further narrow sectors. Alternatively one may just check the $\Theta$ values of $\sigma J_1^2J_2^s$ for the two cases $(5, 2, 2, 1)$ and $(21, 14, 6, 1)$ to see that these do contribute $2 + 1 + 1$ and $2 + 7 + 3$ narrow sectors with $t = 1$, respectively.

From the FJRW degree formula for the narrow sectors, with $N = 6$, we see that the degree is clearly even,  can only be 0 for $J_1J_2$, and therefore 6 only for $(J_1J_2)^{-1}$. Otherwise since $\mathrm{deg}_W(h) = 6 - \mathrm{deg}_W(h^{-1})$, for the rest the degree is equi-distributed between 2 and 4.

Ordering by degree, the narrow FJRW state space is then generated by sectors coming from
\begin{itemize}
\item $J_1J_2$
\item $J_1^3J_2, \sigma J_1^2J_2^2,  1+ \sum_{i, j: \mathrm{gcd}(i, j) >1}(p_{i, j} - 1)$ sectors of the form $J_1J_2^s$,\\
 $\frac{1}{2}\sum_{w_j \not\vert w_0} p_{0, j}$ sectors of the form $\sigma J_1^2J_2^s$
\item $J_1J_2^{d-1}, 1+ \sum_{i, j: \mathrm{gcd}(i, j) >1}(p_{i, j} - 1)$ sectors of the form $J_1^3J_2^s, \sigma J_1^2 J_2^{d-2},$\\
 $\frac{1}{2}\sum_{w_j \not\vert w_0} p_{0, j}$ sectors of the form $\sigma J_1^2J_2^s$
\item $J_1^3J_2^{d-1}$
\end{itemize}

For example, for $(w_0, w_1, w_2, w_3) = (5, 2, 2, 1)$ we have narrow sectors coming from
$$ \begin{matrix}
J_1J_2\\
J_1^3J_2,, J_1J_2^3, J_1J_2^7, \sigma J_1^2J_2^2, \sigma J_1^2J_2^6, \\
J_1J_2^9, J_1^3J_2^3, J_1^3 J_2^7, \sigma J_1^2 J_2^8, \sigma J_1J_2^4 \\
J_1^3J_2^9
\end{matrix}$$

Here, $J_1J_2$ is the unit of our product structure $\mathbf{1}$ (and, in fact, carries all the \textit{unstable} enumerative data). 

For $(3, 1, 1, 1)$, things simplify yet further. The cup product can be found entirely by three facts, all found in \cite{FJR}:
\begin{enumerate}
\item $\langle \alpha \cup \beta, \mathbf{1}\rangle = \langle \alpha \cdot \beta \cdot \mathbf{1} \rangle = \eta(\alpha, \beta)$;
\item For $\langle \alpha \cdot \beta \cdot \gamma \rangle = \eta (\alpha \cup \beta, \gamma)$ to be non-zero, we need their degrees to sum to 6;
\item For the previous condition to hold we need all the line bundles to be integral. For three-point FJRW-invariants this implies $q_j - \sum_{i=1}^3(\Theta_i(h_j) \in \mathbb{Z}$.
\end{enumerate}
All non-zero 3-point invariants (and hence the cup product) are determined from the pairing except for two: $\langle \phi_{\sigma J_1^2J_2^2} , \phi_{\sigma J_1^2 J_2^2}, \phi_{J_1J_2^3}\rangle$ and $\langle \phi_{J_1^3J_2}, \phi_{J_1J_2^3}, \phi_{J_1J_2^3}\rangle$, which we have freedom to set to 1. In terms of the cup product, this is exactly the ring structure of Chen-Ruan cohomology under the identification 
\begin{align*}
& D_E  \mapsto \phi_{J_1^3J_2}\\
& D_K \mapsto \phi_{J_1J_2^3}\\
& \Sigma\mathbf{1}_{\sigma} \mapsto \phi_{\sigma J_1^2J_2^2}
\end{align*}

\subsection{The FJRW I-function}

In genus zero, the virtual class takes a simple form.
\begin{lem}\label{vir}
In genus zero, we have
$$[\mathcal{W}_{0,n}^{(W, G)}]^{\mathrm{vir}} = R^1\pi_*(\bigoplus_{k=1}^7 \mathcal{L}_k)$$
\end{lem}
\begin{proof}
By definition $[\mathcal{W}_{0,n}^{(W, G)}]^{\mathrm{vir}} = (-R^0+R^1)\pi_*(\bigoplus_{k=1}^8 \mathcal{L}_k),$ and over each point $(\mathcal{C}, p_1, \ldots, p_n, \mathcal{L}_1, \ldots, \mathcal{L}_N, \varphi_1, \varphi_2, \ldots, \varphi_N)$, the  fibre of $R^0\pi_*(\bigoplus_{k=1}^N \mathcal{L}_k)$ is given by $\bigoplus_{k=1}^N H^0(\mathcal{C}, \mathbb{L}_k)$. It is thus sufficient to show that, for all $k$, $\mathbb{L}_k$ has no non-trivial global sections.

We proceed by induction on the dual graph $\Gamma$ of $\mathcal{C}$. Vertices of $\Gamma$ correspond to irreducible components of $\mathcal{C}$, and edges correspond to nodes. Each vertex $v$ is marked by $g_v$, the genus of the component, and $S_v$, the set of marked points or nodes, which has cardinality $k_v$.

From \cite{FJR}, the degree of the pushforward of $\mathcal{L}_i$ to the coarse curve is $$\mathrm{deg}\, (\vert \mathcal{L}_k \vert ) = q_k (k_v-2) - \sum_{i=1}^n \Theta_k(h_i),$$ where $h_i$ records the multiplicity of $\mathcal{L}_i$ at the marked points. For the narrow sectors, $\Theta_k(h_i) > 0$ for all $i, k$. If $\mathcal{C}$ is irreducible (so there are no nodes) we have $\Theta_k(h_i) = m_{i, k}q_k$ for $m_{i, k} \in \mathbb{Z}_{>0}$, and so $$\mathrm{deg}(\vert \mathcal{L}_k \vert) = -2q_k + q_k \sum_{i=1}^n (1-m_{i, k}) < 0,$$ so that $H^0(\mathcal{C}, \mathcal{L}_k) = 0$.
 
Otherwise, since $\mathcal{C}$ is a compact connected curve of genus zero, $\Gamma$ is a finite tree where every irreducible component has at least one node. Let $\sigma_k$ be any global section of $\mathcal{L}_k$ on $\mathcal{C}$. Let $n_v$ be the number of edges (nodes) attached to $v$. Then we have $$\mathrm{deg}(\mathcal{L}_k\vert_v) \le q_k (n_v-2) < n_v-1.$$ If $v$ is a leaf of $\Gamma$ corresponding to irreducible component $\mathcal{C}_v$, then $n_v = 1$ and $\mathrm{deg}(\mathcal{L}_k\vert_v) < 0$ and so $\sigma_k\vert_v = 0$. We proceed by induction: if at each stage we remove all the leaves and consider the leaves of the new graph we obtain, these correspond to vertices of the original graph all of whose adjacent vertices but one have been removed earlier; that is, to components all of whose nodes but one are connected to components on which $\sigma$ is known to be zero. We must show that $\sigma$ is zero on all the leaves. This follows from the the above inequality, as $\sigma$ has more zeroes on the given component than its degree there, so must be constantly zero. The result follows.
\end{proof}

\begin{dfn}
For $h \in G$, let $i_k(h) = \langle \Theta_k(h) - q_k\rangle$.
\end{dfn}
Note that since each $k$ such that $\Theta_k(h) = 0$ increases $N$ by 1, we have  $\mathrm{deg}_W(h) = 2\sum_k i_k(h)$.


We have already discussed the twisted variants in \S 2. These correspond to a full twisted FJRW theory, which corresponds to considering a slightly different set of line bundles. Intuitively, we may separate out the line multiplicities by the greatest common divisors of their multiplicities at the points; accounting for the fact that  $d$-th roots may be lower degree roots too. First, note that over marked curves, to give the $n$-th root $\mathcal{L}$ of $\omega_{\mathrm{log}}$ is equivalent to giving an $n$-th root $\tilde{\mathcal{L}}$ of $\omega_{\mathrm{log}}(-D)$ for some divisor $0 \le D < \sum_i nD_i$, with multiplicity 0 at $p_i$. This correspondence can be given by  $\mathcal{L} \mapsto p^*p_* \mathcal{L}$, where $p$ is the map which forgets the monodromy at the marked points. Then from this perspective we can separate our narrow sectors from the rest in a slightly modified moduli space $\tilde{\mathcal{W}}_{g, n}.$ This can be constructed for one line bundle at time, for each component corresponding to a series of multiplicities $\mathbf{h}$, and then given in full as a union of all possible fibre products of such moduli spaces across all 7 line bundles, satisfying conditions specified by the group action. Thus the moduli space decomposes into a disjoint union of moduli spaces $\tilde{\mathcal{W}}(\mathbf{h})$. We will not dwell on the full machinery here, but see \cite{FJR} for details.

We wish to show that considering these modified line bundles (and thus defining our invariants over this twisted moduli space instead) removes the issue of broad sectors in a clean way for the twisted invariants in genus zero.

\begin{lem}\textbf{(Ramond Vanishing)}
Over $\tilde{\mathcal{W}}_{0, n}(\mathbf{h})$, we have that $\pi_*(\sum_{k=1}^7 \mathcal{L}_k) = 0$ and $R^1\pi_*(\bigoplus_{k=1}^7 \mathcal{L}_k)$ is locally free.
\end{lem}
\begin{proof}
This is clear if all sectors in $\mathbf{h}$ are narrow. Assume there is a broad sector $h_i$. Then  $\Theta_k(h_i) = 0$ for some $k$. The short exact sequence $$0 \to \tilde{\mathcal{L}}_k \to \tilde{\mathcal{L}}_k(D_i) \to \tilde{\mathcal{L}}_k(D_i) \vert_{D_i} \to 0$$ induces a long exact sequence
\begin{align*}
0 & \to \pi_*\tilde{\mathcal{L}}_k \to \pi_* \tilde{\mathcal{L}}_k(D_i) \to \pi_* \tilde{\mathcal{L}}_k(D_i) \vert_{D_i}\\
& \to R^1\pi_*\tilde{\mathcal{L}}_k \to R^1\pi_* \tilde{\mathcal{L}}_k(D_i) \to R^1\pi_* \tilde{\mathcal{L}}_k(D_i) \vert_{D_i}.
\end{align*}
The last term is zero by dimension considerations, and the first two terms are zero by a similar argument to lemma \ref{vir}, but with one extra subtlety - if $\mathcal{C}$ is reducible and we follow through the same argument for the irreducible component $\mathcal{C}_v$ containing $p_i$, we may conclude that $\mathrm{deg}\tilde{\mathcal{L}}_k(D_i)\vert_{\mathcal{C}_{p_i}} < n_v$ (the number of nodes on $\mathcal{C}_v$), not $<n_v - 1$. But since in our inductive step we are considering \textit{non}-tails it must be connected to at least two other components, and since we can prove this step for all other components it follows for $C_v$ too.
Thus our long exact sequence is reduced to 
$$0 \to \pi_* \tilde{\mathcal{L}}_k(D_i) \vert_{D_i} \to R^1\pi_*\tilde{\mathcal{L}}_k \to R^1\pi_* \tilde{\mathcal{L}}_k(D_i) \to 0,$$
and so $$c_{\mathrm{top}}(R^1\pi_*\tilde{\mathcal{L}}_k) = c_{\mathrm{top}}(\pi_* \tilde{\mathcal{L}_k}(D_1)\vert_{D_1})\cdot c_{\mathrm{top}}(R^1\pi_*\tilde{\mathcal{L}}_k(D_1)).$$
Note that $\tilde{\mathcal{L}}_k(D_i)\vert_{D_i} \cong \mathcal{L}_k\vert_{D_i}$, and since by assumption our multiplicity is zero here, this is a root of $\omega_{\mathrm{log}}\vert_{D_1}$, and this is trivial, we have that $c_{\mathrm{top}}(\pi_*(\tilde{\mathcal{L}}_k(D_1)\vert_{D_1}) = 0$, and so $c_{\mathrm{top}}(R^1\pi_*\tilde{\mathcal{L}}_k) = 0$.
\end{proof}

The derivation of the I-function  will follow similar reasoning to that of the analogous theorem in \cite{ChiRu1}. First, we compute the untwisted J-function. For \begin{equation*}\label{eq:fjrwinv}\langle \phi_{h_1}, \ldots, \psi^a\phi_h\rangle_{0,n+1}^{\mathrm{untwisted}} = \int_{[\mathcal{W}_{0, n, +1, G}(h_1, \ldots, h_n, h)]^{\mathrm{vir}}} \psi^{a},\end{equation*} to be non-zero we require the moduli space to be non-empty, which when all sectors are narrow requires that $$-2q_k +\sum_i i_k(h_i) = i_k(h)$$ be an integer, so that $i_k(h) = \langle 2q_k - \sum_{i=1}^n i_k(h_i) \rangle$. Accordingly, for $\mathbf{n} = (n_1, n_2, \ldots, n_n)$ we define $h_{\mathbf{n}}$ by $i_k(h_{\mathbf{n}}) = \langle -2q_k + \sum_{i=1}^n i_k(h_{n_i}) \rangle$. We further require that the degree of $\psi^a$  matches the dimension of the moduli space; that is, $a = n-2$. In this case, by the string equation proved in \cite{FJR}, we find that \eqref{eq:fjrwinv} is 1.

Therefore $$J^{\mathrm{un}}(\mathbf{t}, z) = \sum_{\mathbf{n} : G \to \mathbb{N}_0} \frac{1}{z^{\vert\mathbf{n}\vert-1} }\prod_{h \in G} \frac{(t^h)^{\mathbf{n}(h)}}{\mathbf{n}(h)!}\phi_{h_{\mathbf{n}}},$$
where we label the term in the sum corresponding to $\mathbf{n}$ by $J^{\mathrm{un}}_{\mathbf{n}}(\mathbf{t}, z)$.

We define a \textit{twisting map} $\Delta:\mathcal{V}^{\mathrm{un}} \to \mathcal{V}^{\mathrm{tw}}$ which takes the \textit{untwisted} Lagrangian cone $\mathcal{L}^{\mathrm{un}}$ defined via the untwisted invariants to the \textit{twisted} Lagrangian cone $\mathcal{L}^{\mathrm{tw}}$ defined similarly via the twisted invariants.
\begin{prop}
If we set $$\Delta = \bigoplus_h \prod_{k=1}^7 \mathrm{exp}(\sum_{d\ge0} s_d \frac{B_{d+1}(i_k(h) + q_k)}{(d+1)!}z^d),$$ where $B_d(x)$ is the $d$-th Bernoulli polynomial, then $\Delta$ is a symplectic transformation and $\Delta(\mathcal{L}^{\mathrm{un}}) = \mathcal{L}^{\mathrm{tw}}$.
\end{prop}
\begin{proof}
First, we show that $\Delta$ is symplectic. Consider $ \Omega (\Delta f, \Delta g)$ for $f, g \in \mathcal{V}^{\mathrm{un}}$. Expanding $f, g$ we note that the only non-zero terms in the product are given by pairing  terms in $\phi_h$ in $f$ with terms in $\phi_{h^{-1}}$ in $g$. We write $f = \sum_h f_h \phi_h$ and $g = \sum_h g_h \phi_h$. Note that in this case $i_k(h^{-1}) +q_k = 1-(i_k(h)+q_k)$ and since the Bernoulli polynomials satisfy $B_{d+1}(1-x) = (-1)^{d+1} B_d(x)$, it follows that $ \Omega (\Delta f, \Delta g)$ is the residue at $z=0$ of the sum given for each $h$ by 
\begin{align*}
& \langle \prod_ke^{\sum_{d\ge0} s_d \frac{B_{d+1}(i_k(h) + q_k)}{(d+1)!}z^d} f_h(z) \phi_h, \prod_ke^{\sum_{d\ge0} s_d \frac{B_{d+1}(i_k(h^{-1}) + q_k)}{(d+1)!}(-z)^d} g_{h^{-1}}(-z) \phi_{h^{-1}}\rangle\\
& = e^{\sum_{k, d} (s_d \frac{B_{d+1}(i_k(h) + q_k)}{(d+1)!}z^d +
s_d \frac{(-1)^{d+1}B_{d+1}(i_k(h) + q_k)}{(d+1)!}(-1)^dz^d)} \langle f_h(z) \phi_h, g_{h^{-1}}(-z) \phi_{h^{-1}}\rangle\\
& =  \langle f_h(z) \phi_h, g_{h^{-1}}(-z) \phi_{h^{-1}}\rangle,
\end{align*}
which gives $\mathrm{Res}_{z=0} \bigoplus_h  \langle f_h(z) \phi_h, g_{h^{-1}}(-z) \phi_{h^{-1}}\rangle = \Omega(f, g)$, as required.

The potential is defined in terms of a quantum expansion; we show that $\hat{\Delta} \mathcal{D}^{\mathrm{un}} = \mathcal{D}^{\mathrm{tw}}$, where $\hat{\Delta}$ is the quantisation of $\Delta$ (see the detailed exposition in \cite{CPS} for details). 
This is given by first forming the \textit{Hamiltonian} corresponding to $\sum_{g\ge0} \hbar^{g-1} \mathcal{F}$ by $h_F(v) = \frac{1}{2}\langle \sum_{g\ge0} \hbar^{g-1}\mathcal{F}v, v\rangle$, and then quantising the coordinates of $\mathcal{V}^{\mathrm{un}}, \mathcal{V}^{\mathrm{tw}}$ by replacing them with the operators
\begin{align*}
\hat{q}^h & = q^h,\\
\hat{p}^h & = \hbar \frac{\partial}{\partial q^h},
\end{align*} 
we obtain $\hat{h}_F$, and define $\hat{\mathcal{D}} = e^{\hat{h}_F/\hbar}$.

We show that $\hat{\Delta}\mathcal{D}^{\mathrm{un}} = \mathcal{D}^{\mathrm{tw}},$. 

Remembering the dilaton shift \eqref{dilaton}, it is straightforward to check term by term that $\hat{\Delta}\mathcal{D}^{\mathrm{un}}$ is a solution to
\begin{align*}
\frac{\partial \Phi}{\partial s_d} = & \sum_{k=0}^7 P_d^{(k)}\Phi,\\
P_d^ {(k)} & = \frac{B_{d+1}(q_k)}{(d+1)!}\frac{\partial}{\partial t^{J_1J_2}_{d+1}} - \sum_{\substack{a\ge 0\\ h \in G}} \frac{B_{d+1}(i_k(h)+q_k)}{(d+1)!} t_a^h\frac{\partial}{\partial t_{d+a}^h} \\
& + \frac{\hbar}{2} \sum_{\substack{a+a' = d-1\\h, h' \in G}}(-1)^{a'} \eta^{h,h'} \frac{B_{d+1}(i_k(h) + q_k)}{(d+1)!} \frac{\partial^2}{\partial t_a^h \partial t_{a'}^{h'}}.
\end{align*}
Here $\eta^{h, h'} = e^{s_0}\delta_{h', h^{-1}}$, the inverse of the pairing.

We wish to show that  and $\mathcal{D}^{\mathrm{tw}}$ is also a solution to the above equation; the equality will then follow. Expanding $\mathcal{D}^{\mathrm{tw}}$ and remembering the integral definition of the invariants and $\mathbf{c}(E)$, after some calculation we find that  $\mathcal{D}^{\mathrm{tw}}$ is a solution if and only if
\begin{align*}
& \sum_{n=0}^{\infty} \frac{1}{n!} \langle \mathbf{t}(\psi), \mathbf{t}(\psi), \ldots, \mathbf{t}(\psi), \mathrm{ch}_d(R\pi_*(\mathcal{L}_k))\cdot\mathbf{c}(R\pi_*(\bigoplus_l \mathcal{L}_l))\rangle_{0, n}\\
= & P_d^{(k)}\mathcal{F}_g^{\mathrm{tw}} + \frac{\hbar}{2}\sum_{\substack{a+a' = d-1\\h, h' \in G}}(-1)^{a'} \eta^{h,h'} \frac{B_{d+1}(i_k(h) + q_k)}{(d+1)!} \frac{\partial \mathcal{F}_g^{\mathrm{tw}}}{\partial t_a^h} \frac{\partial \mathcal{F}_g^{\mathrm{tw}}}{\partial t_{a'}^{h'}}
\end{align*}
Here we appeal to \cite{CZ}, where this equation was proved in some generality.

Therefore we have $\hat{\Delta}(\mathcal{D}^{\mathrm{un}}) = \mathcal{D}^{\mathrm{tw}}$. Taking the semi-classical limit $\hbar \to 0$, the conclusion follows.
\end{proof}

We set
\begin{align*}
& D_h =  t^h \frac{\partial}{\partial t_0^h},\\
& D^k =  \sum_{h \in G} i_k(h) D_h,\\
& \mathbf{s}(x) =  \sum_{d\ge0} s_d \frac{x^d}{d!},\\
& G_y(x ,z) =  \sum_{l, m \ge 0} s_{l+m-1} \frac{B_m^*(y)}{m!}\frac{x^l}{l!}z^{m-1}.
\end{align*}

Note that
\begin{align*}
&D_hJ^{\mathrm{un}}_{\mathbf{n}}(\mathbf{t}, z) = \mathbf{n}(h)J_{\mathbf{n}}^{\mathrm{un}} (\mathbf{t}, z),\\
&G_y(x, z) = G_0(x+yz, z),\\
&G_0(x+z, z) = g_0(x, z) + \mathbf{s}(x).
\end{align*}

We will build our twisted I-function from these functions, and prove that it lies on $\mathcal{L}^{\mathrm{tw}}$ as given above. We set $$J^{\mathbf{s}}(\mathbf{t}, z) =\mathrm{exp}(\sum_{k=1}^7 (-G_{q_k}(zD^k, z))).$$

Note that we may write 
\begin{align*}
J^{\mathbf{s}}(\mathbf{t}, z) & = \sum_{\mathbf{n}} \prod_{k=1}^7 \mathrm{exp}(-G_{q_k}((\sum_{h\in G} \mathbf{n}(h)i_k(h))z, z)J_{\mathbf{n}}^{\mathrm{un}}(\mathbf{t}, z),\\
\Delta & = \prod_{k=1}^7 \bigoplus_{h \in G} \mathrm{exp}(G_{q_k}(i_k(h)z, z)).
\end{align*}

\begin{lem}
$J^{\mathbf{s}}(\mathbf{t}, z)$ lies on $\mathcal{L}^{\mathrm{un}}$.
\end{lem}
\begin{proof}
We follow the proof found in \cite{ChiRu1}. We write elements of $\mathcal{V}^{\mathrm{tw}}$ in standard form: $$f = -z\phi_{J_1J_2} + \sum_{l \ge 0} \mathbf{t}_l z^l +  \sum_{l \ge 0} \frac{\mathbf{p}_l(f)}{(-z)^{l+1}}$$
for some $\mathbf{p}_l(f) = \sum_{h \in G} p_{l, j}(f)\phi^h.$ Define $$E_l(f) = \mathbf{p}_l(f) - \sum_{\substack{n \ge 0 \\ h \in G}} \frac{1}{n!} \langle \mathbf{t}(\psi), \ldots, \mathbf{t}(\psi), \psi^l\phi_h\rangle J^{\mathrm{un}}_{0, n+1} \phi^h.$$ Then $$\mathcal{L}^{\mathrm{un}} = \{ f \in \mathcal{V}^{\mathrm{un}} \,\vert\, E_l(f) = 0\}.$$

For the purposes of an enumeration for induction we set $\mathrm{deg}(s_d) = d+1$ and perform induction on the terms of $J^{\mathbf{s}}(\mathbf{t}, z)$, ordered by that degree, treating it as a power series in $s_1, s_2, \ldots$. (The $s_d$ are mixed since they appear in the exponential factor, so we do not simply perform induction on $d$.) The degree zero part has no equivariant terms and reduces to $J^{\mathrm{un}}(\mathbf{t}, z)$, which is of course in $\mathcal{L}^{\mathrm{un}}$. This is the base case. Now assume that the degree $n$ part vanishes. Then we can find some `corrected' $\tilde{J^{\mathbf{s}}}(\mathbf{t}, z)$ which \textit{is} in $\mathcal{L}^{\mathrm{un}}$ which has the same degree-$n$ part as $J^{\mathbf{s}}$. 

The idea is that taking the derivative with respect to each $s_d$ knocks down the degree of each term by at least 1, and we can use our inductive hypothesis and properties of the Lagrangian cone to prove the hypothesis for the degree $n+1$ part. By the chain rule, viewing $E_l$ as a function of $J^{\mathbf{s}}$:
$$\frac{\partial}{\partial s_d} (E_j (J^{\mathbf{s}}(\mathbf{t}, -z)) = (d_{J_{\mathbf{s}}}E_l)\circ(z^{-1}P_dJ^{\mathbf{s}}(\mathbf{t}, z)),$$ where from the definition of $J^{\mathbf{s}}$ we have $$P_d = \sum_{m=0}^{d+1} \frac{1}{m!(i+1-m)!} z^m B_m(q_k) (zD^k)^{d+1-m}.$$ But then up to degree $n$, $$\frac{\partial}{\partial s_d} (E_l J^{\mathbf{s}}) = \frac{\partial}{\partial s_d} (E_l \tilde{J^{\mathbf{s}}}).$$ Since $zT_{J^{\mathbf{s}}}\mathcal{L}^{\mathrm{un}} = \mathcal{L}^{\mathrm{un}} \cap T_{J^{\mathbf{s}}}\mathcal{L}^{\mathrm{un}}$, and $\mathcal{L}^{\mathrm{un}}$ is a cone, we have that operating on any point of $\mathcal{L}^{\mathrm{un}}$ by $zD$ gives us a point of $zT_{J^{\mathbf{s}}}\mathcal{L}^{\mathrm{un}} \subset T_{J^{\mathbf{s}}}\mathcal{L}^{\mathrm{un}}$. By repeated application a polynomial in such operators still keeps the point in this space, since $z^kT_{J^{\mathbf{s}}}\mathcal{L}^{\mathrm{un}} \subset zT_{J^{\mathbf{s}}}\mathcal{L}^{\mathrm{un}}$, so $P_d\tilde{J^{\mathbf{s}}}$ is still in $zT_{J^{\mathbf{s}}}\mathcal{L}^{\mathrm{un}}$; and thus $z^{-1}P_d\tilde{J^{\mathbf{s}}}$ lies in $T_{J^{\mathbf{s}}}\mathcal{L}^{\mathrm{un}}$. Therefore the whole derivative vanishes, so that the hypothesis holds up to degree $n+1$, and our conclusion follows.
\end{proof}

It follows from the two previous lemmas that $\Delta(J^{\mathbf{s}})$ lies in the twisted Lagrangian cone; this will be our equivariant I-function. We have 
$$I^{\mathrm{tw}}(\mathbf{t}, z) := \Delta(J^{\mathbf{s}}(\mathbf{t}, -z)) = \sum_{\mathbf{n}} \mathbf{M}_{\mathbf{n}}(z)J_{\mathbf{n}}^{\mathrm{un}}(\mathbf{t}, z),$$
where
\begin{align*}
\mathbf{M}_{\mathbf{n}}(z) & = \prod_{k=1}^7 \mathrm{exp}(G_{q_k}(\langle \sum_{h\in G}\mathbf{n}(h)i_k(h)\rangle z, z) - G_{q_k}(\sum_{h \in G} \mathbf{n}(h) i_k(h) z, z))\\
& = \prod_{\substack{1\le k \le 7\\0 \le b < \lfloor \sum_{h\in G} \mathbf{n}(h) i_k(h) \rfloor}}\mathrm{exp}(-\sum_{d\ge 0} s_d \frac{(-q_kz - \langle \sum_{h \in G}\mathrm{n}(h)i_k(h) \rangle z - bz ) ^d}{d!})\\
& = \prod_{\substack{1\le k \le 7\\0 \le b < \lfloor \sum_{h\in G} \mathbf{n}(h) i_k(h) \rfloor}}\mathrm{exp}(\mathrm{ln}(\lambda) +  \sum_{d>0} (-1)^{d+1} \frac{(q_k z + \langle \sum_{h \in G}\mathrm{n}(h)i_k(h) \rangle z + bz)^d}{d!})\\
& = \prod_{\substack{1\le k \le 7\\0 \le b < \lfloor \sum_{h\in G} \mathbf{n}(h) i_k(h) \rfloor}} (\lambda + q_k z +  \langle \sum_{h \in G}\mathrm{n}(h)i_k(h) \rangle z + bz)
\end{align*}
similarly to the argument in \cite{ChiRu1}, and likewise as $\lambda \to 0$ we obtain the non-equivariant I-function, which in our case we restrict to the narrow sectors of degree at most 2, so that
\begin{align}
& I^{\mathrm{nar}}_{FJRW}(\mathbf{t}, z)  = \sum_{\mathbf{n}:\mathcal{N}_{\mathrm{deg \le 2}} \to \mathbb{N}_0} \bigg(\prod_{\substack{1\le k \le 7\\0 \le b < \lfloor \sum_{h\in G} \mathbf{n}(h) i_k(h) \rfloor}}  \Big(q_k  +  \langle \sum_{h }\mathrm{n}(h)i_k(h) \rangle + b\Big)\bigg) \nonumber \\
& \times \,  z^{\sum_k\lfloor \sum_{h }\mathrm{n}(h)i_k(h) \rfloor - \vert\mathbf{n}\vert+1}\prod_h \frac{\mathbf{t}^{\mathbf{n}(h)}}{\mathbf{n}(h)!}\phi_{h_{\mathbf{n}}}\nonumber \\
& =  \sum_{\mathbf{n}:\mathcal{N}_{\mathrm{deg \le 2}} \to \mathbb{N}_0} \bigg(\prod_{1\le k \le 7} \frac{\Gamma\Big(q_k  +  \sum_{h}\mathrm{n}(h)i_k(h) \Big)}{\Gamma\Big(q_k  +  \langle \sum_{h }\mathrm{n}(h)i_k(h) \rangle + 1\Big)}\bigg) z^{\sum_k\lfloor \sum_{h }\mathrm{n}(h)i_k(h) \rfloor - \vert\mathbf{n}\vert+1}\prod_h \frac{\mathbf{t}^{\mathbf{n}(h)}}{\mathbf{n}(h)!}\phi_{h_{\mathbf{n}}}.\nonumber \\
& =: \sum_{h \in \mathcal{N}_{\mathrm{deg \le 2}}} \omega_h^{\mathrm{FJRW}}z^{\sum_k\lfloor i_k(h) \rfloor - \vert\mathbf{n}\vert+1} \phi_h. \nonumber
\end{align}
Separating out $J_1J_2$ contributes a constant factor of $\sum_{n_{J_1J_2}=0}^{\infty}\frac{z^{-n_{J_1J_2}}}{n_{J_1J_2}!} = e^{-z}$.

This expression may be most easily expanded for $W = X^2+Y^4+Z^4+x^2 + y^6 + z^6 + w^6 = 0$:
\begin{align*}
& (i_k(J_1^3J_2))_k = (0, \frac{1}{2}, \frac{1}{2}, 0, 0, 0, 0),\\
&  (i_k(J_1J_2^3))_k = (0, 0, 0, 0, \frac{1}{3}, \frac{1}{3}, \frac{1}{3}),\\
& (i_k(\sigma J_1^2 J_2^2))_k = (0, \frac{1}{4}, \frac{1}{4}, 0, \frac{1}{6}, \frac{1}{6}, \frac{1}{6}).
\end{align*}
We re-express the sum over $\mathbf{n}$ as a sum over the triple $(\mathbf{n}(J_1^3J_2), \mathbf{n}(J_1J_2^3), \mathbf{n}(\sigma J_1^2 J_2^2)) = (M, N, C)$, as follows:

\begin{align*}
&I^{\mathrm{nar}}_{FJRW}(\mathbf{t}, z) =  \sum_{M, N, C \ge 0} \frac{\Gamma(\frac{1}{2})\Gamma(\frac{1}{4} + \frac{M}{2} + \frac{C}{4})^2 \Gamma(\frac{1}{2})\Gamma(\frac{1}{6} + \frac{N}{3} + \frac{C}{6})^3}{\Gamma(\frac{1}{2}+1) \Gamma(\frac{1}{4} + \langle \frac{M}{2} + \frac{C}{4}\rangle + 1)^2 \Gamma(\frac{1}{2} + 1) \Gamma(\frac{1}{6} + \langle \frac{N}{3} + \frac{C}{6}\rangle + 1)^3}\\
& \times \frac{T_1^MT_2^NT_3^C}{M!N!C!}z^{1- M - N - C + 2 \lfloor \frac{M}{2} + \frac{C}{4}\rfloor + 3 \lfloor \frac{N}{3} + \frac{C}{6}\rfloor}\phi_{h({M, N, C})}\\
& =: \sum_h \sum_{(M, N, C): h(M, N, C) = h} \omega_h^{\mathrm{FJRW}}z^{1- M - N - C + 2 \lfloor \frac{M}{2} + \frac{C}{4}\rfloor + 3 \lfloor \frac{N}{3} + \frac{C}{6}\rfloor} \phi_h.
\end{align*}
It is easily seen that the exponent of $z$ depends only on $(M \, \mathrm{mod} \, 2, N \, \mathrm{mod} \,3, C \, \mathrm{mod} \, 12)$, and that  for each $h$ there are only 6 cases. For other orbifolds of our type, there are similarly few cases, which depend similarly on the lowest common multiples of $w_i$.

In the general case, consider for which terms the exponent of $z$ can be 1. This would require $$\sum_k \mathbf{n} (h) = \sum_k \lfloor \sum_h \mathbf{n}(h) \rfloor,$$ but since it is always true that $i_k(h)<1$, we must have $i_k(h_{\mathbf{n}}) = \langle \sum_h \mathbf{n}(h) i_k(h) \rangle = 0$ for all $k$, whence $h_{\mathbf{n}} = J_1J_2.$ Since we require that $\langle \sum_h \mathbf{n}(h) i_k(h) \rangle = 0$, the linear coefficient of $\mathbf{t}$ must be zero, so it follows that the term in $z^1$ must be $f(\mathbf{t})z\phi_{J_1J_2}, $ where $f(\mathbf{t}) = f_0(t^{J_1J_2}) + \mathcal{O}(\mathbf{t}^2).$

Then we may write $I_{\mathrm{FJRW}}(z, \mathbf{t}) = f(\mathbf{t})z\phi_{J_1J_2} + \mathbf{g}(\mathbf{t}) + \mathcal{O}(z^{-1}).$ Then if we set $\mathbf{\tau}(\mathbf{t}) = \mathbf{g}(t)/f(t),$ we find that $$ \frac{I_{\mathrm{FJRW}}^{(W, G)} (\mathbf{t}, z)}{f(\mathbf{t}}$$ lies on $\mathcal{L}^{(W, G)}$ and is of the form $z\phi_{J_1J_2} + \mathbf{t} + \mathcal{O}(z^{-1}).$ Since the J-function is unique with respect to this property, this gives us the following FJRW `mirror theorem'.

\begin{prop} 
$$J_{\mathrm{FJRW}}^{(W, G)}(\mathbf{\tau}(\mathbf{t}), z) = \frac{I_{\mathrm{FJRW}}^{(W, G)} (\mathbf{t}, z)}{f(\mathbf{t})}.$$
\end{prop}

\section{The Intermediate Mixed Theories}


We consider the Borcea-Voisin orbifolds as arising as complete intersections $\mathbb{Y}$ in GIT quotients $$\mathcal{X}_{\theta} := [V//_{\theta}(\mathbb{C}^*)^3], V = (\mathbb{C}^3 \times \mathbb{C}^4  \times \mathbb{C} \times \mathbb{C}^2)$$ for some character $\theta$ of $(\mathbb{C}^*)^3,$ where $(\mathbb{C}^*)^3$ acts by $$\begin{pmatrix} 2 & 1 & 1 & 0 & 0 & 0 & 0 -4 & 0\\ 0 & 0 & 0 & w_0 & w_1 & w_2 & w_3 & 0 & -2w_0\\ 1 & 0 & 0 & 1 & 0 & 0 & 0 & 2 & 0 & 0\end{pmatrix}.$$ The potential for this theory is given by $p_1W_1 + p_2W_2$ where here the variables $p_1, p_2$ give coordinates of the last factor $\mathbb{C}^2$ to fit in with the notation of \cite{FJRnew}. The critical locus of this potential is therefore $$\{ W_1 = 0, W_2 = 0, p_1\frac{\partial W_1}{\partial X}, \ldots, p_1\frac{\partial W_1}{\partial Z}, p_2\frac{\partial W_2}{\partial x}, \ldots, p_2\frac{\partial W_2}{\partial w} \}.$$ Since $W_1, W_2$ are non-degenerate, either $p_1 = W_1 = 0$ or $p_1 \ne 0$, $(X, Y, Z) = 0$, and either $p_2 = W_2 = 0$ or $p_2 \ne 0$, $(x, y, z, w) = 0$.



We split into cases according to the characters $\theta:G \to \mathbb{C}^*$, which acts on the total space of $L_{\theta}$ by $$g = (\lambda_1, \lambda_2, \lambda_3):(\mathbf{v},\tilde{z}) \mapsto (g \cdot \mathbf{v}, \lambda_1^{e_1}\lambda_2^{e_2}\lambda_3^{e_3}\tilde{z})$$. In each case we shall consider $e_3 >0$.
\begin{enumerate}
\item If $e_1<0$ and $e_2<0$, then the semi-stable points require some section $f$ of $ L_{\theta}^{\otimes k}$ to be $G$-invariant and non-zero there. Since the weights are negative for $p_1, p_2, p_3$ but positive for the other coordinates, to ensure invariance we require each monomial in $f$ to have at least some non-zero $X, Y, Z$ (for $\lambda_1 \ne 1)$ and some non-zero $x, y, z, w$ (for $\lambda_2 \ne 1$) to cancel the negative weights from $\theta$. Otherwise, we have complete freedom to choose $k$ and $f$.  Therefore, $$V^{ss}(\theta) = \{(X, Y, Z), (x, y, z, w) \ne 0\}.$$ We choose $\vartheta$ to be the trivial lift, which is clearly good. The intersection with the critical locus is then $$[\frac{(\mathbb{C}^3 \backslash 0) \times (\mathbb{C}^4\backslash 0) \times (\mathbb{C} \backslash 0) \times \{0\} \times \{0\}}{(\mathbb{C}^*)^3}].$$ This is the geometric phase, which cedes Gromov-Witten theory.
\item If $e_1>0$ and $e_2 < 0$, then the semistable locus is $$[V^{ss}//_{\theta} G] = [(\mathbb{C}^3 \times (\mathbb{C}^4\backslash 0) \times \mathbb{C}^* \times \mathbb{C})/(\mathbb{C}^*)^3].$$ The ambient space can then  be viewed as a non-trivial $B\mathbb{Z}_2$-gerbe over $[K / \langle \sigma_K].$ This is the first mixed theory.
\item  If $e_1<0$ and $e_2 > 0$, then the semistable locus is $$[V^{ss}//_{\theta} G] = [((\mathbb{C}^3\backslash 0) \times \mathbb{C}^4 \times \mathbb{C} \times \mathbb{C}^*)/(\mathbb{C}^*)^3].$$ The ambient space can then  be viewed as a non-trivial $B\mathbb{Z}_2$-gerbe over $[E / \langle \sigma_E].$ This is the second mixed theory.
\item If $e_1 >0$ and $e_2 > 0$ then the semistable locus is $$[V^{ss}//_{\theta} G] = [(\mathbb{C}^3 \times \mathbb{C}^4 \times \mathbb{C}^* \times \mathbb{C}^*)/(\mathbb{C}^*)^3].$$ This is the FJRW theory. In our case, we have chosen $p_1 = p_2 = 1$ and the group action is the same as that of the group $\langle J_1, J_2, \sigma\rangle$ considered in the FJRW section. 
\end{enumerate}

We shall calculate the I-functions for the two mixed theories; it will be convenient to express the state space in a  way that relates to the GW and FJRW theories.  We first compute the first mixed theory, which we label by $(FJRW, GW)$, in that it is related to the FJRW theory for the elliptic curve part, and the Gromov-Witten theory for the K3 part. The situation for the other mixed theory is entirely similar.

Consider the induced actions of $\sigma \in G$ on $H_{\mathrm{CR}}(\mathbf{C}^3, W_1^{+\infty})$ and $H_{\mathrm{CR}}(\{x^2 + P(x, y, z)=0\})$. These split into positive and negative eigenspaces, $H_{\mathrm{FJRW}, E}^{\pm}$ and $H_{\mathrm{GW}, K}^{\pm}$ respectively.

If $\mathrm{Re}W$ takes on value greater than $2\rho$, then either $W_1$ (the polynomial defining the elliptic curve) or $W_2$, the polynomial defining the K3 surface) must have real part greater than $\rho$. Taking $\rho$ arbitrarily large, this allows us to decompose $W^{+\infty}$ and apply the K\"unneth theorem for relative cohomology to the GLSM state space $H_{\mathrm{CR}}(\mathcal{X}_{\theta}, W^{+\infty}; \mathbb{C})$. Then the part of the total Chen-Ruan cohomology untwisted by $\sigma$ is given by the singular K\"unneth theorem as $$(\mathcal{H}_{\mathrm{FJRW}, E}^+ \otimes \mathcal{H}_{\mathrm{GW}, K}^{+}) \oplus (\mathcal{H}_{\mathrm{FJRW}, E}^- \otimes \mathcal{H}_{\mathrm{GW}, K}^{-}).$$
To this we must add the part twisted by $\sigma$, which is given as the $G$-invariant relative Chen-Ruan cohomology of $$\mathrm{Fix}(\sigma)=(\mathrm{Fix}(\langle J_1, \sigma_E\rangle)) \times \mathrm{Fix}(\sigma_K) \subseteq \mathbb{C}^3 \times ([\{x^2 + P(x, y, z)=0\}/\langle J_2 \rangle]).$$ By the K\"unneth theorem this is the tensor product of the $\sigma$-twisted parts $H_{\mathrm{FJRW}, E}^{\sigma_E} \otimes H_{\mathrm{GW}, K}^{\sigma_K}$. 

We are interest in the sectors of compact type. These are by definition the spans of the narrow sectors (which are induced by the narrow sectors of $\mathcal{H}_{\mathrm{FJRW}, E}$, and the critical sectors (which are induced by the ambient sectors of $\mathcal{H}_{\mathrm{GW}, K}$). Thus we have $$\mathcal{H}_{\mathrm{FJRW},\,\mathrm{GW}}^{\mathrm{comp}} = (H_{\mathrm{FJRW}, E}^{\mathrm{nar}, +} \otimes H_{\mathrm{GW}, K}^{\mathrm{amb}, +}) \oplus (H_{\mathrm{FJRW}, E}^{\mathrm{nar},-} \otimes H_{\mathrm{GW}, K}^{\mathrm{amb}, -}) \oplus (H_{\mathrm{FJRW}, E}^{\sigma_E, \mathrm{nar}} \otimes H_{\mathrm{GW}, K}^{\sigma_K, \mathrm{amb}}).$$

We have a basis of $\mathcal{H}_{\mathrm{FJRW},\,\mathrm{GW}}^{\mathrm{comp}}$ given by elements we may write of the form:
\begin{itemize}
\item $\phi_{J_1}\mathbf{1}_0$
\item $\phi_{J_1^3}\mathbf{1}_0, \phi_{\sigma_E J_1^2}\mathbf{1}_{\sigma_K},  1+ \sum_{i, j: \mathrm{gcd}(i, j) >1}(p_{i, j} - 1)$ sectors of the form $\phi_{J_1}\mathbf{1}_{g_{i, j}^r}$,\\
 $\frac{1}{2}\sum_{w_j \not\vert w_0} p_{0, j}$ sectors of the form $\sigma J_1^2\mathbf{1}_{g_{\frac{r}{2p}}}$
\item $\phi_{J_1}D_K^2, 1+ \sum_{i, j: \mathrm{gcd}(i, j) >1}(p_{i, j} - 1)$ sectors of the form $\phi_{J_1}^3\mathbf{1}_{g_{i, j}^r}, \phi_{\sigma_E J_1^2} D_K\mathbf{1}_{\sigma_K},$\\
 $\frac{1}{2}\sum_{w_j \not\vert w_0} p_{0, j}$ sectors of the form $\phi_{\sigma_E J_1^2}\mathbf{1}_{g_{\frac{r}{2p}}}$, and a sector $\phi_{J_1^3}D_K\mathbf{1}_{\tilde{g}}$ when $w_1 = \frac{d}{3}$
\item $\phi_{J_1^3}D_K^2$
\end{itemize}
The new GLSM theory unifies the notion of degree, as well as the pairing. We have $\mathrm{deg}(\phi_h \mathbf{\beta}) = \mathrm{deg}(\phi_h) + \mathrm{deg}(\beta)$. This gives an isomorphism between all narrow state spaces. 

In our case, we re-express the Landau-Ginzburg quasimaps for the case of toric stacks as given in \cite{FJRnew}, decomposing the line bundle $\mathcal{P}$ and the sections $\sigma, \kappa$ into summands. Then, for $E \subseteq \mathbb{P}(2, 1, 1)$, the moduli space $LGQ_{g, n}^{\vartheta, \epsilon}(\mathcal{X}_{\theta}, \beta)$ is given by \begin{align*}& \{(\mathcal{C}, z_1, \ldots, z_n, \mathcal{L}_1,   \mathcal{L}_2, \mathcal{T}, s_X, s_Y, s_Z, s_x, s_y, s_z, s_w, s_{p_1}, s_{p_2}): \\ & s_X \in H^0(\mathcal{L}_1^{\otimes 3} \otimes \mathcal{T}), s_Y, s_Z \in H^0(\mathcal{L}_1),  s_x \in H^0(\mathcal{L}_2^{\otimes w_0} \otimes \mathcal{T}), \\ & s_y \in H^0(\mathcal{L}_2^{\otimes w_1}), s_y \in H^0(\mathcal{L}_2^{\otimes w_2}), s_z \in H^0(\mathcal{L}_2^{\otimes w_3}),\\ & s_{p_1} \in H^0(\mathcal{L}_1^{-4} \otimes \omega_{\mathrm{log}}), s_{p_2} \in H^0(\mathcal{L}_2^{-2w_0}\otimes \omega_{\mathrm{log}}), \mathcal{T}^{\otimes 2} \cong \mathcal{O}\},\end{align*} where the stability conditions are satisfied.

For Gromov-Witten theory, the sections $s_i$ induce a map $f:\mathcal{C} \to \mathcal{X}_{\vartheta},$ so this is the moduli space of stable maps. For FJRW theory, this is the moduli space of spin curves subject to the conditions provided by these sections. For the mixed theories, we exploit the fact that the moduli space decomposes. Note that there are two possible 2-torsion bundles $\mathcal{T}$: one trivial and one not. Suppressing the markings, each LG quasimap $(\mathcal{C}, \mathcal{L}_1, \mathcal{L}_2, \mathcal{T}, \sigma_X, \ldots, \sigma_x \ldots, \sigma_{p_1}, \sigma_{p_2}),$ gives a pair of LG quasimaps $$(\mathcal{C}, \mathcal{L}_1, \mathcal{T}, \sigma_X, \sigma_Y, \sigma_Z, \sigma_{p_1}), (\mathcal{C}, \mathcal{L}_2, \mathcal{T}, \sigma_x, \sigma_y, \sigma_z, \sigma_w, \sigma_{p_2}).$$ $\mathcal{T}$ depends on $\beta$, which corresponds to a member of the state space. Thus, provided this condition is satisfied, that moduli space presents as a fibre product over $\mathcal{M}_{g, n}$.

In the Gromov-Witten case, this is to say that giving a stable map $f: \mathcal{C} \to [E \times K/\mathbb{Z}_2]$ is equivalent to giving one stable map $f_1: \mathcal{C} \to [E/\langle \sigma_E \rangle]$ and another $f_2:\mathcal{C} \to [K/\langle \sigma_K \rangle]$, provided that the images of $f_1, f_2$ either both lie in the untwisted component of their inertia stacks, or both in the twisted components (giving a well-defined class $\beta$). In the FJRW case, we see that the we get separate line bundles subject to conditions equivalent to the FJRW stability conditions, i.e. the structure of $W$-spin curves.

For  the first mixed theory, for $\mathbf{h} \in \mathcal{H}^{\mathrm{FJRW}}(W_1, \langle J_1, \sigma_E\rangle), \beta \in \mathcal{H}^{\mathrm{GW}}([K/\langle \sigma_K \rangle])$ such that  $\mathbf{h} \beta \in \mathcal{H}_{\mathrm{FJRW}, \mathrm{GW}}$, the genus-0 n-pointed moduli space for $\mathbf{h} \beta$ may be written
\begin{align*}
\overline{\mathcal{M}}^{\mathrm{FJRW, GW}}_{0, n}(\mathbf{W}, G, \beta \mathbf{h}) = & \mathcal{W}^{\mathrm{FJRW}}_{0, n, \mathbf{h}} (W_1, \langle J_1, \sigma_E\rangle) \times_{\overline{\mathcal{M}}_{0, n}} \mathcal{M}^{\mathrm{GW}}_{0, n}([K/\langle \sigma_K \rangle], \beta).
\end{align*}
This justifies the subscript notation for the mixed theories.

The Hodge bundle of this theory decomposes similarly, and we have $$\psi_{\mathrm{FJRW},\,\mathrm{GW}, i} = \psi_{\mathrm{FJRW}, i}\psi_{\mathrm{GW}, i},$$ so we see that the invariants above in fact decompose as $$\langle \tau_{a_1}(h_1), \ldots, \tau_{a_n}(h_n)\rangle_{0, n}^{(W_1, \langle J_1, \sigma_E\rangle)}\langle \tau_{a_1}(\alpha_n), \ldots, \tau_{a_n}(\alpha_n)\rangle_{0, n}^{([K/\langle \sigma_K\rangle], \beta)}.$$ Composing the evaluation maps with the two projections gives a similar composition there. The only difficulty  in determining the invariants lies in the fact that the virtual class does not necessarily decompose naturally.

We compute the untwisted invariants as follows. Since we are working with toric stacks, we may let $\rho: LGQ_{0, n}^{\vartheta, \epsilon}(\mathcal{X}_{\theta}, \beta) \to M_{0, n}(\mathcal{X}_{\theta}, \beta')$ be the natural map sending a stable $LG$-quasimap to the induced stable map to $\mathcal{X}_{\theta}$, where $\beta'$ is the induced homology class. For genus 0, the  virtual class is given by the cosection construction in both cases, and is preserved. The untwisted invariants are therefore given by 
$$\langle\tau_{a_1}(\alpha_1), \ldots, \tau_{a_n}(\alpha_n) \rangle_{0, n}^{\vartheta} = \int_{\rho^*([M_{0, n}(\mathcal{X}_{\theta}, \beta')]^{\mathrm{vir})}} \prod_{i=1}^n \mathrm{ev}_i(\alpha_i)\psi_i^{a_i}$$ which by substitution are the genus-0 Gromov-Witten invariants of the `ambient' $\mathcal{X}_{\theta}.$ This is a toric stack, and as such can be computed from the main theorem of \cite{CCIT1}, as before, after naturally identifying the Chen-Ruan basis with the mixed basis.

For example, we can see that the untwisted FJRW theory is equivalent to the Gromov-Witten theory of $\mathcal{B}G$, and indeed that the untwisted FJRW J-function from the previous section may be given by $ze^{\sum_{h\in G} t_h\phi_{h}/z}$, identifying the FJRW basis elements $\phi_h$ with the fundamental classes of the $h$-sectors. 

The string equation allows us to write this as \begin{align*}J_{\mathrm{FJRW}, \mathrm{GW}}^{\mathrm{un}}(\mathbf{t}, z) = & e^{(\phi_{J_1}\mathbf{1}_0t_{J_1}  + \sum_{i = 0}^3 w_i\phi_0D_Kt_{4+i} +  \phi_{\sigma_E}\mathbf{1}_{\sigma_K} + \sum_{g: \mathrm{deg}(\phi_{\sigma_E} \mathbf{1}_{\sigma_K g} ) = 2} \phi_{\sigma_E} \mathbf{1}_{\sigma_K g})/z})\\ & \times \frac{\prod_{i=1}^4\Gamma(w_iD_K/z+1)}{\Gamma(w_0D_K+w_0b+c+1)\prod_{i=1}^3\Gamma(w_iD_K/z+w_ib+1)}. \end{align*}

The twisted I-function may be found by the same methods of quantisation Givental given in the previous section and the orbifold Grothendieck-Riemann-Roch theorem, and from the work of Tseng, as detailed in \cite{CCIT2}. We express it slightly differently, exploiting the decomposition of the moduli space and the additivity of the functor $c_{\mathrm{top}}R^1\pi_*$. As in \cite{CIT}, let $\beta \in \mathrm{Box}([K/\sigma_K])$ correspond to the fundamental class of a component of the inertia stack of $[K/\langle\sigma_K\rangle]$, let $\mathcal{K}$ be the line bundle whose first Chern class corresponds to $K$, and for $(p, g) \in I[K/\langle \sigma \rangle]$, let $g$ act on $\mathcal{K}\vert_p$ by $e^{2\pi i f(b)}$. Then we have the following decomposition. $$\Delta = (\prod_{k=4}^7\bigoplus_{h \in \langle J_1, \sigma_E\rangle} \mathrm{exp}({G_{q_k}(i_k(h)z, z)} )\bigoplus_{\beta \in \mathrm{Box}([K/\sigma_K])\\ h\beta \in \mathcal{H}_{\mathrm{FJRW}, \mathrm{GW}}} \mathrm{exp}(G_{f(b)}(E, z)).$$

Then we have
\begin{align*}
& I_{\mathrm{FJRW},\,\mathrm{GW}}(\mathbf{t}, z) = ze^{(w_0t_{4} + w_1t_5+w_2t_6+w_3t_7)D_K/z} e^{-z} \times \\
& \sum_{n_{3}, n_{\sigma} \in \mathbb{N}_0^3} 2 \frac{\Gamma(\frac{1}{4}+\frac{n_{3}}{2} + \frac{n_{\sigma}}{4})^2 t_3^{n_3} t_{\sigma}^{n_{\sigma}}}{\Gamma(\frac{1}{4}+\langle\frac{n_3}{2} + \frac{n_{\sigma}}{4}\rangle+1)^2n_3!n_{\sigma}!}z^{\lfloor \frac{n_{3}}{2} + \frac{n_{\sigma}}{4} \rfloor - (n_{3} + n_{\sigma})}  \sum_{\mathbf{b} \in \mathrm{Box}([K/\langle \sigma_K\rangle])}L_{\mathbf{b}}\times\\
& \sum_{\substack{(b, c, \mathbf{k}) \in \Lambda E_{\mathbf{b}}^S([K/\langle \sigma_K\rangle])\\ 2c+ \sum_{j=l+1}^mk_j = n_{\sigma}}} (q_2^b q_3^c \prod_{j=1}^mx_j^{k_j})e^{b(\sum_{i=1}^nw_it_i)+c(t_4+2t_8)}G_{\mathbf{b}}(b, c, \mathbf{k})\phi_{h_{n_3, n_{\sigma}}}\mathbf{1}_{\mathbf{b}}\\ 
=: & \sum_{n_1, n_3, n_{\sigma}} \sum_{\mathbf{b} \in \mathrm{Box}([K/\langle \sigma_K\rangle])}  \omega_{n_3, n_{\sigma}, \mathbf{b}}^{\mathrm{FJRW},\,\mathrm{GW}} z^{\lfloor \frac{n_{3}}{2} + \frac{n_{\sigma}}{4} \rfloor - (n_{1} + n_{3} + n_{\sigma})} \phi_{n_1, n_3, n_{\sigma}} \mathbf{1_b}
\end{align*}
where a factor involving $c$ appears only once. Here we have extracted a factor of $\sum_{n_1=0}^{\infty} \frac{z^{-n_{1}}}{n_1!} = e^{-z}$ to suppress the term-wise dependence on $n_1$, and of course $2 =  \frac{\Gamma(\frac{1}{2})}{\Gamma(1+\frac{1}{2})}.$

For the second mixed theory,  we compute the I-function for either of the two considered elliptic curves $E$ and $W_2 = x^2 + y^6 + z^6 + w^6.$ Again, we may write a basis of $\mathcal{H}_{\mathrm{GW}, \, \mathrm{FJRW}}^{\mathrm{comp}}$ as 
\begin{itemize}
\item $\mathbf{1}_0\phi_{J_2}$
\item $\mathbf{1}_0\phi_{J_2^3}, D_E\phi_{J_2}, \mathbf{1}_{\sigma_E}\phi_{\sigma_K J_2^2}[, \mathbf{1}_{\sigma_E g} \phi_{\sigma_K J_2^2}]$\\
\item $\mathbf{1}_0\phi_{J_2^5}, D_E\phi_{J_2^3}, \mathbf{1}_{\sigma_E}\phi_{\sigma_K J_2^4}[, \mathbf{1}_{\sigma_E g} \phi_{\sigma_K J_2^4}]$\\
\item $D_E\phi_{J_1^5}$
\end{itemize}
where the terms in square brackets come from the extra sector $\mathbf{1}_{\sigma g}$ that appears for $E = \{X^2+Y^3+Z^6 = 0\}.$

As we had before for $W_1$, let $h_{(n_1, n_3, n_{\sigma})}$ be the unique element of $\langle J_2, \sigma_K\rangle$ such that $$i_k(h_{(n_1, n_3, n_{\sigma})}) = n_1i_k(J_2) + n_3i_k(J_2^3) + n_{\sigma}i_k(\sigma_K J_2^2).$$ Similarly to the first mixed theory, we find, for $E = \{X^2+Y^4+Z^4 = 0\},$
\begin{align*}
& I_{\mathrm{GW},\,\mathrm{FJRW}}(\mathbf{t}, z) = ze^{(2t_{1} + t_2+t_3)D_E/z} e^{-z}\times \\
& \sum_{n_{3}, n_{\sigma} \in \mathbb{N}_0^3} 2 \frac{\Gamma(\frac{1}{6}+\frac{n_{3}}{3} + \frac{n_{\sigma}}{6})^3t_3^{n_3} t_{\sigma}^{n_{\sigma}}}{\Gamma(\frac{1}{6}+\langle\frac{n_3}{3} + \frac{n_{\sigma}}{6}\rangle+1)^3n_3!n_{\sigma}!}z^{\lfloor \frac{n_{3}}{3} + \frac{n_{\sigma}}{6} \rfloor - (n_{1} + n_{3} + n_{\sigma})} \times\\
&  \sum_{\mathbf{b} \in \mathrm{Box}([E/\langle \sigma_E\rangle])}K_{\mathbf{b}}\sum_{\substack{(a, c) \in \Lambda E_{\mathbf{b}}^S([E/\langle \sigma_E\rangle])\\ 2c = n_{\sigma}}} q_2^a q_3^ce^{a(2t_1+t_2+t_3)+c(t_1+2t_8)}F_{\mathbf{b}}(a, c)\mathbf{1}_{\mathbf{b}}\phi_{h_{n_3, n_{\sigma}}}\\ 
=: & \sum_{n_1, n_3, n_{\sigma}} \sum_{\mathbf{b} \in \mathrm{Box}([E/\langle \sigma_E\rangle])}  \omega_{\mathbf{b}, n_3, n_{\sigma}}^{\mathrm{GW},\,\mathrm{FJRW}} z^{\lfloor \frac{n_{3}}{3} + \frac{n_{\sigma}}{6} \rfloor - (n_{3} + n_{\sigma})} \mathbf{1_b} \phi_{n_1, n_3, n_{\sigma}}. 
\end{align*}

For both of these mixed theories, the exponent of $z$ is then only 1 for the term corresponding to the identity ($\phi_{J_1}\mathbf{1}_0$ or $\mathbf{1}\phi_{J_2}$, respectively), and the coefficient of $z^0$ is clearly linear in $\mathbf{t}$. If write
\begin{align*}
& I_{\mathrm{FJRW}, \mathrm{GW}}(z, \mathbf{t}) = f_1(\mathbf{t})z\phi_{J_1\mathbf{1}_0} + \mathbf{g}_1(\mathbf{t}) + \mathcal{O}(z^{-1}),\\
& I_{\mathrm{GW}, \mathrm{FJRW}}(z, \mathbf{t}) = f_2(\mathbf{t})z\phi_{\mathbf{1}_0J_2} + \mathbf{g}_2(\mathbf{t}) + \mathcal{O}(z^{-1}).
\end{align*}

Then for $i=1, 2$, if we set $\mathbf{\tau}_i(\mathbf{t}) = \mathbf{g}_i(t)/f_i(t),$ we find that $$\frac{I_{\mathrm{FJRW}}^{(W, G)} (\mathbf{t}, z)}{f(\mathbf{t}}$$ lies on $\mathcal{L}^{(W, G)}$ and is of the form $z\phi_{J_1J_2} + \mathbf{t} + \mathcal{O}(z^{-1}).$ Since the J-function is unique with respect to this property, this gives us the following FJRW `mirror theorems'.

\begin{prop}
\begin{align*}
& J_{\mathrm{FJRW}, \mathrm{GW}}^{(W, G)}(\mathbf{\tau}_1(\mathbf{t}), z) = \frac{I_{\mathrm{FJRW}, \mathrm{GW}}^{(W_1, \langle J_1, \sigma_E \rangle), [K/\langle \sigma_K\rangle]} (\mathbf{t}, z)}{f_1(\mathbf{t})},\\
& J_{\mathrm{GW}, \mathrm{FJRW}}^{(W, G)}(\mathbf{\tau}_2(\mathbf{t}), z) = \frac{I_{\mathrm{GW}, \mathrm{FJRW}}^{[E/\langle \sigma_E \rangle], (W_2, \langle J_2, \sigma_K\rangle)}(\mathbf{t}, z)}{f_2(\mathbf{t})}.
\end{align*}
\end{prop}

\section{The Correspondence}

\subsection{The State Space Correspondence}

Artebani, Boissi\`ere and Sarti have proved in \cite{ABS} that $\mathcal{H}_{\mathrm{FJRW}} \cong \mathcal{H}_{\mathrm{GW}}$ for all Borcea-Voisin orbifolds except those for which $6 \vert w_0$, by constructing birational models for them. There are three such Fermat cases, where $(w_0, w_1, w_2, w_3) = (6, 3, 2, 1), (6, 4, 1, 1), (12, 8, 3, 1)$. We compute these too, along similar lines to how we found the state space for $(3, 1, 1, 1)$. The argument for $\mathcal{H}_{\mathrm{FJRW},\,\mathrm{GW}}$ and $\mathcal{H}_{\mathrm{GW}, \mathrm{FJRW}}$ follow identically.

Consider the possible restrictions of $W$ to fixpoint sets of elements of $G$ which have non-zero-dimensional $G$-invariant Milnor ring. They must satisfy the following properties: that $X^2$ appears if and only if $x^2$ appears (since $\sigma$ must be preserved), and that if $X^2$ (resp. $x^2$) appears, then other terms in $Y, Z$ (resp. $y, z, w$) must appear (from the invariance under $J_1$, resp. $J_2$). Splitting the coordinates $X, Y, Z$ from $x, y, z, w$ and finding the corresponding fixpoint spaces, we list the possibilities these conditions leave and tabulate their contributions below.
$$\begin{tabular}{|l|l|l|}
\hline
$W'$ & $\#\{g \in G\vert W\vert_{\mathrm{Fix}(g)} = W'\}$ & $\mathrm{dim}(\mathcal{H}_{\mathrm{FJRW}}^g)$\\
\hline
$X^2 + Y^4 + Z^4 + x^2 + y^4+ z^6 + w^{12}$ & 1 & 30\\
$X^2 + Y^4 + Z^4 + x^2 + y^4$ & 2 & 2\\
$X^2 + Y^4 + Z^4 + x^2 +z^6 $ & 1 & 2\\
$Y^4 + Z^4 + y^4+ z^6 + w^{12}$ & 1 & 42\\
$Y^4+Z^4+y^4$ & 2 & 0\\
$Y^4 + Z^4 + z^6$ & 1& 0\\
$Y^4+Z^4$ & 2 & 3\\
$y^4+z^6 + w^{12}$ & 1 & 14\\
$-$ (narrow) & 14 & 1\\
\hline
\end{tabular}$$
This gives a total dimension of 112. Checking the group elements of degree 2 with non-trivial restricted Milnor rings, we find they contribute dimension 9, so from the symmetries of the FJRW bi-degree with respect to inverses and swapping indices, we find the following diamond
$$\begin{tabular}{llllllll}
    & & & 1 & & & \\
   & & 0 & & 0 & & \\
   & 0 &  & 9 &  & 0 & \\
   1 & & 45 & & 45 & & 1 \\
   & 0 &  & 9 &  & 0 & \\
   & & 0 & & 0 & & \\
   & & & 1 & & & \\
 \end{tabular}$$
 which is exactly the Hodge diamond on the Gromov-Witten side.
 
For $X^2+Y^4+Z^4+ x^2+y^3+z^{12}+w^{12}$, we have the following table.
$$\begin{tabular}{|l|l|l|}
\hline
$W'$ & $\#\{g \in G\vert W\vert_{\mathrm{Fix}(g)} = W'\}$ & $\mathrm{dim}(\mathcal{H}_{\mathrm{FJRW}}^g)$\\
\hline
$X^2 + Y^4 + Z^4 + x^2 + y^3+ z^{12} + w^{12}$ & 1 & 40\\
$X^2 + Y^4+Z^4 + x^2 + y^3$ & 1 & 0\\
$y^3$ & 4 & 0\\
$Y^4 + Z^4 + y^3+ z^{12} + w^{12}$ & 1 & 60\\
$Y^4 + Z^4$  & 4 & 3\\
$y^3+ z^{12} + w^{12}$ & 1 & 20\\
- (narrow) & 12 & 1\\
\hline
\end{tabular}$$
There are 11 sectors of degree 2, and we have the same symmetries from the degree formula. The FJRW diamond and the Hodge diamond on the Gromov-Witten side are both found to be as follows.
 $$\begin{tabular}{llllllll}
    & & & 1 & & & \\
   & & 0 & & 0 & & \\
   & 0 &  & 11 &  & 0 & \\
   1 & & 59 & & 59 & & 1 \\
   & 0 &  & 11 &  & 0 & \\
   & & 0 & & 0 & & \\
   & & & 1 & & & \\
 \end{tabular}$$
 Finally, for $X^2+Y^4+Z^4+ x^2+y^3+z^8+w^{24}$, we have the following table.
$$\begin{tabular}{|l|l|l|}
\hline
$W'$ & $\#\{g \in G\vert W\vert_{\mathrm{Fix}(g)} = W'\}$ & $\mathrm{dim}(\mathcal{H}_{\mathrm{FJRW}}^g)$\\
\hline
$X^2 + Y^4 + Z^4 + x^2 + y^3+ z^8 + w^{24}$ & 1 & 28\\
$X^2 + Y^4 + Z^4 + x^2 + y^3$ & 3 & 0\\
$X^2 + Y^4 + Z^4 + x^2 + z^8$ & 2 & 2\\
$Y^4+Z^4 + y^3 + z^8 + w^{24}$ & 1 & 42\\
$Y^4 + Z^4+y^3$ & 3 & 0\\
$Y^4 + Z^4 + z^8$ & 2 & 0\\
$Y^4 + Z^4$ & 6 & 3\\
$y^3 + z^8 + w^{24}$ & 1 & 14\\ 
$y^3$ & 6 & 0\\
$z^8$ & 2 & 0\\
- (narrow) & 22 & 1\\
\hline
\end{tabular}$$
There are 19 sectors of degree 2, and we have the same symmetries from the degree formula. Again, the FJRW diamond and Hodge diamond on the Gromov-Witten side are both found to be as follows.
 $$\begin{tabular}{llllllll}
    & & & 1 & & & \\
   & & 0 & & 0 & & \\
   & 0 &  & 19 &  & 0 & \\
   1 & & 43 & & 43 & & 1 \\
   & 0 &  & 19 &  & 0 & \\
   & & 0 & & 0 & & \\
   & & & 1 & & & \\
 \end{tabular}$$
The procedure and results are identical for $X^2+Y^3+Z^6$ in place of $X^2+Y^4+Z^4$ (though the narrow subspaces are larger, as computed previously).

The narrow, ambient and narrow mixed state spaces all certainly isomorphic as graded vector spaces with pairing, as has been made clear by computing explicit bases in the previous three sections.

\subsection{The Quantum Correspondence}

Here we relate the I-functions of the Gromov-Witten, mixed, and FJRW theories. All I-functions in this section are taken to be the genus zero, narrow/ambient I-functions.

\begin{prop}
\textbf{(A partial Landau-Ginzburg/Calabi-Yau correspondence for Fermat-type Borcea-Voisin orbifolds, with respect to the elliptic curve)} Let $\mathcal{Y}$ be a Borcea-Voisin orbifold $[E \times K/\mathbb{Z}_2]$ given by $E = \{X^2 + Y^4+Z^4 = 0\},$ and $K$ a Fermat-type K3 surface with Nikulin involution, with $1 \in \mathbb{Z}_2$ acting by sending $(X, x) \mapsto (-X, -x)$. Then there exists an analytic continuation ${I'}_{\mathrm{GW}}^{\mathcal{Y}}$ of $I_{\mathrm{GW}}^{\mathcal{Y}}$ and a symplectic transformation $\mathbb{U}_E:H_{\mathrm{FJRW},\,\mathrm{GW}}^{(W_1, [K/\mathbb{Z}_2], G)} \to H_{\mathrm{GW}}$  sending  $I_{\mathrm{FJRW},\,\mathrm{GW}}^{(W, [K/\mathbb{Z}_2], G)}$ to ${I'}^{\mathcal{Y}}_{\mathrm{GW}}$. 
\end{prop}

\begin{proof}
It is clear that the GW and FJRW state spaces are isomorphic as graded inner product spaces, as are $\mathcal{H}_{\mathrm{GW}}^{\mathrm{amb}}$ and $\mathcal{H}_{\mathrm{FJRW}}^{\mathrm{nar}}$. 

By convention, we set the Novikov variables $q_1, q_2, q_3 = 1$ and vary $\tilde{q}_1 = e^{2t_1+t_2 + t_3}$, $\tilde{q_2} = e^{3t_4 + t_5 + t_6 + t_7}$, $\tilde{q_3} = e^{t_1 + t_4 + 2t_8}t_{\sigma}$. There are two fundamental issues here that must be addressed before performing the calculation itself. 

Firstly there is the issue of convergence in these variables. Applying a ratio test to $I_{\mathrm{GW}}(\mathcal{Y})$ for $a$, holding $b, c$ fixed, we get a radius of convergence $\tilde{q}_1<\frac{1}{4^3}$. Similarly for $b$ we find convergence for $\tilde{q}_2 < \frac{1}{6^4}$, and for $c, k_j$ we have convergence for $\vert \tilde{q}_3 \vert, \vert x_i\vert < 1$, $i = 1, \ldots, m$. Each term separates its dependence on $a$ and $b$ into separate factors, so we have convergence when all of these conditions hold.

We shall analytically continue in the variables $\tilde{q_1}$ (corresponding to $a$ and the hyperplane divisor class in the elliptic curve factor), via the Mellin-Barnes method.

The function $\frac{1}{e^{2\pi i w} -1}$ has only simple poles at the integers, at each of which it has residue 1. Varying $a$ then, and suppressing the dependence on $c, \mathbf{k}$ and $z$, we may write our $I$-function as a $\mathcal{H}_{\mathrm{GW}}[[z^{-1}]]$-linear combination of contour integrals of the form
$$ \int_{C_E} \frac{1}{e^{2 \pi i s_E} -1} F(s_E)  \, \mathrm{d}s_E.$$
where the cohomology classes are taken to be complex variables, and the contours are taken to be any curve in the $s_E$-plane, stretching from $i\infty$ to $-i\infty$ with a detour so that all singularities of the Gamma functions appearing in $F$ to the left, and another detour so that all non-negative integers are to the left, and all positive integers are to the right. The picture below illustrates this.

\begin{center}
\includegraphics[width=3in]{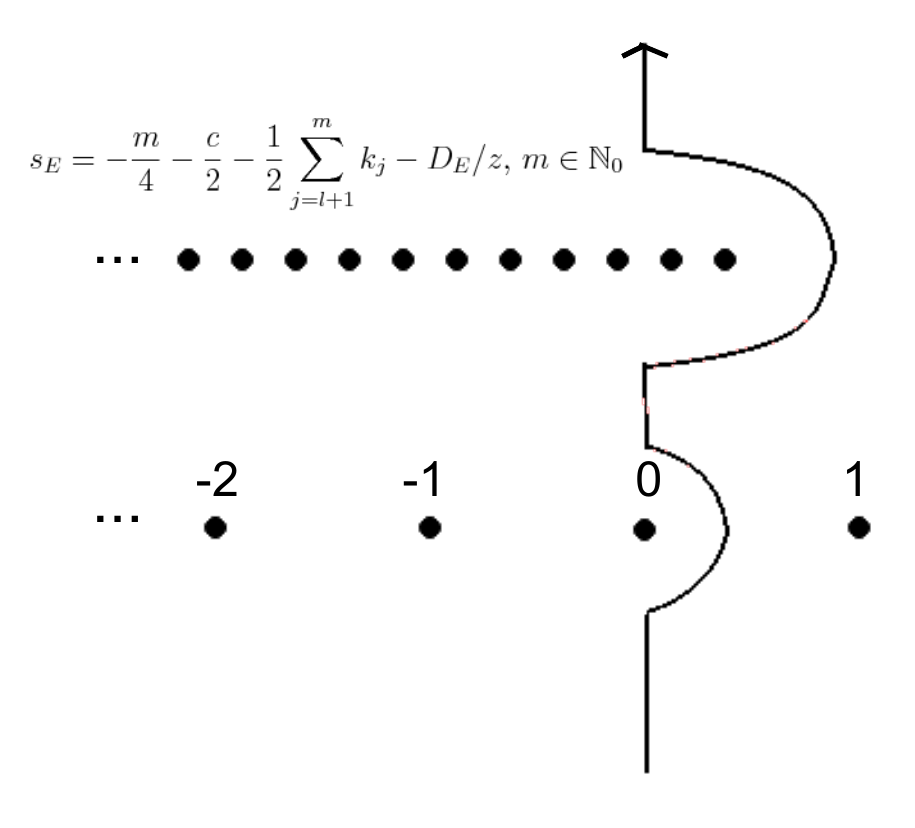}
\end{center}

 
Then closing the curve on the right gives gives the corresponding term of $I_{\mathrm{GW}}$, but closing it on the right gives a different sum. Here, we need to sum over all the negative integers; but from the functional equation of the Gamma function these are all multiples of $D_E^2 = 0$ and so these residues are zero. We are left with the residues at the $s_E$ for which $$4D_E/z + 4s_E + 2c +\sum_{j=l+1}^m k_j  = -m, \, m \in \mathbb{Z}_{>0},$$ that is, $s_E = -\frac{m}{4} - \frac{c}{2} -\frac{1}{4}\sum_{\mu=1}^mk_{\mu}s_{\mu, 1} - D_E/z$. Note that $\sum_{\mu = 1}^m k_{\mu}s_{\mu, 1} = -\frac{1}{2}\sum_{j = l+1}^m k_j$. The residues of the Gamma function are well known, and we have $$\mathrm{Res}_{s_E = -\frac{m}{4} - \frac{c}{2} +\frac{1}{4}\sum_{j=l+1}^mk_j - D_E/z}\Gamma(2D_E/z + 4a+2c +\sum_{j=l+1}^m k_j+1) = -\frac{1}{4}\frac{(-1)^m}{\Gamma(m)}.$$ 

The factor $1/(e^{2\pi i s_E} -1)$ is analytic at these points, but must be multiplied by a factor of $2 \pi i$ from the residue theorem.

We may then rewrite the sum over $a$ as a sum over $m$, with $$ -\frac{1}{4}\frac{(2\pi i)e^{2\pi i(\frac{m}{4} + \frac{c}{2} +\frac{1}{4}\sum_{j=l+1}^m k_j)}}{e^{-2\pi i D_E/z} -e^{2\pi i (\frac{m}{4} + \frac{c}{2} -\frac{1}{4}\sum_{j=l+1}^m k_j)}}\frac{(-1)^m}{\Gamma(m)\Gamma(1-\frac{m}{2})\Gamma(1-\frac{m}{4}-\frac{c}{2} -\frac{1}{4}\sum_{j=l+1}^m k_j)^2}$$ replacing all factors in the sum involving $a$, and the exponentials involving $t_1 ,t_2, t_3$ replaced by $e^{-\frac{m}{4}(2t_1+t_2+t_3)}e^{-\frac{c}{2}(t_2+t_3)}$. 

Using the identity $\Gamma(x)\Gamma(1-x) = \frac{\pi}{\mathrm{sin}(\pi x)}$, we may put this all together as $$-\frac{1}{4}(-1)^m \frac{\Gamma(\frac{m}{2})\mathrm{sin}(\frac{\pi m} {2})\Gamma(\frac{m}{4}+\frac{c}{2} +\frac{1}{4}\sum_{j=l+1}^m k_j)^2\mathrm{sin}(\pi(\frac{m}{4}+ \frac{c}{2}+\frac{1}{4}\sum_{j=l+1}^m k_j))^2}{\pi^3}.$$ Clearly this is zero when $m$ is even. Let $\mathbf{k}' = (k_{l+1}, \ldots, k_m)$. Define
\begin{align*}
&\mu := m\, \mathrm{mod} \,4, m = 4l_m + \mu\\
&\tilde{\sigma} = (m+2c+\sum_{j=l+1}^m k_j)\,\mathrm{mod}\, 4.
\end{align*}
Clearly $E_1(m, c, \mathbf{k})$ depends only on $(\mu, \tilde{\sigma}).$ Then we may split the factors of form $\mathrm{sin}(\pi x)/\pi$ again to give Gamma functions depending only on $\mu, \tilde{\sigma}$:
 \begin{align*}
 &-\frac{1}{4}\sum_{\substack{\mu=1, 3\\ \tilde{\sigma} = 1, 2, 3}}E_1(\mu, \tilde{\sigma})\frac{(-1)^{\mu}}{\Gamma(1-\frac{\mu}{2})\Gamma(1-\frac{\tilde{\sigma}}{4})^2}\\
 & \sum_{\substack{l_m, \in \mathbb{N}_0\\m := 4l_m + \mu, \, c \in \frac{1}{2}\mathbb{N}_0,\, \mathbf{k}' \in \mathbb{N}_0^{m-l}\\m+2c + \sum_{j=l+1}^m k_j \, \mathrm{mod} \, 4 = \tilde{\sigma}}} \frac{(-1)^{l_m}}{\Gamma(\frac{\mu}{2})\Gamma(\frac{\tilde{\sigma}}{4})^2}\frac{\Gamma(\frac{m}{2})\Gamma(\frac{m}{4} + \frac{c}{2}\sum_{j=l+1}^m k_j)^2}{\Gamma(m)}
  \end{align*}
 Note that the terms where $\mu$ is even and $\tilde{\sigma} = 0$ are zero.


Now applying the identity $\Gamma(\frac{m}{2})/\Gamma(m) = 2\sqrt{\pi}/(2^m\Gamma(\frac{m}{2}+\frac{1}{2}))$ we may rewrite this as 
\begin{align*}
& -\frac{2\sqrt{\pi}}{4}  \sum_{\substack{\mu=1, 3\\ \tilde{\sigma} = 1, 2, 3}} K_{\mathbf{b}_{\mu, \tilde{\sigma}}} E_1(\mu, \tilde{\sigma})\frac{(-1)^{\mu}}{\Gamma(1-\frac{\mu}{2})\Gamma(1-\frac{\tilde{\sigma}}{4})^2} e^{-\frac{m}{4}D_E(2t_1+t_2+t_3)}\\
\times & \sum_{\substack{l_m, \in \mathbb{N}_0\\m := 4l_m + \mu, \\ c \in \frac{1}{2}\mathbb{N}_0, b \ge -c/3, \mathbf{k} \in \mathbb{N}_0^{m-l}\\m+2c + \sum_{j=l+1}^m k_j \, \mathrm{mod} \, 4 = \tilde{\sigma}\\v_K^S(b, c, \mathbf{k}) = \mathbf{b}}} \frac{(-1)^{l_m}}{\Gamma(\frac{\mu}{2})\Gamma(\frac{\mu}{4} + \frac{c}{2}+ \frac{1}{4}\sum_{j=l+1}^m k_j)^2}
\frac{\Gamma(\frac{m}{2})\Gamma(\frac{m}{4} + \frac{c}{2}\sum_{j=l+1}^m k_j)^2}{2^m\Gamma(\frac{m}{2}+\frac{1}{2})\Gamma(2c+1)}\\
\times & L_{\mathbf{b}}(q_2^b q_3^c \prod_{j=1}^mx_j^{k_j})e^{b(\sum_{i=1}^nw_it_i)+c(t_4+2t_8)}G(b, c, \mathbf{k}')\mathbf{1}_{\mathbf{b}}.
\end{align*}
where $\mathbf{b}_{\mu, \tilde{\sigma}}$ is $\mathbf{1}_0$ for $\tilde{\sigma} = \frac{1}{4}, \frac{3}{4}$ and $\mathbf{1}_{\sigma}$ for $\tilde{\sigma} = \frac{1}{2}$. 

For any $g \in \langle J_2\rangle$, factorise $\mathbf{1}_{\sigma g}$ as $\mathbf{1}_{\sigma_E}\mathbf{1}_{\sigma_K g}.$ Identify the generators
\begin{align*}
& \phi_{J_1} \mapsto \mathbf{1}_0\\
& \phi_{J_1^3} \mapsto D_E\\
& \phi_{\sigma_E J_1^2} \mapsto \mathbf{1}_{\sigma_E},
\end{align*}
and the numbers
\begin{align*}
& 2n_{1}+1 \mapsto m\\
& n_{\sigma} \mapsto 2c + \sum_{j=l+1}^m k_j.
\end{align*}
Then if we set $\tilde{t}_1 = \frac{1}{2}\tilde{q_1}^{-\frac{1}{4}},$ the interior sum above is exactly $\frac{1}{2}e^{-z}\omega_{n_3, n_{\sigma}, \mathbf{b}}^{\mathrm{FJRW},\,\mathrm{GW}}$. This gives a linear map $$\mathbb{U}_E:\mathcal{V}_{\mathrm{FJRW},\,\mathrm{GW}} \mapsto \mathcal{V}_{\mathrm{GW}}$$ sending $I_{\mathrm{FJRW},\,\mathrm{GW}}$ to $I_{\mathrm{GW}}'$.

To represent this as a matrix, note that for each $\mathbf{b} \in \langle J_2 \rangle$, the map restricts to a $2 \times 2$ matrix between bases $\{ \phi_{J_1} \mathbf{1}_b, \phi_{J_1^3}\mathbf{b}\}$ and $\{\mathbf{1}_b, D_E\mathbf{1}_b\}$.  Each term in $\phi_h$ is given by $\tilde{\sigma} = \Theta_2(h)$, and both possibilities for $\mu$ appear. We must expand $K_{\mathbf{b}_{\mu, \tilde{sigma}}}$ and $E(\mu, \tilde{\sigma})$ in terms of $D_E/z$. We do this with the following Taylor series, which we only need up to linear order for now, since $D_E^2 = 0$:
 \begin{itemize}
\item $\Gamma(1+x) = 1 - \gamma x$.
\item $1/\Gamma(1+x) = 1 + \gamma x $.
\item $\Gamma(\frac{1}{2}+x) = \sqrt{\pi} - \sqrt{\pi}(2 \, \mathrm{ln}\, 2 + \gamma) x$.
\item $\frac{1}{e^{-x}-k} = \frac{1}{1-k} + \frac{1}{(1-k)^2} x$,
\end{itemize}
where $\gamma$ is the Euler-Mascheroni constant.

Let $\xi = e^{2 \pi i \tilde{\sigma}}$. For given $h$, the corresponding $\tilde{\sigma}$ may be given by $\Theta_2(h)$. Then we have for $k = 1, 3$
 \begin{align*}
 \phi_{J_1^{\tilde{\sigma}}} \mathbf{1}_b \mapsto & +\frac{2\sqrt{\pi}}{4}\frac{1}{2}(2\pi i) e^{-z} z^{\frac{\tilde{\sigma}-1}{2}}\bigg( \frac{\xi}{1-\xi}\sum_{\mu = 1, 3}\frac{1}{\Gamma(1-\frac{\mu}{2})\Gamma(1-\frac{\tilde{\sigma}}{4})^2}\\
 + & \big( \frac{2\pi i \xi}{(1-\xi)^2} + \frac{1}{(1 -\xi)}(-2\gamma - \gamma - \gamma + 4\gamma)\big)\sum_{\mu = 1, 3} \frac{1}{\Gamma(1-\frac{\mu}{2})\Gamma(1-\frac{\tilde{\sigma}}{4})^2}D_E/z\bigg).
 \end{align*}  
The constant in front reduces to $\frac{\pi}{4}$ since $$\frac{1}{\Gamma(\frac{1}{2})} + \frac{1}{\Gamma(-\frac{1}{2})} = \frac{1}{2\sqrt{\pi}}.$$ For the part corresponding to $\phi_{J_1}\mathbf{1_b}, \phi_{J_1^3}\mathbf{1_b}$, we get a $(2\times 2)$-matrix.
After scaling, and using the identity $\Gamma(\frac{3}{4}) = \frac{\sqrt{2}\pi}{\Gamma(\frac{1}{4})}$, direct computation shows this to be symplectic and degree preserving for $\mathrm{deg}z = 2$.

For $\phi_{\sigma_E}\mathbf{1}_{\sigma g}$, $g$ possibly the identity, we get a $(1 \times 1)$-matrix, since $D_E \mathbf{1}_{\sigma} = 0$, with entry simplifying drastically to $f(z):= -ie^{-z}$.  $f(z)f^*(-z) = 1,$ so this also symplectic, and is clearly degree-preserving.
Since it is the direct sum of symplectic matrices, the whole matrix $\mathbb{U}_E$ is also symplectic.
\end{proof}

\begin{rem}
The situation is more complex for the other elliptic curve, $E = \{ X^2 + Y^3 + Z^6 = 0\}$, since we may not so directly express this in the basis $\omega_{h, \mathbf{b}}^{\mathrm{FJRW},\,\mathrm{GW}}$ by the procedure above. This is because we cannot simply identify the appearances of $n_{J_1^3}i_k(J_1^3)$ with those of $2n_{J_1^3}q_k$ in $I_{\mathrm{FJRW},\,\mathrm{GW}}(\mathbf{t}, z)$ as before, since if $\Theta_k(h) = 2q_k + q_k$ then $h$ cannot be narrow, since 3 appears as an exponent of $Y$.
\end{rem}

\begin{prop}
\textbf{(A partial Landau-Ginzburg/Calabi-Yau correspondence for a Fermat-type Borcea-Voisin orbifolds, with respect to the K3 surface)} Let $\mathcal{Y}$ be the Borcea-Voisin orbifold $[E \times K/\mathbb{Z}_2]$ given by $E = \{X^2 + Y^4+Z^4 = 0\},$ and $K = \{x^2 + y^6 + z^6 + w^6 = 0\}$, with $1 \in \mathbb{Z}_2$ acting by sending $(X, x) \mapsto (-X, -x)$.  Then there exists an analytic continuation ${I'}_{\mathrm{GW}}^{\mathcal{Y}}$ of $I_{\mathrm{GW}}^{\mathcal{Y}}$ and a symplectic transformation $\mathbb{U}_K:H_{\mathrm{GW}, \, \mathrm{FJRW}} \to H_{\mathrm{GW}}$  sending  $I_{\mathrm{GW}\,\mathrm{FJRW}}$ to ${I'}^{\mathcal{Y}}_{\mathrm{GW}}.$
\end{prop}

\begin{proof}
We shall analytically continue $I'_{\mathrm{GW}}(\mathbf{t}, z)$ from $\tilde{q_2} = 0$ to $\tilde{q_2} = 1$ as above, varying $b$ to $s_K$ and making the substitution $$s_K = \frac{-n}{2w_0} - \frac{c}{w_0} - \frac{\sum_{j=1}^m k_js_{j,4}}{w_0}$$ and obtain factors of the form
\begin{align*}
&-\frac{1}{6}\sum_{\substack{\nu=1, 3, 5\\ \tilde{\sigma} = 1, 2, 3}}E_2(\nu, \tilde{\sigma})\frac{(-1)^{\nu}}{\Gamma(1-\frac{\nu}{2})\Gamma(1-\frac{\tilde{\sigma}}{4})^2}\\
 & \sum_{\substack{l_n, \in \mathbb{N}_0\\n := 6l_m + \nu, \, c \in \frac{1}{2}\mathbb{N}_0\\n+2c \, \mathrm{mod} \, 4 = \tilde{\sigma}}} \frac{(-1)^{l_m}}{\Gamma(\frac{\mu}{2})\Gamma(\frac{\tilde{\sigma}}{4})^3}\frac{\Gamma(\frac{n}{2})\Gamma(\frac{n}{6} + \frac{c}{2})^3}{\Gamma(n)}
\end{align*}

Let $\nu = n\, \mathrm{mod}\, 6, \tilde{\sigma} = c\,\mathrm{mod}\, 4,$ and set $$E_2(\nu, \tilde{\sigma}) = E_2(n, c) = (-4\pi)\frac{e^{2\pi i(\frac{m}{4} + \frac{n}{6} + \frac{5c}{6})}}{(e^{-2\pi i D_E/z}- e^{2\pi i (\frac{m}{4} + \frac{c}{2})})(e^{-2\pi i D_K/z}- e^{2\pi i (\frac{n}{6} + \frac{c}{3})})}.$$
After the same manipulations, we obtain
\begin{align*}
& -\frac{2\sqrt{\pi}}{6}  \sum_{\substack{\nu=1, 3\\ \tilde{\sigma} = 1, 2, 3}} L_{\mathbf{b}_{\mu, \tilde{\sigma}}} E_2(\nu, \tilde{\sigma})\frac{(-1)^{\nu}}{\Gamma(1-\frac{\nu}{2})\Gamma(1-\frac{\tilde{\sigma}}{4})^2} e^{-\frac{m}{4}D_E(2t_1+t_2+t_3)}\\
\times & \sum_{\substack{l_n, \in \mathbb{N}_0\\m := 6l_m + \nu, \\ c \in \frac{1}{2}\mathbb{N}_0, a \ge -c/2\\n+2c\, \mathrm{mod} \, 6 = \tilde{\sigma}\\v_E^S(a, c) = \mathbf{b}}} \frac{(-1)^{l_n}}{\Gamma(\frac{\nu}{2})\Gamma(\frac{\nu}{6} + \frac{c}{2})^3}
\frac{\Gamma(\frac{n}{2})\Gamma(\frac{n}{6} + \frac{c}{2})^2}{2^n\Gamma(\frac{n}{2}+\frac{1}{2})\Gamma(2c+1)}\\
\times & K_{\mathbf{b}}q_2^a q_3^c e^{a(2t_1+t_2+t_3)+c(t_4+2t_8)}F(b, c, \mathbf{k}')\mathbf{1}_{\mathbf{b}}.
\end{align*}
Again, there is a factor of $\mathrm{sin}(\frac{n}{2}\pi)$ so only terms with odd $\nu$ appear.

If we identify the generators
\begin{align*}
&\phi_{J_2} \mapsto \mathbf{1}_0,\\
&\phi_{J_2^3} \mapsto D_K,\\
&\phi_{\sigma_K J_2^2} \mapsto \mathbf{1}_{\sigma_K},
\end{align*}
for the classes of degree 2, and the numbers
\begin{align*}
&2n_2+1 \mapsto n,\\
&n_{\sigma} \mapsto 2c,
\end{align*}
and set $\tilde{t}_2 = \frac{1}{2}\tilde{q}_2^{-\frac{1}{6}},$ the interior sum above is exactly $\frac{1}{2}e^{-z}\omega_{n_3, n_{\sigma}, \mathbf{b}}^{\mathrm{GW}, \, \mathrm{FJRW}}.$ This gives a linear map $$\mathbb{U}_K:\mathcal{V}_{\mathrm{GW},\,\mathrm{FJRW}} \mapsto \mathcal{V}_{\mathrm{GW}}$$ sending $I_{\mathrm{GW}, \mathrm{FJRW}}$ to $I'_{\mathrm{GW}}.$ This time, we expand the analogous factors in $D_K/z$ up to \textit{quadratic} order:
\begin{itemize}
\item $\Gamma(1+x) = 1 - \gamma x + \frac{1}{2}(\gamma^2 + \frac{\pi^2}{6})$.
\item $1/\Gamma(1+x) = 1 + \gamma x + \frac{1}{2}(\gamma^2-\frac{\pi^2}{6})$.
\item $\Gamma(\frac{1}{2}+x) = \sqrt{\pi} - \sqrt{\pi}(2 \, \mathrm{ln}\, 2 + \gamma) x (\frac{1}{2}\pi^{\frac{5}{2}} + \sqrt{\pi}(2\,\mathrm{ln}\, 2 + \gamma)^2)x^2$.
\item $\frac{k}{e^{-x}-k} = \frac{k}{1-k} + \frac{k}{(1-k)^2} x + \frac{k(1+k)}{(1-k)^3}x^2$.
\end{itemize}
This allows us to write $\mathbb{U}_K$ as a direct sum of $(3 \times 3)$ blocks and $(1 \times 1)$ blocks, which may be verified to be symplectic just as before. Therefore, $\mathbb{U}_K$ is symplectic.
\end{proof}

\begin{rem}
The case for $E = \{X^2 + Y^3 + Z^6 = 0\}$ proceeds entirely similarly, with extra terms coming from the sectors $\mathbf{1}_{\sigma g}.$
\end{rem}

\begin{thm}
\textbf{(Two-parameter Landau-Ginzburg/Calabi-Yau correspondence for a Fermat-type Borcea-Voisin orbifold)} Let $\mathcal{Y}$ be the Borcea-Voisin orbifold $[E \times K/\mathbb{Z}_2]$ given by $E = \{X^2 + Y^4+Z^4 = 0\},$ and $K = \{x^2 + y^6 + z^6 + w^6 = 0\}$, with $1 \in \mathbb{Z}_2$ acting by sending $(X, x) \mapsto (-X, -x)$.  Then there exists a two-parameter analytic continuation ${I''}_{\mathrm{GW}}^{\mathcal{Y}}$ of $I_{\mathrm{GW}}^{\mathcal{Y}}$ and a symplectic transformation $\mathbb{U}:H_{\mathrm{FJRW}}^{(W, G)} \to H_{\mathrm{GW}}$  sending  $I_{\mathrm{FJRW}}^{(W, G)}$ to ${I''}^{\mathcal{Y}}_{\mathrm{GW}}$.
\end{thm}

\begin{proof}
We analytically continue $I_{\mathrm{GW}}(\mathbf{t}, z)$ with respect to $\tilde{q}_1$ and $\tilde{q}_2$ as in the previous two theorems. After precisely the same termwise manipulations, we find that $I_{\mathrm{GW}}''(\mathbf{t}, z)$ is given by 
\begin{align*}
& \frac{\pi}{6} (K_0L_0\sum_{\substack{\mu = 1, 3\\\nu = 1, 3, 5\\ \tilde{\sigma} \in \mathbb{Z}_6}} + \, K_{\sigma}L_{\sigma} \sum _{\substack{\mu = 1, 3\\\nu = 1, 3, 5\\ \tilde{\sigma} \in \frac{1}{2}\mathbb{Z}_6\backslash \mathbb{Z}_6}})  \frac{(-1)^{\mu + \nu}E_1(\mu, \tilde{\sigma})E_2(\nu, \tilde{\sigma})}{\Gamma(1-\frac{\mu}{2})\Gamma(1-(\frac{\mu}{4}+\frac{\tilde{\sigma}}{2}))^2\Gamma(1-\frac{\nu}{2})\Gamma(1-(\frac{\nu}{6}+\frac{\tilde{\sigma}}{2}))^3}\\
& \times \sum_{\substack{l_m, l_n, l_c \ge 0\\ m :=4l_m + \mu\\n:=6l_n + \nu\\ c := 4l_c + \tilde{\sigma}}} \frac{1}{2^m2^n}\ \frac{(-1)^{l_m+l_n+l_c}\tilde{q}_1^{-m/4}\tilde{q}_2^{-n/6}\tilde{q}_3^c\Gamma(\frac{m}{4} + \frac{c}{2})^2\Gamma(\frac{n}{6}+ \frac{c}{3})^3}{\Gamma(\frac{\mu}{2})\Gamma(\frac{\mu}{4}+\frac{\tilde{\sigma}}{2})^2\Gamma(\frac{\nu}{2})\Gamma(\frac{\nu}{6} + \frac{\tilde{c}}{2})^3\Gamma(\frac{m}{2}+\frac{1}{2})\Gamma(\frac{n}{2} + \frac{1}{2})\Gamma(1+2c)}.\\
\end{align*} 

If we identify
\begin{align*}
&t_{J_1^3J_2} \mapsto \tilde{q}_1^{-1/4},\\
&t_{J_1J_2^3} \mapsto \tilde{q}_2^{-1/6},\\
&t_{\sigma J_1^2J_2^2} \mapsto \tilde{q}_3,\\
&2\mathbf{n}(J_1^3J_2) + 1 \mapsto m,\\
&2\mathbf{n}(J_1J_2^3)+1 \mapsto n,\\
&n_{\sigma_KJ_1^2} \mapsto 2c,
\end{align*}
for the classes of degree 2, then the interior sum is exactly $\frac{1}{4}\omega_{\mathrm{FJRW}}^h$. This allows us to set up a linear map $\mathbb{U}$ as before, put together from $\mathbb{U}_E$ and $\mathbb{U}_K$. That is, splitting the FJRW basis vectors into their $E$- and $K$-parts as we have for the other two theories, this gives us $\mathbb{U}$ as a direct sum of $(3 \times 3)$ blocks (for $\phi_{J_2}, \phi_{J_2^3}, \phi_{J_2^5}$) and $(2 \times 2)$ blocks (for $\phi_{\sigma_KJ_2^2}$ and $\phi_{\sigma_K J_2^4}$), which can all be checked to be symplectic. Therefore, $\mathbb{U}$ is also symplectic, and this verifies the Landau-Ginzburg/Calabi-Yau correspondence for our case.
\end{proof}













\end{document}